\documentclass[12pt]{article}
\usepackage{ct}
\usepackage[english]{babel}
\usepackage[utf8]{inputenc}
\usepackage{listings}
\usepackage{caption}
\usepackage{graphicx}
\usepackage{mathtools}

\usepackage{refcount}
\usepackage{fp}

\usepackage{algorithm}
\usepackage{algpseudocodex}
\algnewcommand\Input{\textbf{Input: }}
\algnewcommand\Outpout{\textbf{Output: }}

\specs{n}{vol (issue)}{year}{}

\dateline{\today}{ ? }{ ? }



\title{Counting and Sampling Git Graphs}

\author[1]{Courtiel Julien}
\author[2]{Pépin Martin}

\affil[1,2]{
Normandie Univ, UNICAEN, ENSICAEN, CNRS, GREYC, 14000 Caen, France

\email{julien.courtiel@unicaen.fr}, \email{martin.pepin@unicaen.fr}%
}

\newcommand\softO{\widetilde O}

\newcommand\gitgraph{Git graph}
\newcommand\fullnamegitgraph{Git feature branch graph}
\newcommand\cyclarium{cyclarium}
\newcommand\labelled{labeled-main}
\newcommand\main{main commit}
\newcommand\feature{feature commit}

\newcommand\C{\mathcal{C}}

\newcommand\U{\mathcal{U}}
\newcommand\Z{\mathcal{Z}}
\renewcommand\digamma\psi

\newcommand\MetaOp[3]{\ensuremath{\textsc{#1}_{#2}\left(#3\right)}}
\newcommand\Cyc[2][]{\MetaOp{Cyc}{#1}{#2}}
\newcommand\Seq[2][]{\MetaOp{Seq}{#1}{#2}}
\newcommand\Set[2][]{\MetaOp{Set}{#1}{#2}}

\newcommand{\multiset}[2]{
\left.\mathchoice
  {\left(\kern-0.48em\binom{#1}{#2}\kern-0.48em\right)}
  {\big(\kern-0.30em\binom{\smash{#1}}{\smash{#2}}\kern-0.30em\big)}
  {\left(\kern-0.30em\binom{\smash{#1}}{\smash{#2}}\kern-0.30em\right)}
  {\left(\kern-0.30em\binom{\smash{#1}}{\smash{#2}}\kern-0.30em\right)}
\right.}
\newcommand\stirling[3][0]{\genfrac{[}{]}{0pt}{#1}{#2}{#3}}

\graphicspath{{images/}}

\begin{document}
%\writedatatofile

\maketitle

\begin{abstract} 
  Version Control Systems, like Git and Mercurial, represent the history of a project as a Directed Acyclic Graph,
  capturing divergences and merges throughout development.
  A widely adopted workflow, the \emph{feature branch workflow},
  imposes a particularly simple structure:
  a single main branch and features branches evolving in parallel without interfering with one another.

  This work presents an enumerative study of these \emph{Git graphs}. 
  We first establish a bijection between Git graphs and a simpler combinatorial class, which we term \emph{cyclariums},
  and derive their exact count at size $n$, namely
  \(
  \sum_{k=2}^n \sum_{f = 1}^{k-1} {k \brack f} { n - k - 1 \choose k  - f - 1}
  \). 
  Using a saddle-point method, we further show that this number is asymptotically equivalent to
  \[\frac{e^{1/8}}{2}
    \left(\frac{n}{2e}\right)^{\frac n 2}
    \exp\left(\frac 1 2 \ln\frac n 2 \sqrt{\frac n 2} + \sqrt{2 n} + \frac{(\ln\frac n 2)^2}{32} \right)
    \left(\frac n 2\right)^{-3/8}.\]
  Additionally, we analyze the limit laws of key parameters.

  Finally, we address random generation by designing both a uniform sampler and a Boltzmann sampler for these graphs.
\end{abstract}

\section{Introducing Git Graphs}

\subsection{Graph of Commits}

Version Control Systems (VCS, in short) such as Git and Mercurial have revolutionized
the way software development teams manage and track changes in their projects.
Notably they allow multiple contributors to work simultaneously without overwriting each other's work.
Because they maintain a history of changes,
it is possible to revert to previous versions, track the evolution of a project, and understand the context of modifications.

\emph{Commits} (or \emph{revisions}) refer to snapshots of the project's files at a specific point in time.
Unlike traditional linear versioning methods, VCS like Git or Mercurial utilize graph-based structures to represent the history of a project:
the commits are the vertices of a directed acyclic graph (DAG)
where edges indicate the parent-child relationships between commits (cf. Figure ~\ref{fig:graph_random_github}).
Developers can create new \emph{branches} to work on features or fixes,
and subsequently \emph{merge} these branches to integrate their changes into the overall codebase.
This graph-based approach not only allows for efficient tracking of modifications but also facilitates complex branching and merging strategies, enabling developers to work seamlessly.

\begin{figure}[ht!]
\begin{center}
\includegraphics[width=\textwidth]{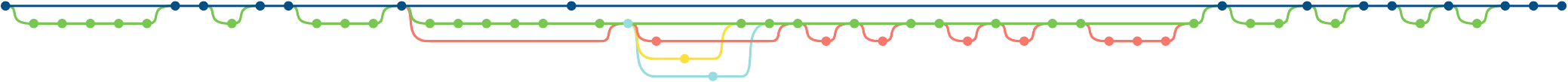}
\end{center}
\vspace{-7pt}
\caption{ Graph of commits extracted from a \textit{real} GitHub repository \cite{github}.  }
\label{fig:graph_random_github}
\end{figure}

Many algorithms designed for VCS exploit the graph structure to be more efficient. For example:
\begin{itemize}
  \item Testing \emph{reachability} from a commit to another, which is essential for VCS,
  can be solved by graph traversals, often used together with a heuristic~\cite{devblog2}.
  Additionally, to speed up traversal,
  each commit maintains an index (if $index(A) < index(B)$ then commit A cannot reach B),
  enabling a pruned search, operating in sublinear time complexity in practice~\cite{mercurial-index,devblog3,marnais1,marnais2}.
  \item At the heart of \emph{merging}
  (the process of integrating changes from different commits)
  is the identification of the lowest common ancestor of the branches involved.
  This ancestor acts as a base version against which both branches' changes are compared~\cite{marnais0,merge}.
  \item The search for a \emph{regression} (which refers to a commit that introduces a bug) is accomplished through a binary bisection of the graph~\cite{bisection,bisection2}
  implemented by \texttt{git bisect} in Git, and by \texttt{hg bisect} in Mercurial.
  This is a heuristic that has been shown to be non-optimal in the worst-case scenario~\cite{CDL2023}.
\end{itemize}

A key motivation for us is to develop a framework to analyze these algorithms,
whether implemented or theoretically conceived,
based on their average performance.
In the literature,
this type of analysis is usually conducted on benchmarks~\cite{benchmark-merging,bisection,marnais3}
and even on synthetic data generated by machine learning~\cite{ml-benchmark}.
However our overarching goal is rather to study the average complexity of these algorithms from a theoretical standpoint.

The challenge lies in determining an appropriate probability distribution for the graphs of commits.
VCS can generate any graph without form limitations\footnote{
  Exception for Mercurial: commits must have at most 2 parents.
  },
but this may not reflect real-world scenarios.

Moreover, once the probability distribution has been chosen,
it would be valuable to design a random generator for our graphs.
Indeed, such generators would facilitate the implementation of VCS algorithms
via \emph{property based testing}.
This software testing technique consists in verifying
the correctness of an algorithm
by checking if some upstream-defined properties are satisfied for a wide range of graph of commits
that are generated randomly during the test.
Additionally, a random generator can help in benchmarking and comparing different algorithms
by generating synthetic data that simulates real-world behavior.

To find a convincing probabilistic model for the graphs of commits,
we suggest examining the workflows commonly adopted by many projects.
In the context of a VCS, a \emph{workflow} refers to
the set of predefined rules and procedures
that guide developers in managing code changes,
including the processes for creating branches and integrating updates into the main codebase.

This paper will focus on a one of the simplest workflow,
yet one that is widely adopted in the corporate world:
the \emph{feature branch workflow}.
In this workflow,
each new feature or bug fix is developed in its own branch,
independently of the others,
and is integrated into the main branch once completed (see for example~\cite{feature_branch_website}).

This article is the continuation of a preliminary study whose short version was published at GASCOM~\cite{short-version}.

\subsection{Definition}

In this subsection,
we provide a formal definition of the graphs
that we will study throughout this paper.
The definition originally comes from Lecoq's PhD manuscript~\cite{LecoqThesis},
and aims to be consistent with the practices established by the feature branch workflow.
For the sake of brevity,
we will refrain in the future from
explicitly mentioning the feature branch workflow
and will refer to the graphs simply as \textit{\gitgraph{}s}.

\begin{figure}[ht!]
\begin{center}
\includegraphics[scale=1]{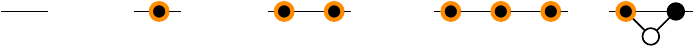}
\vspace{20pt}

\includegraphics[scale=1]{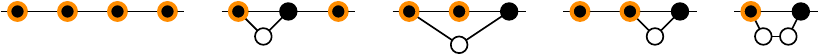}
\vspace{20pt}

\includegraphics[scale=1]{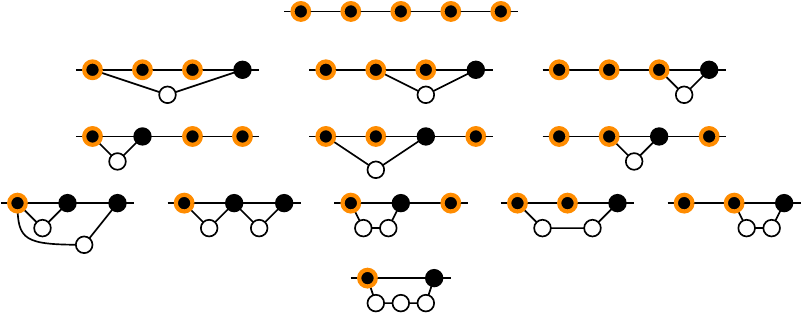}
\end{center}
\vspace{-7pt}
\caption{ All \gitgraph{}s with at most $5$ vertices.
Edges are oriented from left to right.
Free vertices are outlined in orange.
}
\label{fig:smallgraphs}
\end{figure}

\begin{definition}[\gitgraph]

  A \emph{\fullnamegitgraph} (or just \emph{\gitgraph})
  is a Directed Acyclic Graph that consists of:
  \begin{itemize}
    \item a \emph{main branch}, which is a directed path of black vertices called \emph{\main{}s};
    \item zero, one or several \emph{features branches}, which are directed paths:
    \begin{itemize}
      \item starting with a \main{},
      \item ending at a \main{},
      \item  with intermediary vertices --- at least one --- colored in white and called \emph{\feature{}s}.
    \end{itemize}
    We say that a feature branch \emph{forks} from a \main{}
    if this \main{} is the starting point of the branch
    --- this vertex is also called the \emph{base} of the branch ---
    and a feature branch \emph{merges} into a \main{}
    if this \main{} is the ending point of the branch.

    Moreover, two different feature branches cannot:
    \begin{itemize}
      \item share a \feature{} (this condition forces the indegree and the outdegree of any \feature{} to be $1$),
      \item merge into the same \main{} (this condition forces the indegree of any \main{} to be at most $2$).
    \end{itemize}
    (However two feature branches can have the same base. Furthermore, a feature branch can merge into a \main{} which is the base of other branches.)
  \end{itemize}
\end{definition}

The size of a \gitgraph{} is its number of vertices.
By convention, we assume that there is a unique \gitgraph{} of size $0$.
Figure~\ref{fig:smallgraphs} lists all \gitgraph{}s of size at most $5$.

We now define a special type of vertex amongst \main{}s, named \emph{free commits}. In Figure~\ref{fig:smallgraphs}, they are depicted in orange.

\begin{definition}[free commit]
  A \main{} is said to be \emph{free}
  if no feature branch merges into it.
  Equivalently, a free commit is a \main{} with indegree at most $1$.\footnote{The only vertex with indegree $0$ is the first \main{}, and is free.}
  \label{def:free}
\end{definition}

In summary,
the \textit{\fullnamegitgraph} model is characterized by
non-intersecting feature branches,
and the constraint that only one feature branch can be merged into the main branch at a time, never two or more.
Although Git supports multiple branches merges,
the latter restriction contributes to a clearer and more comprehensive project history,
minimizes the risk of conflicts,
and enhances traceability and maintainability.
This constraint is also discussed in~\cite{CDL2023}.

While this model is certainly simple and limited in comparison to real-world graphs,
it is compatible with a commonly used workflow.
Additionally, this model represents an initial step before delving into more intricate workflows.
In particular, it can serve as a basic building block for other models.

\subsection{Outline of the Article}

This paper aims to provide a comprehensive combinatorial and probabilistic understanding of \gitgraph{}s,
with a focus on exact counting, asymptotic behavior, and random generation.

We begin in Section~\ref{first:steps} by establishing a recursive decomposition of \gitgraph{}s.
We derive from this a suitable generating function, with a simple closed formula, to address the enumerative problem.
To advance combinatorial insight,
Section~\ref{s:bijection} introduces a family of combinatorial objects, 
which we term cyclariums, 
and establishes a bijection between these objects and \gitgraph{}s. 
As a corollary, we obtain an elegant formula for the number of \gitgraph{}s.

A key contribution of this article is the asymptotic estimate (Section~\ref{sec:asymptotic}) for the number of \gitgraph{}s of size \(n\), 
which reveals a stretched exponential behavior.
The (long and technical) proof relies on the saddle-point method.

We first determine the asymptotic behavior of \(g_{n,k}\), the number of \gitgraph{}s of size \(n\) with \(k\) \main{}s,
and then sum over all \(k\).
The main challenge lies in ensuring uniformity in our asymptotic estimates to guarantee the validity of the summation.
Efforts have been made to make the proof pedagogically accessible.

Section~\ref{s:limit_law} further analyzes the distribution of structural parameters in random \gitgraph{}s under the uniform distribution,
including the numbers of \main{}s and free commits,
exhibiting Gaussian behavior in both cases,
as well as the width parameter.

Finally, we address random generation by presenting a uniform sampler and a Boltzmann sampler. 
The Boltzmann sampler is particularly efficient, 
generating in linear time a \gitgraph{} of size \(n\) with a desired number of \main{}s 
with relative error proportional to \(n\).

\section{First Steps for Counting}
\label{first:steps}

\subsection{Recursive Decomposition}

First, let us  establish a recurrence relation
that arises naturally from the recursive decomposition of \gitgraph{}s.
The decomposition is illustrated by Figure~\ref{fig:decomposition}.

\begin{figure}[ht!]
\begin{center}
\includegraphics[scale = 1]{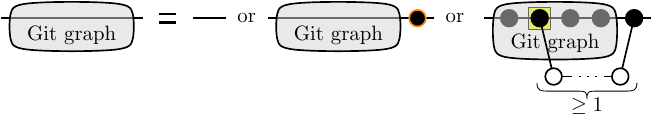}
\end{center}
\vspace{-7pt}
\caption{ Schematic illustration of the recursive decomposition of \gitgraph{}s }
\label{fig:decomposition}
\end{figure}

\begin{proposition}The numbers~$g_{n,k}$ of \gitgraph{}s with~$n$ vertices including~$k$ \main{}s satisfy \(g_{0,k} =  \begin{cases} 1 & \text{if }k = 0 \\ 0 & \textrm{otherwise} \end{cases}\), and for $n \geq 1$
  \begin{equation*}
    g_{n,k} = g_{n-1,k-1} + \sum_{\ell \geq 1} (k-1) g_{n-1-\ell,k-1}.
 \end{equation*}
 \label{prop:recurrence}
\end{proposition}

\begin{proof}
  The case $n=0$ is clear.
  We consider now the set of \gitgraph{}s of size $n$ with $k$ \main{}s,
  where $n \geq 1$ is fixed.

  There are two possibilities:
  either the last \main{} is free,
  or there is one feature branch merging into the last \main{}.
  (By definition of \gitgraph{}s,
  there cannot be two or more branches merging into it.)
  \begin{itemize}

    \item \textbf{\gitgraph{}s where the last \main{} is free} are
    in bijection with \gitgraph{}s of size $n-1$ with $k-1$ \main{}s
    (the operation of removing the last \main{} is reversible).
    Thus their number is $g_{n-1,k-1}$.

    \item In \textbf{\gitgraph{}s where a feature branch merges into the last \main{} and contains $\ell$ \feature{}s}, removing both the feature branch and the last \main{} results in \gitgraph{}s of size $n-1-\ell$ with $k-1$ \main{}s.
     Reconstructing the original graph requires to know which commit served as the base of the removed feature branch.
    Since choosing any of the~$k-1$ \main{}s would yield a valid \gitgraph{}, the number of such \gitgraph{}s is
    \( (k-1)g_{n-1-\ell,k-1} \).
  \end{itemize}
  Ultimately, we arrive at the formula for $g_{n,k}$ as stated in the proposition.
\end{proof}

\subsection{Finding the Good Generating Function}
\label{ss:generating_functions}

While deriving the recurrence from Proposition~\ref{prop:recurrence} was not particularly difficult,
it is less straightforward to figure out how to effectively utilize it.
In this subsection, we will discuss how to choose the most suitable generating function to address our enumerative problem.

Initially,  one might consider the ordinary generating function,
as it directly arises from the recurrence.
Unfortunately, the function is not analytic,
making it impractical for further analysis.

\begin{proposition} The ordinary generating function $G(z,u) := \sum_{n \geq 0, k\geq 0} g_{n,k} z^n u^k$ of \gitgraph{}s is solution of the differential equation
  \begin{equation*}
    G(z,u) = 1 + zu G(z,u) + \frac{z^2u^2}{1-z}\frac{\partial G}{\partial u}(z,u)
  \end{equation*}
but it is \emph{not} analytic at $(0,0)$.
\end{proposition}

\begin{proof}
  The differential equation can be directly derived from the recursive decomposition outlined in the proof of Proposition~\ref{prop:recurrence},
  using the symbolic method described in~\cite{FS2009}.
  From those who are not familiar with this method,
  we can also use the recurrence from Proposition~\ref{prop:recurrence},
  multiply by $z^n\,u^k$,
  and sum over all $n \geq 1$ and $k \geq 1$ to obtain
  \[\sum_{ \substack{ n \geq 1 \\ k \geq 1} }g_{n,k}z^n\,u^k
  = zu\,\sum_{ \substack{ n \geq 1 \\ k \geq 1} }g_{n-1,k-1}z^{n-1}\,u^{k-1}
  + \sum_{ \substack{ n \geq 1 \\ k \geq 1 \\ \ell \geq 1} } (k-1)   g_{n-1-\ell,k-1}z^{n}\,u^{k}.
   \]
   The term on the left-hand side is $G(z,u)-1$,
   while the first term on the right-hand side is $zuG(z,u)$.
   As for the last term,
   use the variable change $n' \leftarrow n-1-\ell$ and $k' \leftarrow k+1$ to rewrite it as
   \begin{align*}
   \sum_{ \substack{ n \geq 1 \\ k \geq 1 \\ \ell \geq 1} } (k-1)   g_{n-1-\ell,k-1}z^{n}\,u^{k} & =  \sum_{ \substack{ n' \geq 0 \\ k' \geq 0 \\ \ell \geq 1} } k'  g_{n',k'}z^{n'+\ell+1}\,u^{k'+1} \\
    & =
    u^2  \times \left(\sum_{\ell \geq 1} z^{\ell+1}\right)
    \times
    \left(\sum_{ \substack{
    n' \geq 0 \\ k' \geq 0
   }} k' g_{n',k'} z^{n'}u^{k'-1} \right)
   \\ & =
   u^2 \times \frac{z^2}{1-z} \times \frac{\partial G}{\partial u}(z,u).
   \end{align*}
   This is how we arrive at the differential equation.

   As for the non-analyticity of $G(z,u)$,
   observe that we can injectively map permutations of size $n$
   to \gitgraph{}s of size $3n$ with $2n$ \main{}s.
   For example, a permutation $\sigma$ can be encoded as a \gitgraph{}
   where for each $i \in \{1,\dots,n\}$ a feature branch of length $1$
   forks from the $i$-th leftmost \main{} and merges into the $\sigma(i)$-th rightmost \main{}.
   Consequently $g_{3n,2n} \geq n!$,
   which implies that $\sum_{n \geq 0} g_{3n,2n} z^{3n}u^{2n}$ is not analytic at $(0,0)$,
   thereby indicating that $G(z,u)$ is also not analytic at $(0,0)$.
\end{proof}

The use of the ordinary generating function being hindered by its non-analytic nature,
we shift the focus to the exponential generating function,
which is analytic at $(0,0)$.
As it will be shown by the next proposition,
the exponential generating function satisfies a partial differential equation.
Although we have not been able to derive meaningful results from this equation,
we believe it is worth sharing,
as it represents a direction we pursued that turned out to be misleading,
yet may offer valuable insights.

\newcommand{\Gexp}{G_{\text{exp}}}

\begin{proposition} The \emph{exponential} generating function $\Gexp(z,u) = \sum_{n \geq 0, k\geq 0} g_{n,k} \frac{z^n}{n!} u^k$ of \gitgraph{}s is solution of the partial differential equation
  \begin{equation*}
    u\,\Gexp(z,u) + \dfrac{\partial^2 \Gexp}{\partial z^2}(z,u) = (1 + u) \dfrac{\partial \Gexp}{\partial z}(z,u)  + u^2 \frac{\partial \Gexp}{\partial u}(z,u),
  \end{equation*}
but it does \emph{not} seem to have a closed formula.
\end{proposition}

\begin{proof}
  This equation is derived from a recursive decomposition that closely resembles the one in the proof of Proposition~\ref{prop:recurrence},
  but is not exactly the same. It is illustrated by Figure~\ref{fig:decomposition2}.
  This approach was proposed by Wenjie Fang in a personal communication.

  \begin{figure}[ht!]
  \begin{center}
  \includegraphics[scale = 1]{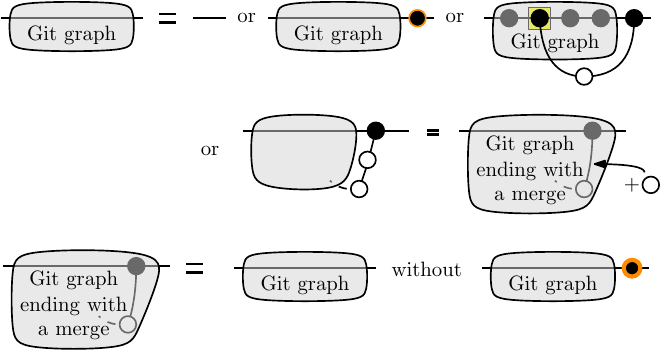}
  \end{center}
  \vspace{-7pt}
  \caption{ Alternative decomposition of \gitgraph{}s.}
  \label{fig:decomposition2}
  \end{figure}

  Let $n$ and $k$ be two positive integers.
  We find a formula for $g_{n,k}$, the number of \gitgraph{}s of size $n$ and with $k$ \feature{}s,
  by partitioning these graphs in three subfamilies:
  \begin{itemize}
    \item \textbf{\gitgraph{}s where the last \main{} is free. }
    These are counted by $g_{n-1,k-1}$,
    as established in the proof of Proposition~\ref{prop:recurrence}.
    \item \textbf{\gitgraph{}s where a feature branch containing exactly one \feature{} merges into the last \main{}. }
    This case, also shown in the proof of Proposition~\ref{prop:recurrence}, is counted by $(k-1) g_{n-2,k-1}$,
    \item \textbf{\gitgraph{}s where a feature branch containing at least two \feature{}s merges into the last \main{}.}
    These are in bijection with \gitgraph{}s of size $n-1$ and $k$ \feature{}s where the last \main{} is \emph{not} free
    (the operation of removing a \feature{} from the last feature branch is reversible).
    Therefore their number is obtained by subtracting $g_{n-2,k-1}$
    (the number of \gitgraph{} where the last \main{} is free)
    from $g_{n-1,k}$ (the number of unrestricted \gitgraph{}s).
  \end{itemize}

  We have just proved the recurrence
  \begin{align}
      g_{n,k} & = g_{n-1,k-1} + (k-1) g_{n-2,k-1} + g_{n-1,k} - g_{n-2,k-1} \label{rec:wenjie}\\
      & = g_{n-1,k-1} + (k-2) g_{n-2,k-1} + g_{n-1,k}. \nonumber
  \end{align}
      By multiplying \eqref{rec:wenjie} by $z^{n-2}u^{k}/(n-2)!$ and summing over all $n$'s and $k$'s, we obtain
      \begin{multline*}
          \sum_{\substack{n \geq 2 \\ k \geq 1}} g_{n,k} \frac{z^{n-2}u^{k}}{(n-2)!} =
          u \sum_{\substack{n \geq 2 \\ k \geq 1}} g_{n-1,k-1} \frac{z^{n-2}u^{k-1}}{(n-2)!}
          + u \sum_{\substack{n \geq 2 \\ k \geq 1}} (k-1) g_{n-2,k-1} \frac{z^{n-2}u^{k-1}}{(n-2)!} \\
          + \sum_{\substack{n \geq 2 \\ k \geq 1}} g_{n-1,k-1} \frac{z^{n-2}u^{k}}{(n-2)!}
          - u \sum_{\substack{n \geq 2 \\ k \geq 1}} g_{n-2,k-1} \frac{z^{n-2}u^{k-1}}{(n-2)!}.
      \end{multline*}

  We recognize from left to right
  $\dfrac{\partial^2 \Gexp}{\partial z^2}$,
  $u\dfrac{\partial \Gexp}{\partial z}$,
  $u^2 \dfrac{\partial \Gexp}{\partial u}(z,u)$,
  $\dfrac{\partial \Gexp}{\partial z}$
  and $-u\,\Gexp$.
\end{proof}

It turns out that the generating function yielding the most fruitful results is not exponential in the primary counting variable,
but rather in the secondary variable that counts \main{}s.
The following theorem presents a closed form for this generating function.

\begin{theorem}
  The generating function $\widetilde G (z,u)$ of \gitgraph{}s ordinary in $z$ (the variable counting vertices) and exponential in $u$ (the variable counting \main{}s),
  that is to say
  \[
  \widetilde G(z,u) :=
  \sum_{\substack{n \geq 0 \\ k \geq 0}} g_{n,k} \, z^n  \, \frac{u^k}{k!},
   \]
    is equal to
\[ \widetilde G (z,u) = {\left(1 - \frac{z^2 u}{1 - z}\right)}^{- \frac{1 - z}{z}}. \]
\label{theo:Gtilde}
\end{theorem}

\begin{proof}
  Similarly to what we did in the two previous proofs,
  we transform the recurrence from Proposition~\ref{prop:recurrence} into a differential equation satisfied by $\widetilde G(z,u)$.
  Indeed,
  by multiplying each side of the recurrence by $ \frac{z^{n} u^{k-1}}{(k-1)!}$
  and summing over all positive integers $n$ and $k$,
  we obtain
  \begin{equation*}
    \sum_{\substack{n \geq 1 \\ k \geq 1}} g_{n,k}  \frac{z^{n} u^{k-1}}{(k-1)!} =
    z \sum_{\substack{n \geq 1 \\ k \geq 1}} g_{n-1,k-1} \frac{z^{n-1} u^{k-1}}{(k-1)!}
    +  \sum_{\substack{n \geq 1 \\ k \geq 1 \\ \ell \geq 1}} z^{\ell+1} (k-1) g_{n-1-\ell,k-1} \frac{z^{n-1-\ell} u^{k-1}}{(k-1)!}.
 \end{equation*}
 Using variable changes,
 we identify from left to right the terms
  $\frac {\partial \widetilde G}{\partial u}$,
  $z\,\widetilde G$ and
  $\frac{z^2 u}{1 - z} \frac{\partial \widetilde G}{\partial u}$, thus forming the differential equation
   \begin{equation}
  \frac {\partial \widetilde G}{\partial u}(z,u) = z\,\widetilde G (z,u) + \frac{z^2 u}{1 - z} \frac{\partial \widetilde G}{\partial u}(z,u)
  \qquad
  \text{with} \qquad \widetilde G(z, 0) = 1.
  \label{eq:ode}
\end{equation}
This differential equation is ordinary with respect to $u$
and is of first order,
making it straightforward to solve,
notably by rewriting it as
\[
\frac{
  \dfrac {\partial \widetilde G}{\partial u}(z,u)
}
{
  \widetilde G(z,u)
}
= \frac z
{1 - \frac{z^2 u}{1-z} }
,
\]
which can be integrated between $0$ and $u$ into
\[ \ln\left( \widetilde G(z,u) \right) - \ln \left( \widetilde G(z,0) \right)
=
- \frac{ 1-z }{z}
\ln \left(
  1 - \frac{z^2}{1-z}
\right)
.
\]
Since $ \widetilde G(z,0) = 1$, we recover the equation stated in the proposition.
\end{proof}

% \todo[inline]{Peut-être donner une interprétation combinatoire de \eqref{eq:ode} en termes du box product?}

\section{Bijection with Cyclariums}
\label{s:bijection}

\subsection{Definition}

The closed formula for $\widetilde G$ featured in Theorem~\ref{theo:Gtilde} calls for a combinatorial interpretation.
To address this, we introduce a new family of combinatorial objects termed \cyclarium{}s.

\begin{definition}[cyclarium]
  A \emph{\cyclarium} is
  a set of cycles of $k$ black vertices labeled by $\{1,\dots,k\}$
  where each black vertex that has not the smallest label inside its own cycle
  carries a non-empty chain of white unlabeled vertices.
  \label{def:cyclarium}
\end{definition}

\begin{figure}[ht!]
\begin{center}
\includegraphics[scale=1.2]{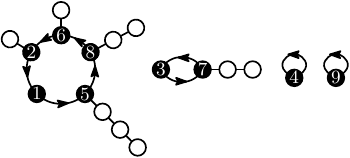}
\end{center}
\vspace{-7pt}
\caption{ Example of a cyclarium of size $18$ with $9$ black vertices. }
\label{fig:cyclarium}
\end{figure}

Let us explain this definition more informally.
If we remove the white vertices,
a cycla\-rium is just a permutation of $\{1,\dots,k\}$,
viewed as a set of cycles of black vertices.
This permutation may include cycles of length $1$.
The smallest vertices in each cycle
(vertices $1$, $3$, $4$, $9$ in Figure~\ref{fig:cyclarium})
are special
insofar as they are \emph{not} incident to a white vertex.
All other vertices are connected to chains of white vertices,
which may vary in length from $1$ to more.

\begin{figure}[ht!]
\begin{center}
\includegraphics[width = 0.9 \textwidth]{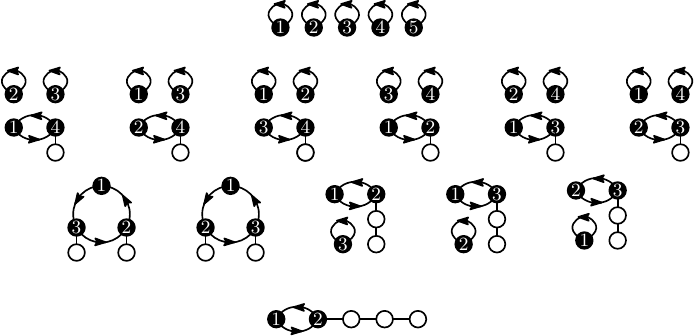}
\end{center}
\vspace{-7pt}
\caption{ All cyclariums of size $5$. }
\label{fig:smallcyclariums}
\end{figure}

The size of a cyclarium is defined by its number of vertices
(black or white).
Figure~\ref{fig:smallcyclariums} presents a complete list of all cyclariums of size 5.

The introduction of cyclariums is motivated by their natural combinatorial specification,
which is made to align with the formula of $\widetilde G$ (see Theorem~\ref{theo:Gtilde}).

In symbolic combinatorics (see \cite{FS2009}),
generating functions encode combinatorial structures,
and functions applied to generating functions, such as logarithm, exponential function or the transformation $X \mapsto \frac{1}{1-X}$,
translate to specific operators on the underlying combinatorial families.

\label{specificition_cyclariums}

Specifically,
let  $\mathcal Y$ represent the family of cyclariums and $\mathcal C$ the family of cyclariums with only one cycle.
We will retrieve the generating function  of $\mathcal Y$ by formulating a \emph{combinatorial specification} of $\mathcal Y$,
which is a description of $\mathcal Y$ in terms of fundamental operators.
This is illustrated by Figure~\ref{fig:specification_cyclarium}.

\begin{figure}[ht!]
\begin{center}
\includegraphics[scale = 1]{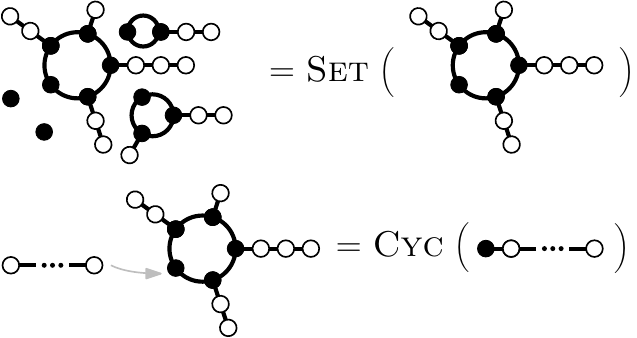}
\end{center}
\vspace{-7pt}
\caption{ Combinatorial specification of cyclariums }
\label{fig:specification_cyclarium}
\end{figure}

According to symbolic method,
we have $\mathcal Y =  \Set{\C}$, where $\Set{\cdot}$ is the set operator\footnote{
  If $\mathcal A$ denotes a combinatorial class that includes labeled elements,
  then $\Set{\mathcal A}$ is defined as
  the new class formed by all possible collections $\{a_1,\dots,a_\ell\}$ drawn from $\mathcal A$.
   Additionally in this construction, the labeled elements are relabeled while preserving their relative order within each $a_i$.
},
which corresponds to $\exp$ in the context of generating functions.
In other words,
if $\widetilde Y(z,u)$ and $\widetilde C(z,u)$ are the respective generating functions for $\mathcal Y$ and $\mathcal C$
(ordinary in $z$ --- the primary variable,
exponential in $u$ --- the variable counting black vertices),
we obtain the relationship
\[\widetilde Y(z,u) = \exp\left(\widetilde C(z,u)\right).\]

Furthermore,
since we can transform one-cycle cyclariums into real cyclical structures by  attaching a non-empty sequence of white vertices to vertex $\boldsymbol{1}$, we have the combinatorial specification
\[
      \Seq[\neq 0]{\Z} \times \C = \Cyc{\U\Z \Seq[\neq 0]{\Z}},
\]
where $\U\Z$ is the atom for black vertices,
 $\Seq[\neq 0]{\Z}$ a non-empty sequence of white vertices,
 and $\Cyc \cdot$ the cycle operator.
When considering generating functions, we find
\[ \frac z {1-z} \, \widetilde{C} (z,u) = -\ln\left( 1 - \frac{u\,z^2}{1-z} \right). \]
From the previous equations, we deduce that the generating function for cyclariums is given by
\[
\widetilde Y(z,u) = \exp\left(
   -\frac{1-z}{z} \ln\left( 1 - \frac{u\,z^2}{1-z} \right)
   \right)
\]
which matches the expression of $\widetilde G(z,u)$ from Theorem~\ref{theo:Gtilde}.

\subsection{Bijection}

Since their generating functions are equal,
\gitgraph{}s and cyclariums are counted by the same numbers.
In fact,
we can describe a bijection between these two families.

\begin{proposition}
  \label{prop:bijection}
  The \gitgraph{}s with $n$ vertices, $k$ \main{}s and $f$ free vertices
  are in bijection with the \cyclarium{}s with $n$ vertices, $k$ black vertices
  and $f$ cycles.
\end{proposition}

\begin{figure}[ht!]
\begin{center}
\includegraphics[scale=1.1]{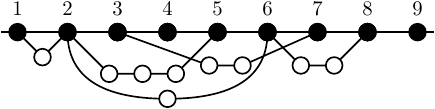}
\end{center}
\vspace{-7pt}
\caption{
  \gitgraph{} in bijection with the cyclarium of Figure~\ref{fig:cyclarium}.
 Numerical labels placed above the \main{}s indicate their position from left to right.
  These labels are included for clarity and are not part of the \gitgraph{} itself.
}
\label{fig:git_graph_bijection}
\end{figure}

\begin{proof}

  Given a cyclarium $C$ with $k$ black vertices, we construct a \gitgraph{} $\phi(C)$ with $k$ \main{}s as follows.
  The cyclarium of Figure~\ref{fig:cyclarium} and the \gitgraph{} of Figure~\ref{fig:git_graph_bijection} are presented to support the understanding of the bijection.
    For each $x \in \{1,\dots,k\}$:
  \begin{itemize}
    \item If $x$ is the smallest label within a cycle of $C$, then let the $x$th leftmost \main{} of $\phi(C)$ be free.

    For example, $1$, $3$, $4$ and $9$ are the minimal labels within their cycles in Figure~\ref{fig:cyclarium},
    and the $1$st, $3$rd, $4$th and $9$th \main{}s  are the only free commits in Figure~\ref{fig:git_graph_bijection}.
    \item  If $x$ is connected to a sequence of $\ell$ white vertices, then we merge a branch containing $\ell$ \feature{}s into the $x$th leftmost \main{} of $\phi(C)$.
    To determine the commit at which this branch forks,
    one traverses $C$ starting from the vertex labeled $x$
    until reaching a vertex with a label $y$ that is smaller than $x$.
    The label $y$ will be called the \emph{nearest smaller label} (refer to Figure~\ref{fig:bijection} for an illustration).
    The base vertex of the feature branch is then designated as the $y$th leftmost vertex of $\phi(C)$.

      \begin{figure}[ht!]
\begin{center}
\includegraphics[width=0.9\textwidth]{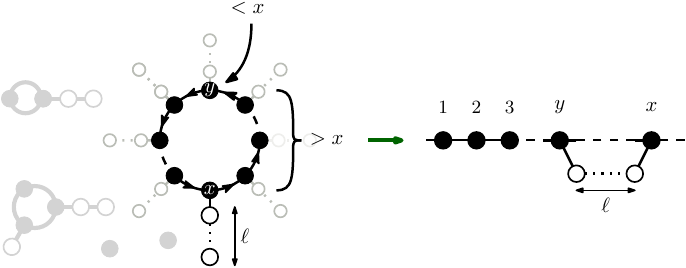}
\end{center}
\vspace{-7pt}
\caption{
  Finding the nearest smaller label, and utilizing it in the bijection from cyclariums to \gitgraph{}s
   }
\label{fig:bijection}
\end{figure}

    For example, consider label $x=5$ in Figure~\ref{fig:cyclarium}.
    It is incident to a chain of $\ell = 3$ white vertices.
     Moreover, since the numbers following $5$ in the cyclarium are $8$, then $6$ and $2$,
     and since $8$ and $6$ are greater than $5$,
    the nearest smaller label $y$ is $2$.
    In Figure~\ref{fig:git_graph_bijection},
    a feature branch merges into the $5$th \main{}: it has $\ell = 3$ \feature{}s,
    and forks from the \main{} at position $y = 2$.
  \end{itemize}

We will now describe the reverse transformation $\phi^{-1}$ by successively
inserting black vertices potentially with incident chains of white vertices.

Let $G$ be a \gitgraph{} with $k$ \main{}s.
We start with a cyclarium $Y$ of size $1$. Then, for each $x \in \{2,\dots,k\}$:
\begin{itemize}
  \item If the $x$th \main{} of $G$ is free, then add to $Y$ a new cycle of length $1$ consisting of a single vertex labeled $x$.
  \item If there is a feature branch merging into the $x$th \main{} of $G$,
  then consider the position $y$ of the base (from left to right), and the number $\ell$ of \feature{}s in it.
   Insert in $Y$ a black vertex labeled $x$ just before the vertex labeled $y$ and attach to it a chain of $\ell$ white vertices.
\end{itemize}
The final cyclarium $Y$ is defined as $\phi^{-1}(G)$.

We can see that $\phi^{-1}$ is the inverse function of $\phi$:
\begin{itemize}
  \item For any free commit at position $x$, label $x$ starts as the sole label in a cycle of length $1$.
   With the addition of larger labels,
   $x$ remains the smallest label within its cycle.
   Therefore, commit at position $x$ will be also free in $\phi(\phi^{-1}(G))$.
  \item For any non-free commit at position $x$,
  there is a feature branch merging into it.
  Let $y$ be the position of the base of the branch.
  Label $x$ is initially set as the predecessor of $y$ in $Y$, which means that $y$ is the nearest smaller label.
  But the nearest smaller labels remain invariant when $Y$ grows,
  since inserting larger labels has no effect on how the nearest smaller label is determined.
  Therefore $y$ is also the nearest smaller label in $\phi^{-1}(G)$.
  Thus in $\phi(\phi^{-1}(G))$, there is also a feature branch of length $\ell$ merging into the $x$th \main{} and forking from the $y$th \main{}.
\end{itemize}
We conclude that $\phi(\phi^{-1}(G)) = G$,
and since \gitgraph{}s and cyclariums share the same gene\-rating function,
it follows that $\phi$ is a bijection.

Furthermore, $\phi$ puts in correspondence minimal labels within their cycles and free commits.
 So, if there are $f$ cycles in the cyclarium $C$, there will be $f$ free commits in $\phi(D)$.
\end{proof}

\begin{remark}
  This bijection differs from the one outlined in the short version of this paper~\cite{short-version}.
  It is notably simpler.
\end{remark}

\begin{remark}
  From an algorithmic point of view, transformation $\phi$ from cyclariums  to \gitgraph{}s can be computed in $O(n)$, where $n$ denotes the size of the cyclarium.
  The main difficulty is to find all nearest smaller labels,
  but this is a classic algorithmic problem which can be solved in linear time~\cite{all_nearest_smaller_values}.
  \label{remark:complexity_bijection}
  Using a suitable encoding for the feature branches, the time complexity of~$\phi$ is actually linear in the size of the main branch.
\end{remark}

As a direct result of this bijection, we are able to obtain a closed formula
for the number of \gitgraph{}s,
as cyclariums can be viewed as permutations with extra structure,
and thereby can be counted via the Stirling numbers of the first kind.

\begin{corollary}%
  \label{cor:formula}
  The number of \gitgraph{}s $g_{n,k}$ of size $n$ and with $k$ \main{}s is $1$ if $k = n$ and
  \[g_{n,k} = \sum_{f = 1}^{k-1} {k \brack f} { n - k - 1 \choose k  - f - 1}\]
  for $ k < n$,
  where \({\cdot \brack \cdot}\) denotes the (unsigned) Stirling number of the first kind.
\end{corollary}

\begin{proof}
  To establish this formula, we focus on counting cyclariums rather than \gitgraph{}s,
  since these objects are in bijective correspondence (cf. Proposition~\ref{prop:bijection}).
  We only consider $k < n$ since the case $k=n$ is obvious.

  Let $g_{n,k,f}$ be the number of cyclariums of size $n$ with $k$ black vertices and $f$ cycles.
  The number of ways to form $f$ cycles from $k$ labeled black vertices,
  which is equivalent to determining a permutation of size $k$ with  $f$ cycles,
  is \({k \brack f}\),
   by definition of the Stirling numbers of the first kind.
  To form a cyclarium, we need to attach a chain of white vertices to every black vertex
  that do not have a minimal label within their respective cycles.
  Therefore we must distribute $n-k$ white vertices into $k-f$ non-empty chains.
  To ensure that every chain is non-empty, we first place $k-f$ white vertices,
  one per chain.
  The count of remaining configurations is given by the multiset coefficient \(  \multiset {k-f} {(n-k) - (k-f)} \),
  which can be rewritten as \(  { n-k-1 \choose k-f-1 }\).
  Thus, \(g_{n,k,f} = {k \brack f}  { n-k-1 \choose k-f-1 } \)
  and given that \(g_{n,k} = \sum_{f=1}^{k-1}  g_{n,k,f}\),
  we arrive at the desired result.
\end{proof}

\section{Asymptotic Estimates Using the Saddle-Point Method}\label{sec:asymptotic}

In this section, we establish an asymptotic estimate for the number of \gitgraph{}s.

\begin{theorem}\label{thm:equivalent}
  The number~$g_n$ of Git graphs of size~$n$ is asymptotically equivalent to
  \begin{equation*}
    g_n \underset{n \to \infty}\sim
    \frac{e^{1/8}}{2}
    \left(\frac{n}{2e}\right)^{\frac n 2}
    \exp\left(\frac 1 2 \ln\frac n 2 \sqrt{\frac n 2} + \sqrt{2 n} + \frac{(\ln\frac n 2)^2}{32} \right)
    \left(\frac n 2\right)^{-3/8}
    \left(1 + \widetilde O \left( n^{-1/2} \right)\right).
  \end{equation*}
\end{theorem}

Throughout this paper, $\widetilde O$ (also known as the \emph{soft big oh notation}) is used to denote the standard big \(O\) notation,
extended to suppress polylogarithmic factors.  
Formally,
\(f(n) = \widetilde O(g(n))\) means that there exists a constant \(m\) such that \( f(n) = O\left( g(n) \ln(n)^m \right)\).

This asymptotic estimate was achieved through independent calculations
performed both manually and with the help of Maple software.
Additionally, this result was verified through numerical experimentation.

The proof of the theorem is technical and intricate, and will be developed in this section.

We adopt a deliberate pedagogical approach: 
we unfold the proof with technical transparency,
articulating our strategy at every juncture,
so that a graduate student encountering this material
can follow the argument in its entirety.

\subsection{Principle of the Proof}
\label{start:proof}

In a few words, our approach
%to derive an asymptotic estimate for the numbers $g_n$
starts by investigating the asymptotic behavior of $g_{n,k_n}$, the number of \gitgraph{}s of size $n$ with $k_n$ \main{}s,
where $k_n$ depends on the size $n$.
This is achieved using the saddle-point method.
Then,
we sum these contributions over all $k_n$
to obtain the asymptotic behavior of $g_n$.

In order to know how $g_{n,k_n}$ asymptotically behaves,
we rely on a generating function for \gitgraph{}s,
different from what we saw at Subsection~\ref{ss:generating_functions},
where this time we fix $k$ the number of \main{}s.

\begin{proposition}
  Let $k \geq 0$ be fixed.
  The ordinary generating function $G_k(z) := \sum_{n \geq 0} g_{n,k}\,z^n$ of \gitgraph{}s with $k$ \main{}s is given by
  \[ G_k(z)
  = z^k \, \prod_{j=0}^{k-1} \left(1 + \frac{jz}{1 - z}\right)
  = \frac{z^{2k}}{(1-z)^k}
   \frac{\Gamma(\frac{1-z}{z} + k)}{\Gamma(\frac{1-z}{z})}. \]
   \label{prop:Gkz}
\end{proposition}

\begin{proof}
The relation
\begin{equation}
  \label{eq:Gkz}
G_k(z)
  = z^k \, \prod_{j=0}^{k-1} \left(1 + \frac{jz}{1 - z}\right)
\end{equation}
can be derived by expanding  \( \widetilde G (z,u) = {\left(1 - \frac{z^2 u}{1 - z}\right)}^{- \frac{1 - z}{z}} \) (see Theorem~\ref{theo:Gtilde}) and extracting the coefficient in $u^k/k!$.
Alternatively --- and this is the approach we are going to develop for this proof,
this relation can also be explained combinatorially
with symbolic method~\cite{FS2009}.

Indeed, to construct a \gitgraph{} with $k$ \main{}s, we start with a main branch made of $k$ commits.
Then, for each $j \in \{0,\dots,k-1\}$,
we determine whether the \main{} located at position $j+1$ (from left to right)
is either free (weight of a single vertex : $z$),
or not.
In the latter case, it is the commit where a feature branch is merged
(weight of a \main{} with a non-empty sequence of \feature{}s : $z^2/(1-z)$)
with $j$ possibilities for choosing a \main{} to its left.
If we aggregate all possibilities,
we see that the generating function is
\(\prod_{j=0}^{k-1} \left(z + j\frac{z^2}{1 - z} \right),\)
which is consistent with \eqref{eq:Gkz}.

To obtain the expression in terms of $\Gamma$,
we use the fundamental property of the Gamma function, namely \(\Gamma(x+1) = x \Gamma(x)\). Indeed, by a simple induction on $k$, we see that
\[
\Gamma\left( \frac{1-z} z + k \right)  = \prod_{j=0}^{k-1} \left( \dfrac{1-z} z + j \right)  \Gamma\left( \dfrac{1-z} z  \right)
\]
and so
\[
\left(\frac{z}{1-z}\right)^k
\Gamma\left( \frac{1-z} z + k \right)  =
\prod_{j=0}^{k-1} \left( 1 + j \frac{z}{1-z} \right)  \Gamma\left( \dfrac{1-z} z  \right).
\]
We thereby find the expression of $G_k$ as stated by the proposition by substituting \(\prod_{j=0}^{k-1} \left( 1 + j \frac{z}{1-z} \right)\) by  \( \frac{z^k \Gamma\left( \frac{1-z} z + k \right) }
{ (1-z)^k\Gamma\left( \frac{1-z} z  \right)  }
\) in \eqref{eq:Gkz}.
\end{proof}

From this proposition, we can express the coefficients $g_{n,k}$ of $G_k(z)$ as integrals over a contour, using Cauchy's integral formula:

\begin{corollary}
  For any $n \geq 0$ and $k \in \{1,\dots,n\}$, the number of \gitgraph{}s with $n$ commits and $k$ \main{}s is given by the formula
  \[
  g_{n,k} = \frac 1 {2 \pi i}
  \oint  \dfrac{ z^{2k-n-1}} {(1-z)^k}
  \dfrac{\Gamma(\frac{1-z}{z} + k)}{\Gamma(\frac{1-z}{z})} dz
  \]
  where the contour is any circle of radius $<1$ around the origin in the complex plane.
  \label{cor:cauchy}
\end{corollary}

\begin{proof}
  Straightforward by Cauchy's integral formula
  for derivatives with function $G_k$. Note that the radius of convergence of $G_k$ is $1$, which can be easily inferred from \eqref{eq:Gkz}.
\end{proof}

Thus, the proof is structured as follows
(as depicted in Figure~\ref{fig:plan}):

  \begin{figure}[ht!]
\begin{center}
\includegraphics[scale = 1.1]{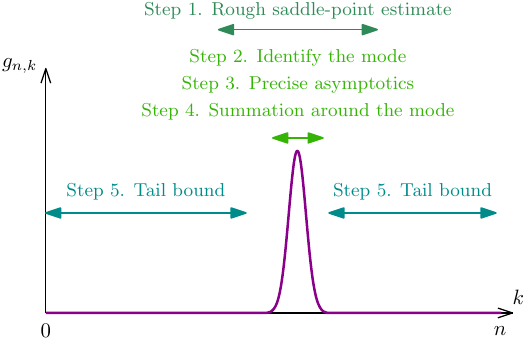}
\end{center}
\vspace{-7pt}
\caption{
  Plot of $g_{n,k}$, the number of \gitgraph{}s of size $n$ (here $n = 150$) and with $k$ \main{}s,
as a function of $k$.
The figure also outlines the different proof steps, associated with ranges of $k$ values.
}
\label{fig:plan}
\end{figure}

  \begin{enumerate}
    \item \textbf{Step 1.}
    Provide a rough estimate for $g_{n,k_n}$
    when $k_n$ is approximately equal to $\frac n 2$
    --- specifically, when $k_n = \frac{n}{2} + \varepsilon_n$
    with $\varepsilon_n = \widetilde O \left( \sqrt n \right)$.
    To achieve this, we use the fact that $ g_{n,k_n} \leq \frac{G_{k_n}(r)}{{r}^n}$ for any positive number $r$ and in particular for $r_n$, the solution of the saddle-point equation \( r_n G_{k_n}'(r_n) = n\, G_{k_n}(r_n) \).

    \item \textbf{Step 2.}
    Utilizing the previous estimate, optimize $\varepsilon_n$ to maximize~$g_{n,k_n}$.
    It will pinpoint the values of~$k$ that most significantly contribute to the sum~$g_n = \sum_k g_{n,k}$
     (specifically, whenever $k_n$ is at distance at most $\left(\frac n 2\right)^{\frac 1 4} \ln{\left(\frac n 2\right)}$
        from the mode).

    \item  \textbf{Step 3.}
    Use the full saddle-point method to give a precise estimate of $g_{n,k_n}$
    for the values of $k_n$ in this interval.

    \item  \textbf{Step 4.}
    Estimate the sum of those contributions.

    \item  \textbf{Step 5.}
    Establish upper bounds for the tails of the sum.
  \end{enumerate}

  Our approach uses standard techniques, but is quite technical due to the maximization process and the fact that we have to sum asymptotic estimates.

\subsection{Step 1. First-order Estimate at the Saddle Point}

\subsubsection{Locating the Saddle Point \texorpdfstring{\(\boldsymbol{r_n}\)}{r\_n} }

The initial step in the saddle-point method involves estimating \(r_n\),
the positive real solution to the equation \(r_n G_{k_n}'(r_n) = n G_{k_n}(r_n)\).

To understand why, remark that this equation is equivalent to
\( \frac{\partial } {\partial z} \left.\left(  \frac{G_{k_n}(z)}{z^n}\right)\right|_{z=r_n} = 0\).
By Corollary~\ref{cor:cauchy},
the number $g_{n,k_n}$, which we wish to estimate, corresponds to
the complex integral of \( \frac{G_{k_n}(z)}{z^{n+1}} \) over a circular contour --- see Figure~\ref{fig:plot3d} for an illustration.
We have the flexibility to choose the radius of the circle,
and aim to select it where the modulus is minimized
--- this is precisely what \( \frac{\partial } {\partial z} \left(  \frac{G_{k_n}(z)}{z^n}\right) = 0\) implies.\footnote{
  The discrepancy of a factor $z$ between \(\dfrac{G_{k_n}(z)}{z^n}\) in the saddle-point equation
  and  \(\dfrac{G_{k_n}(z)}{z^{n+1}}\) in Cauchy's formula (Corollary~\ref{cor:cauchy})
  is due to a change of variable in the integral.
  When we set \(z=re^{i\theta}\),
  the contour integral \(\oint  \dfrac{G_{k_n}(z)}{z^{n+1}}  dz\)
  becomes \(i \int_{-\pi}^\pi   \dfrac{G_{k_n}(re^{i\theta})}
  {   \left(r e^{i \theta}\right)^{n}      } d\theta\).
}
Thus, the solution $r_n$ to the saddle-point equation provides an optimal choice for the circle's radius.

\begin{figure}[ht!]
\begin{center}
\includegraphics[width=0.95\textwidth]{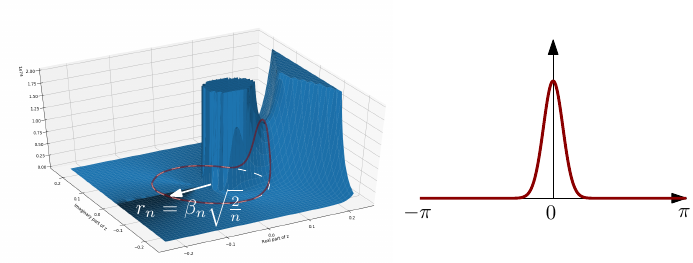}
\end{center}
\vspace{-7pt}
\caption{ \textit{Left.}
3D plot of the function
 $z \mapsto \left| \frac{G_k(z)}{z^{n+1}}
  \right|$
  for $n = 100$ and $k=54$.
  The integral contour we use to estimate $g_{n,k}$ will be the circle of center $0$ and radius $r_n = \beta_n \sqrt{ \frac 2 n}$.
  The curve (which is hand-drawn) represents what we integrate in Corollary~\ref{cor:cauchy}.
  \textit{Right.}
  Plot of $ \left| \frac{G_k(z)}{z^{n}}
  \right|$
for  $n = 100$, $k=54$. Here,
$z = r \exp(i \theta)$  with $\theta$ varying from $-\pi$ to $\pi$,
while $r$ is set to the approximate saddle-point value
$\sqrt{ \frac 2 n } = \frac 1 {\sqrt {50}}$
($\beta_n$ is expected to be closed to $1$ whenever $k_n$ is near the mode).
}
\label{fig:plot3d}
\end{figure}

  \begin{lemma}[Solution \(r_n\) of the saddle point equation]
  \label{lem:def_saddlepoint}
  Let~$k_n$ be any integer sequence.

  When \(1 < k_n \leq n\), there exists a unique number $r_n$ between $0$ and $1$ that satisfies the (saddle point) equation
  \begin{equation*}
    r_n G_{k_n}'(r_n) = n G_{k_n}(r_n).
  \end{equation*} 
  \end{lemma}

  \begin{proof}
    Let $n$ be fixed. We examine how the function \(z \mapsto \frac{z G_{k_n}'(z)}{G_{k_n}(z)}\) varies.

    We now suppose that $k_n$ takes values in \(\{2,\dots,n\}\). Using the expressions of $G_{k}(z)$ from Proposition~\ref{prop:Gkz}, we get
    \( \ln(G_{k_n}(z)) = k_n \ln(z) + \sum_{j=0}^{k_n-1} \ln \left(1 + \frac{jz}{1-z}\right)  \).
    And thus, by differentiation,
    \begin{equation}
      \frac{z G_{k_n}'(z)}{G_{k_n}(z)} = k_n
    + z \sum_{j=0}^{k_n-1} \frac j {\left((j-1)z+1\right)(1-z)}.
      \label{eq:zGp/G}
    \end{equation}
    This expression tends to $k_n$ as $z \rightarrow 0$, and goes to $+\infty$ whenever $z \rightarrow 1^-$.
    Consequently, by the intermediate value theorem,
    since $k_n \leq n$,
    there exists a non-negative number \(r_n\) such that
    \[\frac{r_n G_{k_n}'(r_n)}{G_{k_n}(r_n)} = n.
    \]
    This number \(r_n\) is unique since \(G_{k_n}(z)\) has non-negative coefficients, which implies that \(z \mapsto  \frac{z G_{k_n}'(z)}{G_{k_n}(z)} \) is strictly increasing (cf Remark IV.46 from \cite{FS2009}).
  \end{proof}

  \begin{figure}[ht!]
  \begin{center}
  \includegraphics[width=0.45\textwidth]{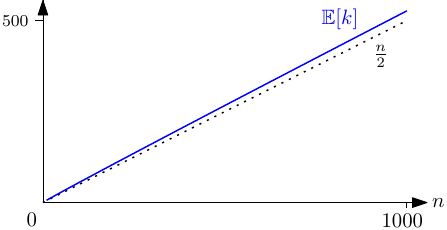}
  \hfill
  \includegraphics[width=0.45\textwidth]{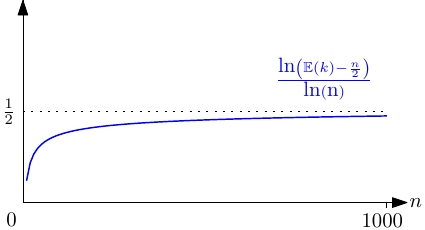}
  \end{center}
  \vspace{-7pt}
  \caption{ \textit{Left.} 
  The average \(k\)-value \( \mathbb E[k] = \frac {\sum_{k} k g_{n,k}} {\sum_{k} g_{n,k}}\) 
  (blue curve)
  remains close to \(\frac n 2\) (dashed line) as \(n\) grows.
  \textit{Right.} 
  The exponent in the deviation \(\mathbb E[k] - \frac n 2\)
  (computed as \(\frac{ \ln\left( \mathbb E[k] - \frac n 2 \right) } {\ln(n)}\))
  seems to converge to \(\frac 1 2\),
  consistent with \(\sqrt{n}\)-scaling.
  }
  \label{fig:mean}
  \end{figure}

  We orient now our study around the region where \(g_{n,k}\) concentrates most of its mass.
  Figure~\ref{fig:mean} provides empirical guidance.
  The left panel suggests that the mean \(k\)-value stays close to \(\frac n 2\),
  while the right panel reveals that the second-order term should scale as \(\sqrt n\).
  We therefore investigate the ansatz \(k_n = \frac n 2 + \alpha_n \sqrt{\frac n 2}\)
  with \(\alpha_n = \widetilde O(1)\).
  This choice is initially motivated by experiment,
  but we shall see that it captures the dominant behavior.

    \begin{lemma}[Estimation of the saddle point]
    \label{lem:estimation_saddlepoint}
     Let~$k_n$ be any integer sequence such that $k_n = \frac{n}{2} + \alpha_n \sqrt{\frac{n}{2}}$
  where $\alpha_n = \widetilde O(1)$.

    Set $\beta_n = r_n \sqrt{\frac n 2}$, 
    where \(r_n\) is the solution of the saddle point equation \(r_n G_{k_n}'(r_n) = n G_{k_n}(r_n)\).

    Then $\beta_n$ is a positive sequence such that
    \(\beta_n = \widetilde O(1)\) and \(\beta_n^{\ -1} = \widetilde O(1)\). 
    Moreover, we have the estimate
    \begin{equation*}
      \alpha_n =
        \frac{\ln\frac n 2}{4\beta_n} - \frac{\beta_n}{2} + \frac{\ln\beta_n}{2\beta_n}
        + E_n(\beta_n)\sqrt{\frac{2}{n}}
        + \widetilde O(n^{-1})
    \end{equation*}
    with 
    \[E_n(x) =
        \frac 1 8 \left( {x}^{-2} - 1\right) \ln  \left( {\frac n 2}
      \right) 
      - \frac {{x}^{2}} 4    
      + \frac {{x}^{-2}} 2   
      + \frac 1 2  
      - \frac {\ln  \left(x \right)} 4 
      + {\frac {\ln  \left( x \right) }{4 {x}^{2}}}.
    \]
  \end{lemma}

  \begin{proof}  
     Since \(k_n\) eventually satisfies \(1 < k_n  \leq n \), 
    Lemma~\ref{lem:def_saddlepoint} guarantees the existence of \(r_n\) for \(n\) sufficiently large.

    \textbf{Proof of \(\boldsymbol{  \beta_n   = \widetilde O (1) }\).} 
    We use the expression of $G_k$ from Proposition~\ref{prop:Gkz} in terms of the Gamma function to derive
    \begin{equation}\label{eq:saddlepoint:function}
      \frac{z G_{k_n}'(z)}{G_{k_n}(z)} =  \frac{2-z}{1-z} k_n - \frac{1}{z}\left(\digamma\left(\frac{1-z}{z} + {k_n}\right) - \digamma\left(\frac{1-z}{z}\right)\right)
    \end{equation}
    where~$\psi$ is the digamma function~$\psi(x) = \frac{\Gamma'(x)}{\Gamma(x)}$.

    We have \(\alpha_n = \widetilde O(1)\), so there exists \(m>0\)
    such \(\alpha_n = O(\ln(n)^m)\).

    Set $z_n :=  \sqrt{\frac 2 n}  \ln\left( \frac n 2 \right)^{m+1}$.
     We wish to estimate Equation~\eqref{eq:saddlepoint:function} at \(z = z_n\).
    We use the asymptotic approximation \( \psi(x) \underset{x\to+\infty}\sim   \ln(x) + O\left(\frac 1 x \right) \)
    to get after routine algebra
    \[
    \frac{z_n G_{k_n}'(z_n)}{G_{k_n}(z_n)} - n \sim \sqrt{\frac n 2} \ln\left( \frac n 2 \right)^{m+1} \rightarrow +\infty.
    \]
    This is positive for large $n$. 
    By monotonicity of \(z \mapsto  \frac{z G_{k_n}'(z)}{G_{k_n}(z)} \), 
    we deduce that $r_n$, the solution of $\frac{r_n G_{k_n}'(r_n)}{G_{k_n}(r_n)} = n$, is smaller than $z_n = \sqrt{\frac 2 n}  \ln\left( \frac n 2 \right)^{m+1}$.
    Thus \(\beta_n := r_n \sqrt{n/2}\) is such that for~$n$ large enough~\(\beta_n < \ln\left( \frac n 2 \right)^{m+1} \), hence \(\beta_n = \widetilde O (1)\).

    \textbf{Proof of \(\boldsymbol{  \frac{1} {\beta_n}   = \widetilde O (1) }\).} 
    We approximate  Equation~\eqref{eq:saddlepoint:function}
    but this time for $z_n :=  \sqrt{\frac 2 n}  \ln\left( \frac n 2 \right)^{-m}$
    where \(m\) is such that \(\alpha_n = O(\ln(n)^m)\).
    Here, we get that 
    \[
    \frac{z_n G_{k_n}'(z_n)}{G_{k_n}(z_n)} - n \sim  - \frac 1 2 \sqrt{\frac n 2} \ln\left( \frac n 2 \right)^{m+1} \rightarrow - \infty.
    \]
    This shows that \( r_n > z_n = \sqrt{\frac 2 n}  \ln\left( \frac n 2 \right)^{-m} \) for $n$ sufficiently large, and so \(\beta_n >  \ln\left( \frac n 2 \right)^{-m}\), which implies \(1/\beta_n = \widetilde O (1)\).

    \textbf{Full expansion of $\boldsymbol{\alpha_n}$ in terms of $\boldsymbol{\beta_n}$.}
    By definition of $r_n$,
    \begin{equation}
      \frac{r_n G_{k_n}'(r_n)}{G_{k_n}(r_n)} = n.
      \label{eq:saddle}
    \end{equation}
    To derive an asymptotic relationship between $\alpha_n$ and $\beta_n$, 
    we substitute $z = r_n = \beta_n / \sqrt{n/2}$ and $k_n = n/2 + \alpha_n \sqrt{n/2}$ in Equation~\eqref{eq:saddlepoint:function}, and carry out an asymptotic analysis, keeping terms up to $\widetilde O(n^{-1/2})$. 
    (Recall that the $\widetilde O$ notation includes polylogarithmic factors, which suppresses $\alpha_n$ and $\beta_n$.)
    More specifically, we have
    \begin{eqnarray}
      \frac{2-r_n}{1-r_n}k_n  = & n
      + (2 \alpha_n + \beta_n) \sqrt{n/2}
      + \alpha_n\beta_n
      + \beta_n^2
      + \widetilde O\left(\dfrac 1 {\sqrt n}\right),
      \label{eq:first_saddle}
      \\
      \frac{ \sqrt{n/2} } {\beta_n}\ \psi\left( \frac{1-r_n}{r_n}+k_n \right) =
      & \dfrac{\ln(n/2)}{\beta_n}\sqrt{n/2}
      + \dfrac{\alpha_n}{\beta_n}
      + \dfrac 1 {\beta_n^2}
      + \widetilde O\left(\dfrac 1 {\sqrt n}\right),
      \label{eq:second_saddle}
      \\
      \frac{ \sqrt{n/2} } {\beta_n} \  \psi\left( \frac{1-r_n}{r_n} \right) =
      & \dfrac{\ln(n/2)}{2 \beta_n}\sqrt{n/2}
      - \dfrac{\ln(\beta_n)}{\beta_n}\sqrt{n/2}
      - \dfrac 3 2
      +  \widetilde O\left(\dfrac 1 {\sqrt n}\right).
      \label{eq:third_saddle}
    \end{eqnarray}
    (We have used the estimate \( \psi(x) \underset{x\to+\infty}\sim   \ln(x) - \frac 1 {2\,x} + O\left(\frac 1 {x^2} \right) \)
    (see~\cite[6.3.18]{abramowitz-stegun}) for Equation~\eqref{eq:third_saddle}.)
    Putting these three expansions together (cf.~\eqref{eq:saddlepoint:function}), we deduce from Equation~\eqref{eq:saddle}
    \begin{equation*}
      0 =
    \left( 2 \alpha_n
    + \beta_n
    - \dfrac{\ln(n/2)}{2 \beta_n}
      - \dfrac{\ln(\beta_n)}{\beta_n}
       \right)\sqrt{n/2}
      + \alpha_n\beta_n
      + \beta_n^2
      - \dfrac{\alpha_n}{\beta_n}
      - \dfrac 1 {\beta_n^2}
      - \dfrac 3 2
      +  \widetilde O\left(\dfrac 1 {\sqrt n}\right),
    \end{equation*}
    or equivalently
    \begin{equation}
      \alpha_n =
    \dfrac{\ln(n/2)}{4 \beta_n}
    + \dfrac{\ln(\beta_n)}{2\beta_n}
    - \frac{ \beta_n} 2
    + \left(
      - \frac {\alpha_n\beta_n } 2
      - \frac {\beta_n^2}  2
      + \dfrac{\alpha_n}{2\beta_n}
      + \dfrac 1 {2 \beta_n^2}
      + \dfrac 3 4
    \right) \dfrac 1 {\sqrt {n/2}}
      +  \widetilde O\left(\dfrac 1 {n}\right).
      \label{eq:another_saddle_equation}
    \end{equation}
    By omitting $ \widetilde O\left(\frac 1 {\sqrt n}\right)$ terms, we find that
    \begin{equation*}
      \alpha_n =
       \dfrac{\ln(n/2)}{4 \beta_n}
      + \dfrac{\ln(\beta_n)}{2 \beta_n}
      - \dfrac {\beta_n} 2
      + \widetilde O\left(\dfrac{1}{\sqrt n}\right).
    \end{equation*}
    Bootstrapping this into the parentheses of Equation~\eqref{eq:another_saddle_equation}, we recover the asymptotic expansion of $\alpha_n$ as specified by the lemma.
  \end{proof}

In Lemma~\ref{lem:estimation_saddlepoint},
we provide the asymptotic expansion of \( \alpha_n \) in terms of \( \beta_n \),
even though \( \alpha_n \) is our initial sequence
and \( \beta_n \) the sequence we wish to study.
This is motivated by the absence of a satisfactory asymptotic representation
of \( \beta_n \) in terms of \( \alpha_n \),
except for some specific \(\alpha_n\)'s.

The following corollary establishes a first-order approximation of \(\beta_n\) in terms of \(\alpha_n\)
for the regime \(\alpha_n \gg \sqrt{ \ln n}\)
that we will encounter.

\begin{corollary}
  Let $\alpha_n$ be a \( \widetilde O(1)\) sequence such that $\frac n 2 + \alpha_n \sqrt{\frac n 2}$ is an integer, and $\beta_n$ be as defined by Lemma~\ref{lem:estimation_saddlepoint}. We have the asymptotic estimate
  \begin{equation*}
    \beta_n \sim \sqrt{\alpha_n^2 + \frac{\ln(n)}{2}} - \alpha_n.
  \end{equation*}
  In particular, when \( \alpha_n \) dominates \(\sqrt{ \ln n } \) asymptotically,
  we have
  \begin{equation*}
    \beta_n \sim  \frac {\ln\left( n  \right)} {4 \alpha_n}.
  \end{equation*}
  \label{cor:beta_in_terms_of_alpha}
\end{corollary}

Observe that the asymptotic behavior of \( \beta_n \) is very sensitive to the choice of \( \alpha_n \).
By example, if \(\alpha_n\) is negligible relatively to \(\sqrt{\ln(n)}\), then one can show that
$\beta_n$ is asymptotically equivalent to
\( \sqrt{ \frac{ \ln(n)} 2 }\).

\begin{proof}[Proof of Corollary~\ref{cor:beta_in_terms_of_alpha}]
  By Lemma~\ref{lem:estimation_saddlepoint}, we have
  \begin{equation*}
    \alpha_n \sim
        \frac{\ln\frac n 2}{4\beta_n} - \frac{\beta_n}{2} + \frac{\ln\beta_n}{2\beta_n}.
  \end{equation*}
  But \(\beta_n = \widetilde O(1)\) and \(\beta_n^{\ -1} = \widetilde O(1)\), 
  which implies that there exist three positive constants $C_1$, $C_2$ and $m$ with
  \( C_1 \ln(n)^{-m} < \ln\left( \beta_n \right) < C_2 \ln(n)^{m}\)
  for all \(n\) sufficiently large.
  Consequently, \(  - m \ln \ln(n) + \ln(C_1) < \beta_n <   m \ln \ln(n) + \ln(C_2) \).
  This means that
  \(\ln \left( \beta_n \right) \) is asymptotically bounded by \(\ln \ln(n)\), hence
  negligible compared to  \( \ln(n) \).
  Thus,
  \begin{equation*}
    \alpha_n \sim
        \frac{\ln (n)}{4\beta_n} - \frac{\beta_n}{2},
  \end{equation*}
  and so
  \begin{equation*}
    \frac {\ln(n)} 2
    \sim \beta_n^2 + 2 \alpha_n \beta_n = (\alpha_n + \beta_n)^2 - \alpha_n^2.
  \end{equation*}
   Adding \(\alpha_n^2\) to both sides and taking the square root provides the first estimate of the corollary.

   As for the second estimate, write
   \begin{equation*}
    \beta_n \sim
    \alpha_n \sqrt{ 1 + \frac{\ln\left( n \right)}{2 \alpha_n^2}  } - \alpha_n
    \sim
    \alpha_n \left(  1 +  \frac{\ln\left( n \right)}{4 \alpha_n^2 } + O\left( \frac{\ln(n)^2}{\alpha_n^4} \right) \right) - \alpha_n
     \sim  \frac {\ln\left( n  \right)} {4 \alpha_n}.
   \end{equation*}
   This completes the proof.
\end{proof}

We now prove that the saddle point parameter \(\beta_n\) is decreasing in \(k_n\),
a key property for controlling the uniformity in the error bounds.

\begin{lemma}[Monotonicity of \({\beta_n}\) as a function of \({k_n}\)]
Let \(k_n\) and \(\widetilde{k_n}\) be two integer sequences 
such that \(k_n \leq \widetilde{k_n}\) for every $n \geq 0$.

Let $r_n$ and $\widetilde{r_n}$ be the solutions of the respective saddle point equations,
and set \(\beta_n = r_n \, \sqrt{n/2} \) and \(\widetilde{\beta_n} = \widetilde{r_n} \,\sqrt{n/2}\).

For every \(n \geq 0\), if $r_n$ and \(\widetilde{r_n}\) are well-defined\footnote{
which happens whenever \(2 \leq k_n \leq n\) and \(2 \leq \widetilde{k_n} \leq n\)},
then \(r_n \geq \widetilde{r_n}\) and so \(\beta_n \geq \widetilde{\beta_n}\).
\label{lem:decreasing_beta}
\end{lemma}

\begin{proof}
By Equation~\eqref{eq:zGp/G}, we have
\begin{equation*}
  \frac{   r_n \, G'_{ \widetilde{k_n} } (r_n)} { G_{ \widetilde{k_n} }(r_n)}
  =  \frac{   r_n \, G'_{ k_n } (r_n)} { G_{ k_n }(r_n)   }
  + \widetilde{k_n} - k_n
  +  \sum_{j = k_n - 1}^{\widetilde{k_n} - 1}  \frac{r_n j}{((j-1)r_n + 1)\,(1-r_n)}.
\end{equation*}
By Equation~\eqref{eq:saddle}, \(\frac{   r_n \, G'_{ k_n } (r_n)} { G_{ k_n }(r_n)   }  = n.\)
Moreover,
\(
  \widetilde{k_n} - k_n \geq 0,
\)
and \(\frac{r_n j}{((j-1)r_n + 1)\,(1-r_n)} \geq 0\) (remember that \(0 \leq r_n < 1\)).
Consequently,
\begin{equation*}
  \frac{   r_n \, G'_{ \widetilde{k_n} } (r_n)} { G_{ \widetilde{k_n} }(r_n)}
  \geq n =  \frac{   \widetilde{r_n} \, G'_{ \widetilde{k_n} } (\widetilde{r_n} )} { G_{ \widetilde{k_n} }(\widetilde{r_n} )}.
\end{equation*}
By monotonicity of \(z \mapsto  \dfrac{z G_{\widetilde{k_n}}'(z)}{G_{\widetilde{k_n}}(z)} \) \cite[Remark IV.46]{FS2009},
we have \(r_n \geq \widetilde{r_n}\),
hence \(\beta_n \geq \widetilde{\beta_n}\).
\end{proof}

The following lemma extends Lemma~\ref{lem:estimation_saddlepoint} by addressing the uniform boundedness
(meaning that the constants hidden in the big-oh notation can be made independent of \(\alpha_n\)).
This property is essential as we will later require the ability to sum these asymptotic equivalents across different values of $(\alpha_n)$.
We can limit \(\alpha_n\) to being within $\ln(n)$ in absolute values,
given that the relevant range of \(\alpha_n\) will asymptotically cluster around $\ln(n)/4$.

\begin{lemma}\label{lem:saddlepoint:unif}
  There exist an integer \(n_0\) and
  a \(\widetilde O(1)\) sequence $\delta_n$ such that for every \(n \geq n_0\) and
  for every sequence \(\alpha_n\)
  satisfying \( \left| \alpha_n \right| \leq \ln (n) \),
  \begin{equation*}
    \left|\alpha_n
      -\frac{\ln\frac n 2}{4\beta_n} + \frac{\beta_n}{2} - \frac{\ln\beta_n}{2\beta_n}
      - E_n(\beta_n)\sqrt{\frac{2}{n}}
      \right| < \frac{ \delta_n }{n}
  \end{equation*}
    (see Lemma~\ref{lem:estimation_saddlepoint} for the definitions of \(\beta_n\) and \(E_n\)).\footnote{
      Equivalently, this lemma states that for all \(\alpha_n\) with \( \left| \alpha_n \right| \leq \ln (n) \),
    \begin{equation*}
      \alpha_n  = \frac{\ln\frac n 2}{4\beta_n} - \frac{\beta_n}{2} + \frac{\ln\beta_n}{2\beta_n}
      + E_n(\beta_n)\sqrt{\frac{2}{n}}
       + \widetilde{O}\left( \frac 1 n \right),
  \end{equation*}
  where the error term \(\widetilde{O}\left( \frac 1 n \right)\) is uniform in \(\alpha_n\).
    }
    \label{lem:uniform_bound_saddlepoint}
\end{lemma}

\begin{proof}

  Consider any sequence \(\alpha_n\) such that \(|\alpha_n| \leq \ln(n)\).

  The sequences
  \(\lfloor n/2 - \ln(n) \sqrt{n/2} \rfloor\) and \(\lceil n/2 + \ln(n) \sqrt{n/2}\rceil\)
  satisfy the hypotheses of Lemma~\ref{lem:estimation_saddlepoint}
  %(in particular, starting from $n = 32$, both sequences belong to   \(\{2,3,\dots,n-2\}\)),
  so we can associate with them two \(\beta_n\) sequences,
  respectively denoted as \( \beta_n^+\) and  \( \beta_n^-\).
  (These two sequences do not depend on \(\alpha_n\).)

  Using the monotonicity established by Lemma~\ref{lem:decreasing_beta},
  we see that \( \beta_n^- \leq \beta_n \leq \beta_n^+\).

  Moreover, we can demonstrate that each \(\widetilde O \left( \frac 1 {\sqrt n}\right)\)
  in Equations~\eqref{eq:first_saddle}, \eqref{eq:second_saddle} and~\eqref{eq:third_saddle}
  can be written as \(C_n \, Pol(\alpha_n, \beta_n, 1/\beta_n, \ln n) \,n^{-1/2}\),
  where $C_n$ is a sequence bounded by a number that does not depend on the sequence \(\alpha_n\) nor \(\beta_n\),
  and $Pol$ is a polynomial.
  These polynomials can be computed explicitly;
  however the exact calculations are not particularly insightful.
  They arise from the expansions of the digamma function, the logarithm or a rational function.

  We carry out the argument for Equation~\eqref{eq:second_saddle} in detail;
  the other cases are analogous.

  Rewrite \(\psi\left( \frac{1-r_n}{r_n}+k_n \right)\) as
  \(\psi\left(\frac n 2 \left( 1 +  \epsilon_n \right)\right)\)
  with \(\epsilon_n = \frac{\alpha_n}{\sqrt{n/2}}  + \frac 1 {\beta_n \sqrt{n/2}} - \frac 2 {n}\).
  We observe that
  \begin{equation*}
    -\frac{\ln(n)}{\sqrt{n/2}} - \frac 2 {n} \leq \epsilon_n \leq \frac{\ln(n)}{\sqrt{n/2}} + \frac 1 {\beta_n^- \sqrt{n/2}}.
  \end{equation*}
  The lower and upper bound tend to \(0\).
  Therefore there exists an integer $n_0$
  such that for every \(n \geq n_0\),
  \begin{equation*}
    -\frac 1 2  < -\frac{\ln(n)}{\sqrt{n/2}} - \frac 2 {n} \quad \text{and} \quad
    \frac{\ln(n)}{\sqrt{n/2}} + \frac 1 {\beta_n^- \sqrt{n/2}} < \frac 1 2.
  \end{equation*}
  Hence $\epsilon_n  \in [- \frac 1 2 , \frac 1 2 ]$ for all $n \geq n_0$.

  The asymptotic estimate of $\psi(x)$ as \(x\) goes to \(+\infty\) can be written as  $\psi(x) = \ln(x) + C_{\psi}(x)/x$,
  where \(C_{\psi}\) is an asymptotically bounded function.
  Since \(C_{\psi} = x (\psi(x)-\ln(x))\),
  the function \(C_{\psi}\) is continuous on $[1/4,+\infty)$,
  and hence bounded there.
  We infer that \( \frac{C_\psi\left( \frac n 2 \left( 1 +  \epsilon_n\right) \right)}
     { \left( 1 +  \epsilon_n\right)} \) is bounded independently of \(\alpha_n\),
     since \( \frac n 2 \left( 1 +  \epsilon_n\right) \geq 1/4 \).

  Furthermore,
  \begin{equation*}
    \ln \left( \frac n 2 \left( 1 +  \epsilon_n\right) \right) =
    \ln \left( \frac n 2 \right)
    + \ln \left( 1 +  \epsilon_n\right)
    =  \ln \left( \frac n 2 \right) + \epsilon_n + C_{\ln}(\epsilon_n) \epsilon_n^2,
  \end{equation*}
  where \( C_{\ln}(x) := (\ln(1+x) - x)/x^2 \) is a function locally bounded near \(0\), and therefore, by continuity, bounded on \([-\frac 1 2, \frac 1 2 ]\).

  Combining these estimates yields
  \begin{equation*}
    \psi\left( \frac{1-r_n}{r_n}+k_n \right) =
     \ln \left( \frac n 2 \right) +
     \epsilon_n + C_{\ln}(\epsilon_n) \epsilon_n^2
     + \frac{C_\psi\left( \frac n 2 \left( 1 +  \epsilon_n\right) \right)}
     {\frac n 2 \left( 1 +  \epsilon_n\right)}
  \end{equation*}
and after multiplying through by $\sqrt{n/2}/\beta_n$ and substituting the definition of $\epsilon_n$,
\begin{equation*}
  \frac{ \sqrt{n/2} } {\beta_n}\ \psi\left( \frac{1-r_n}{r_n}+k_n \right) =
      \dfrac{\ln({\frac n 2})}{\beta_n}\sqrt{\frac n 2}
      + \dfrac{\alpha_n}{\beta_n}
      + \dfrac 1 {\beta_n^2}
      + R_n,
\end{equation*}
where \[R_n := - \frac {2 \beta_n} {\sqrt{n/2}} + C_{\ln}(\epsilon_n) \epsilon_n^2 {\sqrt{\frac n 2} }
     + \frac{C_\psi\left( \frac n 2 \left( 1 +  \epsilon_n\right) \right)}
     {\sqrt{n/2}  \left( 1 +  \epsilon_n\right)} . \]
  The remainder \(R_n\) is of the form
\(C_n \, Pol(\alpha_n, \beta_n, 1/\beta_n, \ln n) \,n^{-1/2}\), as claimed.

Having established the analogous statements for Equations~\eqref{eq:first_saddle} and~\eqref{eq:third_saddle} by the same me\-thod,
we deduce that the quantity
\(\alpha_n
        -\frac{\ln\frac n 2}{4\beta_n} + \frac{\beta_n}{2} - \frac{\ln\beta_n}{2\beta_n}
        - E_n(\beta_n)\sqrt{\frac{2}{n}}\),
        which is equal to
the \(\widetilde O\left( \frac 1 n \right)\) term in Equation~\eqref{eq:another_saddle_equation}, is of the form
\(C_n \, Pol(\alpha_n, \beta_n, 1/\beta_n, \ln n) \,n^{-1}\).

We then obtain the uniform bound
\begin{equation*}
  \left|C_n\, Pol(\alpha_n, \beta_n, 1/\beta_n, \ln n)  \,n^{-1}\right| \leq C \, Pol^+(\ln n, \beta_n^+, 1/\beta_n^-, \ln n) \,n^{-1},
\end{equation*}
where \(C\) is a constant bounding the sequence \(C_n\), and \(Pol^+\) is the polynomial formed by taking the absolute values of the coefficients of $Pol$.
Setting $\delta_n = C \ Pol^+(\ln (n)^3, \beta_n^+, 1/\beta_n^-, \ln n)$ defines the required \(\widetilde O(1)\) sequence.
\end{proof}

\subsubsection{Approximating \texorpdfstring{\(\boldsymbol{r_n^{-n}G_{k_n}(r_n)}\)}{r\_n\^{}(-n) G\_(k\_n)(r\_n)}}

We now proceed to estimate \(r_n^{-n}G_{k_n}(r_n)\).
This is a crucial step in the saddle-point method,
as this term will serve as the ``\textit{dominant}" factor in the expansion of \(g_{n,k_n}\).
Moreover, this quantity is also an upper bound for \(g_{n,k_n}\),
making this estimate at this juncture is both useful and indicative of \(g_{n,k_n}\)'s behavior.

\begin{proposition}\label{prop:evalr}
  Let~$\alpha_n$ be a \(\widetilde O(1)\) sequence such that $k_n = \frac{n}{2} + \alpha_n \sqrt{\frac{n}{2}}$ is an integer sequence.

  Let~\(r_n\) and $\beta_n$ be defined as in Lemma~\ref{lem:estimation_saddlepoint}.
  We have the estimate
    \begin{equation*}
      r_n^{-n}G_{k_n}(r_n) \underset{n\to\infty}\sim
      \left(\frac{n}{2e}\right)^{\frac{n}{2}}
      \exp\left(\lambda_n(\beta_n) \sqrt{\frac{n}{2}} + \mu_n(\beta_n)\right),
    \end{equation*}
    where
    \begin{equation*}
      \lambda_n(x) =
      \frac{1 + \ln(x)}{2x}\ln\left( \frac n 2 \right) + \frac{(\ln x)^2}{x} + \left(\frac{1}{x} - x\right)\ln x + \frac 1 x + x
    \end{equation*}
and
    \begin{multline*}
      \mu_n(x) =
      \frac{\ln\left( \frac n 2 \right)^2}{32 x^2}
      + \left(\left(\frac{3}{8 x^2} - \frac{1}{4}\right)\ln x - \frac 5 8 + \frac{1}{4x^2}\right)\ln\left( \frac n 2 \right) \\
      + \left(\frac{5}{8 x^2} - \frac 1 2\right) (\ln x)^2
      + \left(\frac{3}{2 x^2} - \frac 1 4 - \frac{x^2}{2}\right) \ln x
      + \frac{1}{2 x^2} -\frac 1 2 + \frac{x^2}8.
    \end{multline*}
    \label{prop:asympt_saddle_bound}
\end{proposition}

\begin{proof}
  The proof, although computation-heavy,
  remains theoretically straightforward.
  By Pro\-position~\ref{prop:Gkz}, we have
  \begin{equation*}
    r_n^{-n}G_{k_n}(r_n) =
     \frac{{r_n}^{2 k_n - n}}{(1-r_n)^{k_n}}
   \frac
   {\Gamma(\frac{1-{r_n}}{r_n} + k_n)}
   {\Gamma(\frac{1-r_n}{r_n})}.
  \end{equation*}
  We then compute the following identities and estimates
    \begin{eqnarray*}
      r_n^{2k_n-n}
      = & \exp\left( 2\alpha_n\ln(\beta_n) \sqrt{n/2} - \alpha_n \ln(n/2) \sqrt{n/2}\right) \\
      (1-r_n)^{-k_n}
      = & \exp\left(\beta_n \sqrt{n/2}+ \frac{\beta_n^2}{2} + \alpha_n\beta_n+ \widetilde O \left( \frac 1 {\sqrt n}\right) \right) \\
      \Gamma\left(\frac{1-r_n}{r_n} + k_n\right)
      = & \sqrt{2\pi} \exp\left(
        \begin{aligned}
          \frac{n}{2}\ln \left( \frac{n}{2e} \right)
          + (\alpha_n + \beta_n^{-1}) \ln(n/2) \sqrt{n/2}
           + \frac{(a_n + \beta_n^{-1})^2}{2}
           \\
           - \frac 3 2 \ln(n/2) + \widetilde O \left( \frac 1 {\sqrt n}\right)
        \end{aligned}
      \right) \\
      \Gamma\left(\frac{1-r_n}{r_n}\right)
      = & \sqrt{2\pi} \exp\!\left(\!
        \left(\frac{\ln(n/2)}{2\beta_n}- \frac{\ln\beta_n}{\beta_n} - \frac{1}{\beta_n}\right)\sqrt{\frac{n}{2}}
        + \frac{3}{2}\ln\beta_n - \frac{3}{4}\ln \left( \frac n 2 \right)
        + \widetilde O \left( \frac 1 {\sqrt n}\right)\!
      \right)\!.
  \end{eqnarray*}
  For the last two equations, we have used Stirling's Formula
  \[ \Gamma(z) \underset{z \rightarrow +\infty}{\sim}  \sqrt{2\pi} z^{z-1/2} e^{-z} (1 + O(1/z)) \]
  valid for every \(z\) with \( \left| \arg z \right| \leq \frac \pi 2  \) (see \cite[6.1.37]{abramowitz-stegun}).

  Bringing everything together, we derive
  \begin{equation}
    r_n^{-n}G_{k_n}(r_n)
      =
      \left( \frac n {2e} \right)^{\frac n 2}
    \exp\left(
        \begin{aligned}
           \left(
            \frac{\ln(n/2)}{2\beta_n} + 2\alpha_n\ln\beta_n + \beta_n + \frac{\ln\beta_n}{\beta_n} + \frac{1}{\beta_n}\right)
            \sqrt{n/2}
          + \alpha_n\beta_n  \\
          + \frac{\beta_n^2}{2}
          + \frac{(\alpha_n + \frac{1}{\beta_n})^2}{2}
          - \frac{3}{4}\ln(n/2) - \frac{3}{2}\ln\beta_n
          +\widetilde O \left( \frac 1 {\sqrt n}\right)
        \end{aligned}
        \right).
    \label{eq:asympt_with_mixed_alpha_and_beta}
  \end{equation}
  We then use Lemma~\ref{lem:estimation_saddlepoint} to substitute \(\alpha_n\) by its expansion in terms of \(\beta_n\).
\end{proof}

\begin{corollary}
  For every sequence \(k_n\) satisfying Lemma~\ref{lem:estimation_saddlepoint}, we have
  \begin{equation*}
    g_{n,k_n} = O\left(\left(\frac{n}{2e}\right)^{\frac{n}{2}} \exp\left(\lambda_n(\beta_n) \sqrt{\frac{n}{2}} + \mu_n(\beta_n)\right)\right)
  \end{equation*}
  (see Lemma~\ref{lem:estimation_saddlepoint} for the definition of \(\beta_n\) and
  Proposition~\ref{prop:evalr} for the definitions of \(\lambda_n\) and \(\mu_n\)).
  \label{cor:upperbound}
\end{corollary}
\begin{proof}
  Since every $g_{m,k_n}$ is non-negative, we have \(g_{n,k_n} r_n^n \leq \sum_m g_{m,k_n} r_n^m = G_{k_n}(r_n)\).
  We then use Proposition~\ref{prop:evalr} to derive an asymptotic bound for \(g_{n,k_n}\).
\end{proof}

Just as Lemma~\ref{lem:uniform_bound_saddlepoint} extends Lemma~\ref{lem:estimation_saddlepoint},
the upcoming lemma seeks to generalize Proposition~\ref{prop:evalr} by establishing uniform convergence.

\begin{lemma}\label{lem:evalr:unif}
    There exist an integer \(n_0\) and
    a \(\widetilde O(1)\) sequence $\delta_n$ such that
  for every $n \geq n_0$, and for every sequence \(k_n = n/2 + \alpha_n \sqrt{n/2}\)
  satisfying \( \left| \alpha_n \right| \leq \ln (n) \),
      \begin{equation*}
      \left|r_n^{-n}G_{k_n}(r_n) -
      \left(\frac{n}{2e}\right)^{\frac{n}{2}}
      \exp\left(\lambda_n(\beta_n) \sqrt{\frac{n}{2}} + \mu_n(\beta_n)\right)\right|
      < \frac {\delta_n} {\sqrt{n}}
    \end{equation*}
    (see Lemma~\ref{lem:estimation_saddlepoint} for the definitions of \(r_n\) and \(\beta_n\), and
  Proposition~\ref{prop:evalr} for the definitions of \(\lambda_n\) and \(\mu_n\)).
\end{lemma}

\begin{proof}
  The proof follows the same reasoning as that employed in the proof of Lemma~\ref{lem:uniform_bound_saddlepoint}.

  We observe that each \(\widetilde O(n^{-1/2})\) term in the proof of Proposition~\ref{prop:evalr} is again of the form
  \( C_n Pol(  \alpha_n, \beta_n, 1/\beta_n, \ln(n) ) n^{-1/2} \)
  (where \(C_n\) is a sequence bounded by a number independent of \(\alpha_n\) and \(\beta_n\), and
  \(Pol\) is a polynomial),
  and so can be rewritten as \(O \left( \gamma_n n^{-1/2} \right) \),
  where the constant within the big-oh notation does not depend on \(\alpha_n\),
  and \(\gamma_n\) is a \(\widetilde O(1)\) sequence also independent of \(\alpha_n\).

  When we substitute \(\alpha_n\) by its expansion in terms of \(\beta_n\) into Equation~\eqref{eq:asympt_with_mixed_alpha_and_beta},
  we apply this time Lemma~\ref{lem:uniform_bound_saddlepoint} to achieve a uniform bound.
\end{proof}

\subsection{Step 2. Localizing the Mode}

Our goal is now to determine the range of values for $k$ for which $g_{n,k}$ is asymptotically relevant.
To identify this range,
we use the asymptotic upper bound for \(g_{n,k_n}\),
provided by
Corollary~\ref{cor:upperbound}
which is quite close to the true asymptotic value.\footnote{
This phenomenon is common when applying the saddle-point method.
To illustrate, consider
Cauchy's integral formula (as stated by Corollary~\ref{cor:cauchy}, and illustrated by Figure~\ref{fig:plot3d} Right):
\begin{equation*}
  g_{n,k_n} =   \int_{-1/2}^{1/2} r_n^{-n} G_{k_n}(r_n e^{2 \pi i \theta}) e^{-2 \pi i n \theta} d\theta.
\end{equation*}
In this expression,
 \(r_n^{-n} G_{k_n}(r_n)\) (which constitutes the upper bound) is the maximum of the integrand.
While it is true that the integral can be much smaller than the maximum of the function being integrated, it is not the case here.
By performing an asymptotic expansion around $\theta = 0$
(which we will elaborate on in more detail in  Subsection~\ref{ss:precise})
we will demonstrate that
\[g_{n,k_n} \approx \frac {r_n^{-n} G_{k_n} (r_n)} {2\pi} \int_{-\infty}^{+\infty} e^{-i \Delta_n \frac{\theta^2}2}  d\theta = \frac {r_n^{-n} G_{k_n} (r_n)} {\sqrt{2\pi \Delta_n}},
\]
where $\Delta_n = \widetilde O\left( \sqrt{n/2} \right)$.
}
Recall that this upper bound is of the form
\begin{equation*}
  \left( \frac{n} {2e}\right)^{\frac n 2} \exp\left(  \lambda_n \left( \beta_n \right)  \sqrt{n/2} + O\left( \ln(n)^2 \right)  \right)
\end{equation*}
where
\begin{equation*}
  \lambda_n(x) =
      \frac{1 + \ln(x)}{2x}\ln\left( \frac n 2 \right)+ \frac{(\ln x)^2}{x} + \left(\frac{1}{x} - x\right)\ln x + \frac 1 x + x.
\end{equation*}

In order to optimize the above quantity, we start by pinpointing the maximum of the $\lambda_n$ function.

\begin{lemma}[Study of the~$\lambda_n$ function]\label{lem:lambdan}
    For~$n$ large enough, the function~$\lambda_n(x)$ reaches a global maximum on the interval~$\left[n^{-1/4}, +\infty\right)$ at~$x=1$.

    Furthermore, near~$x=1$, we have
    \begin{multline*}
      \lambda_n(x) =
        \left(\frac 1 2 \ln\left( \frac n 2 \right) + 2 \right)
        - \left(\frac{1}{4}\ln\left( \frac n 2 \right) \right)(x-1)^2
        + \left(\frac 5 {12} \ln\left( \frac n 2 \right) - 1\right) (x-1)^3 \\
        + O\left(\ln\left( \frac n 2 \right) (x-1)^4\right),
    \end{multline*}
  \label{lem:study_lambda}
  where the \(O\)-term does not depend on \(n\).
\end{lemma}

 \begin{figure}[ht!]
 \begin{center}
 \includegraphics[scale=1]{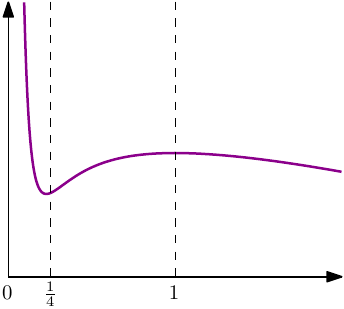}
 \end{center}
 \vspace{-7pt}
 \caption{ Plot of \(x \mapsto \lambda_{256}(x)\). Note that \(\lambda_{256}(x)\) is indeed maximal at $x=1$ over the interval
 \([ \protect\underbrace{1/4}_{=256^{-1/4}},  + \infty)\). }
 \label{fig:lambda_function}
 \end{figure}

The lemma focuses on \(x \geq n^{-1/4}\)
because we will evaluate \(\lambda_n\) at \(x = \beta_n\),
and since \(1/\beta_n = \widetilde O (1)\), we have \(\beta_n \gg n^{-1/4}\) asymptotically.

Moreover, we see at this point that our ansatz \(k_n = \frac n 2 + \alpha_n \sqrt{\frac n 2}\) with \(\alpha_n = \widetilde O(1)\)
was well-chosen:
the upper bound \(r_n^{-n} G_{k_n}(r_n)\) has a maximum at this regime,
which should corroborate that the mass of \(g_{n,k}\) concentrates here.

\begin{proof}[Proof of Lemma~\ref{lem:lambdan}] This proof relies on basic calculus techniques.
  Figure~\ref{fig:lambda_function} shows a plot of \(\lambda_n\) for $n = 256$.

  For all~$x > 0$, we have
    \begin{equation*}
      \lambda_n'(x)
      = - \frac{\ln (x)}{2x^2} \ln\left( \frac n 2 \right) - \ln (x) - \frac{\ln (x)^2}{x^2} + \frac{\ln (x)}{x^2}
      = -\frac{\ln (x) \ln\left(\frac x e e^{x^2}\sqrt{\frac n 2}\right)}{x^2}  .
    \end{equation*}
    When \(x\)  is greater than $n^{-1/4}$, the right-hand-side factor is ultimately positive.
    Indeed, using the fact that \( e^{x^2} > 1 \), we see that
    \begin{equation*}
      \ln\left(\frac x e e^{x^2}\sqrt{\frac n 2}\right) > \ln\left(\frac{n^{1/4}}{e \sqrt 2}\right) \geq 0 \quad\text{whenever } n \geq 4e^4 \approx 218.4.
    \end{equation*}
    It follows that,
    in this domain,~$\lambda_n'(x)$ as the same sign as~$-\ln (x)$ and thus reaches a maximum at~$x=1$.
    For this particular value of~$x$, we have
    \begin{equation*}
      \lambda_n(1) =  \frac 1 2  \ln \left( \frac n 2 \right)+ 2,
      \qquad
      \lambda_n''(1) = - \frac 1 2  \ln \left( \frac n 2 \right)
      \qquad\text{and}\qquad
      \lambda_n'''(1) = \frac 5 2 \ln \left( \frac n 2 \right)- 6.
    \end{equation*}
    The last statement of the lemma follows from Taylor's theorem.
  \end{proof}

\begin{remark}\label{rmk:heuristic}
  To build intuition, let us outline the implications of Lemma~\ref{lem:lambdan} using a heuristic approach.
  First, the maximum of~$g_{n,k}$ should be close to
  \begin{equation*}
    \left(\frac n {2e}\right)^{\frac n 2} \exp\left(\lambda_n(1)\sqrt{\frac n 2} + \mu_n(1)\right).
  \end{equation*}
  This is not necessarily the exact maximum since we have omitted the lower order terms in the exponential,
  but we will see that this incurs a negligible error.
  Moreover, this value is attained when~$\beta_n$ is close to~$1$.
  Using the second-order error term in Lemma~\ref{lem:lambdan}, we can see that in the vicinity of~$\beta_n=1$, we have
  \begin{equation*}
    r_n^{-n} G_{k_n}(r_n) \approx \left(\frac n {2e}\right)^{\frac n 2} \exp\left(\lambda_n(1)\sqrt{\frac n 2} - \frac{(\beta_n - 1)^2}{4}\ln \left( \frac n 2 \right)\sqrt{\frac n 2} + \mu_n(\beta_n)\right).
  \end{equation*}
  But, for all sequences $k_n$ such that \( \left| 1 - \beta_n \right| > 4 n^{-1/4} \), we have
    \begin{equation*}
      \exp \left( - \frac{(\beta_n - 1)^2}{4}\ln \left( \frac n 2 \right)\sqrt{\frac n 2} \right) <
      \exp \left( - \frac{16}{4 \sqrt n }\ln \left( \frac n 2 \right)\sqrt{\frac n 2} \right) = \left( \frac 2 n \right)^{ 2 {\sqrt 2}}  <
      \left( \frac 2 n \right)^2.
    \end{equation*}
  Consequently, \(  r_n^{-n} G_{k_n}(r_n) \) is suppressed by a factor of approximately \( n^2 \)
  when \(\left| 1 - \beta_n \right| > 4 n^{-1/4}\)
  compared to the critical value \( \beta_n = 1 \).
  It follows that contributions to the sum of all~$g_{n,k}$ should come from
  values of~$k$ corresponding to values of~$|1 - \beta_n|$ of order~$O(n^{-1/4})$
  --- this is a rough estimate.
  By Lemma~\ref{lem:estimation_saddlepoint},
  we expect these contributions to correspond to values of~$\alpha_n$
  that differ from~$\frac{\ln(n/2)}{4} - \frac 1 2$ by at most~$O(n^{-1/4} \ln n)$.
  \label{rem:where_mode}
\end{remark}

\subsection{Step 3. Precise Estimate Around the Mode}
\label{ss:precise}

By Remark~\ref{rem:where_mode},
it should be sufficient to compute
an asymptotic estimate of \(g_{n,k_n}\)
when \(k_n\) is near \(\frac{\ln(n/2)} 4\)
(we will show subsequently that contributions from the remaining values are negligible).
What follows is the resulting estimate.

\begin{proposition}\label{prop:gnk_asympt}
  Let~$\alpha_n$ be a \(\widetilde O(1)\) sequence such that $k_n = \frac{n}{2} + \alpha_n \sqrt{\frac{n}{2}}$ is an integer sequence.
  Let \(r_n\) and \(\beta_n\) be the two sequences defined by Lemma~\ref{lem:estimation_saddlepoint}.

  Suppose that for~$n \geq 4$, we have
  \begin{equation*}
    \frac { \ln(n/2) }  5
    \leq
    \alpha_n
    \leq
     \frac { \ln(n/2) }  3.
  \end{equation*}
  Then there exist
  an integer \(n_0\) and
  a \(\widetilde O(1)\) sequence \(\delta_n\) (both independent of \(\alpha_n\))  such that the estimate
  \begin{equation*}
    \left|
      g_{n,k_n}
      -
      \dfrac{r_n^{-n} G_{k_n}(r_n)}{\sqrt{2 \pi \Delta_n}}
    \right|
      <
      \frac{\delta_n} {\sqrt n} \times \dfrac{  r_n^{-n} G_{k_n}(r_n) } { \sqrt{2 \pi \Delta_n} }
  \end{equation*}
  holds for~$n \geq n_0$, where
  \begin{equation*}
    \Delta_n = \left(\frac{\ln\frac n 2}{2\beta_n} + \frac{\ln\beta_n}{\beta_n} - \frac{1}{\beta_n} + \beta_n\right)\sqrt{\frac n 2}.
  \end{equation*}
  \label{prop:asympt_cauchy_integral}
\end{proposition}

\paragraph*{Proof idea.}
The proof is a classic application of the Laplace method
to Cauchy's integral from Corollary~\ref{cor:cauchy}, which can be rewritten as
\begin{equation}
  g_{n,k_n} = \frac 1 {2 \pi} \int_{-\pi}^{\pi} \left( r_n e^{i \theta} \right)^{-n} G_{k_n} \left( r_n e^{i \theta}\right)  d\theta.
  \label{eq:cauchy_integral}
\end{equation}

\begin{figure}[ht!]
\begin{center}
\includegraphics[scale = 1.2]{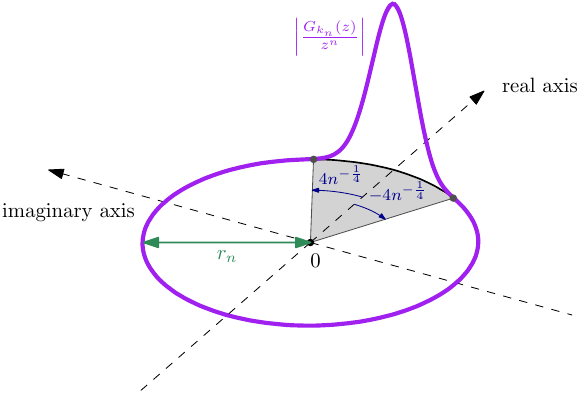}
\end{center}
\vspace{-7pt}
\caption{ Three-dimensional plot of
\(\left|  G_k(z) z^{-n}\right|\)
along the circle \( \left| z \right| = \sqrt{2/n}\).
Here the chosen parameters are $n = 4096$ and $k = 2111$.
The grey shaded region, defined by \(\left| \arg(z) \right| \leq 4 n^{-1/4}\),
shows where the integrand is well-approximated by a Gaussian and dominates.
The remaining region carries a negligible contribution.
}
\label{fig:laplace_method}
\end{figure}

The key observation is that the integrand exhibits a sharp peak at $\theta=0$,
which dominates the contribution.
More precisely, we decompose the integration contour into two regions (see Figure~\ref{fig:laplace_method}):
\begin{itemize}
  \item the neighborhood \([-4 n^{-1/4}, 4 n^{-1/4}]\) where the integrand behaves like a Gaussian,
  \item the complementary region
  \(\left\{ \theta \text{ such that } 4 n^{-1/4} \leq \left| \theta \right| \leq \pi \right\}\), which is negligible compared to the first region.
\end{itemize}
The threshold \(4n^{-1/4}\) is chosen to balance two requirements:
it must be small enough that the integrand remains close to \(\exp(-\Delta_n\theta^2/2)\),
yet large enough to render the tail contribution negligible.

The technical challenge is to establish bounds that are uniform in \(k\),
which is essential for summing the asymptotic estimates of \(g_{n,k}\) afterwards.

\subsubsection{Dominating Part of Cauchy's Integral}

We now prove the following lemma,
which quantifies the asymptotic contribution from the region \( \left| \theta \right| \leq 4 n^{-1/4}\), with explicit error bounds.

\begin{lemma}
  Under the hypotheses of Proposition~\ref{prop:asympt_cauchy_integral},
  there exist
  an integer \(n_0\) and
  a \(\widetilde O(1)\) sequence \(\delta_n\) such that for every $n \geq n_0$
  \begin{equation*}
    \left|
      \int_{-4 n^{-1/4}}^{4 n^{-1/4}}
      \left( r_n e^{i \theta} \right)^{-n} G_{k_n} \left( r_n e^{i \theta}\right)  d\theta
      -
      \sqrt{\dfrac {2 \pi} {\Delta_n}}
    r_n^{-n}G_{k_n}(r_n)
    \right|
      <
      \frac{\delta_n} {\sqrt n} \times \sqrt{\dfrac {2 \pi} {\Delta_n}}
    r_n^{-n}G_{k_n}(r_n).
  \end{equation*}
  \label{lem:dominating}
\end{lemma}

\paragraph*{Proof idea.}
The proof will be decomposed into several lemmas.
The core strategy is to perform an asymptotic expansion of the logarithm of the integrand near \(\theta=0\).
More precisely,
we set
\(L_n(z) := \ln \left(   {G_{k_n}(z)} \right) -  n \ln(z)
\)
so that the integral from Lemma~\ref{lem:dominating} becomes
\[ \int_{-4 n^{-1/4}}^{4 n^{-1/4}}  \exp\left( L_n \left( r_n e^{i \theta} \right) \right) d\theta .\]
By Taylor's theorem, we obtain the approximation
\(  L_n \left( r_n e^{i \theta} \right)  \approx  L_n \left( r_n \right) - \Delta_n \frac{ \theta^2 } 2\)
(the first-order term vanishes by our choice of \(r_n\)).
After appropriate rescaling, the integral becomes asymptotically equivalent to
\( \frac{ L_n(r_n) } {\sqrt {\Delta_n}} \int_{-\infty}^{+\infty} e ^{- \frac{ t^2 } 2 } dt \), which yields the claimed asymptotic.

The first lemma provides a convenient way to compute the Taylor coefficients of \( L_n \left( r_n e^{i \theta} \right)\).

\begin{lemma}
  Let \(r\) be a positive number, and \(f\) be an analytic function
  over the half-circle
  \[
  \left\{ r e^{i\theta} \
  \left|\ \theta \in \left[ - \frac \pi 2, \frac \pi 2 \right] \right. \right\}.\]
  Define the differential operator ${}^\bullet$
  by
  \(
    f^\bullet : z \mapsto z f'(z)  \).
 The Taylor expansion of  \(f(r e^{i \theta})\) around \(\theta = 0\)
    \begin{equation*}
    f(r e^{i \theta}) = f(r)
    + i  f^\bullet(r) \,\theta
    - f^{\bullet \bullet}(r) \frac{\theta^2}2
    - i  f^{\bullet\bullet\bullet}(r) \frac{\theta^3}6 + \frac 1 6 \int_0^\theta f^{\bullet\bullet\bullet\bullet}(r e^{it})(\theta - t)^3 dt
  \end{equation*}
  holds for every number $\theta$ with \( |\theta| \leq \frac \pi 2 \).
  \label{lem:taylor_theorem}
\end{lemma}

\begin{proof}
  We first establish by induction that for every \(k \geq 0\),
  \[ \frac{ \partial^k} {\partial \theta^k} \left( f \left(r e^{i\theta} \right) \right) = i^k f^{\bullet (k)} \left( r e^{i \theta} \right),
  \]
  where \({}^{\bullet (k)}\) denotes the operator \({} ^\bullet\)
  iterated \(k\) times.
  Indeed,
  the identity is obviously true for \(k=0\),
  and for \(k>0\),
  \begin{align*}
    \frac{ \partial^{k}} {\partial \theta^{k}} \left( f \left(r e^{i\theta} \right) \right)
    &= \frac\partial{\partial\theta}\left( i^{k-1} f^{\bullet (k-1)}(r e^{i\theta})\right) & \text{(by induction)} \\
    &= i^{k-1} \cdot i r e^{i\theta} \cdot {f^{\bullet (k-1)}}'(r e^{i\theta}) & \text{(by chain rule)} \\
    & = i^k f^{\bullet (k)} \left( r e^{i \theta} \right) & \text{(by definition of \(^\bullet\))}.
  \end{align*}
  Thus, evaluating at \(\theta = 0\) gives
  \[
  \frac{ \partial^k} {\partial \theta^k} \left( f \left(r e^{i\theta} \right) \right)\bigg|_{\theta = 0}
  = i^k f^{\bullet (k)} \left( r \right).
  \]
  The lemma now follows from Taylor's theorem with integral remainder applied to function \(\theta \mapsto f(r e^{i \theta})\).
\end{proof}

The next lemma shows that \(\beta_n\) is uniformly bounded in the regime of \(\alpha_n\) we consider.

\begin{lemma}
  There exists an integer \(n_0\) such that
  for every sequence~$\alpha_n$
  such that for every integer~\( n \geq 4 \),
  $k_n = \dfrac{n}{2} + \alpha_n \sqrt{\dfrac{n}{2}}$ is an integer and
  \(
    \dfrac { \ln(n/2) }  5    \leq
    \alpha_n    \leq
     \dfrac { \ln(n/2) }  3
  \),
  the sequence \(\beta_n\) defined by Lemma~\ref{lem:estimation_saddlepoint} satisfies
  \(
    \dfrac 1 2 \leq \beta_n \leq 2\) for every \( n \geq n_0 \).
    \label{lem:beta_between_half_and_2}
\end{lemma}
\begin{proof}
  Consider the sequences
   \( \lfloor \frac{n}{2} + \frac {\ln(n/2)} 5  \sqrt{\frac{n}{2}} \rfloor\)
   and
   \( \lceil \frac{n}{2} + \frac {\ln(n/2)} 3  \sqrt{\frac{n}{2}} \rceil\),
  which satisfy the hypotheses of Lemma~\ref{lem:estimation_saddlepoint}.
  Let \(\beta_n^+\) and \(\beta_n^-\) denote
  the corresponding \(\beta_n\) sequences.
  By Lemma~\ref{lem:decreasing_beta}, we have \(
    \beta_n^- \leq \beta_n \leq \beta_n^+\).

  By Corollary~\ref{cor:beta_in_terms_of_alpha},
  \(\beta_n^-\) converges to \(3/4\) and
  \(\beta_n^+\) converges to \(5/4\).
  Therefore there exists an integer \(n_0\) such that for every \(n \geq n_0\),
  we have  \(\frac 1 2 < \beta_n^-\) and \(\beta_n^+ < 2 \).
  This yields the claimed bounds.
\end{proof}

The following lemma provides uniform bounds on the derivatives of $L_n$ up to order four.

\begin{lemma}
Let \(L_n(z) = \ln \left(  G_{k_n}(z) \right) - n \ln(z) \).
There exist an integer \(n_0\)
and a \(\widetilde O(1)\) sequence \(\delta_{n}\)
such that for all \(n \geq n_0\),
for all sub-sequence \(\alpha_n\) satisfying the hypotheses of Proposition~\ref{prop:asympt_cauchy_integral},
and for all \(\theta \in \left[ - \frac \pi 2 , \frac \pi 2 \right]  \)
the following inequalities hold:
\begin{equation*}
   \left|
    L_n^{\bullet \bullet}(r_n)
    + \Delta_n   
   \right|
    < \delta_{n} ,
\end{equation*}
\begin{equation*}
   \left|  L_n^{\bullet \bullet \bullet}(r_n) \right|
   < \delta_{n} \sqrt{ n},
\end{equation*}
\begin{equation*}
   \left|  L_n^{\bullet \bullet \bullet \bullet}\left( r_n e ^{i \theta} \right) \right|
   < \delta_{n} \sqrt{ n },
\end{equation*}
where ${}^\bullet$ denotes the differential operator
  \(   z \frac{\partial}{\partial z}  \) and 
  \begin{equation*}
    \Delta_n := \left(\frac{\ln\frac n 2}{2\beta_n} + \frac{\ln\beta_n}{\beta_n} - \frac{1}{\beta_n} + \beta_n\right)\sqrt{\frac n 2}.
  \end{equation*}
  \label{lem:Lnbullet}
\end{lemma}
\begin{proof}
By Proposition~\ref{prop:Gkz}, we have for \(n \geq 4\),
\begin{equation*}
  L_n(z) =
  (2k_n-n) \ln(z) - k_n \ln(1-z)
  + \ln \Gamma \left( k_n + \frac {1-z} z \right) - \ln \Gamma \left(  \frac {1-z} z \right).
\end{equation*}
Direct differentiation gives
\begin{equation*}
  L_n^\bullet(z) =
  - \frac {D\digamma_0(z)} z  + k_n - n + \frac {k_n} {1-z},
\end{equation*}
where we introduce the notation
\begin{equation*}
  D\digamma_m(z) := \digamma^{(m)}\left(  k_n + \frac{1-z} z\right) - \digamma^{(m)}\left( \frac{1-z} z\right),
\end{equation*}
with \(\digamma^{(m)}\)  the $m$th polygamma function,
i.e. the \((m+1)\)th derivative of \(\ln \Gamma\).
This notation is convenient because it allows us to use the identity
\((D\digamma_m)'(z) = - D\digamma_{m+1}(z)/z^2  \).

Thus we get the expressions
\begin{equation}
    L_n^{\bullet \bullet}(z) =
  \frac {D\digamma_0(z)} z + \frac {D\digamma_1(z)} {z^2}  + \frac { z } {(1-z)^2} k_n,
  \label{eq:bullet2}
\end{equation}
\begin{equation}
    L_n^{\bullet \bullet \bullet}(z) =
  - \frac {D\digamma_0(z)} z - 3 \, \frac {D\digamma_1(z)} {z^2} - \frac {D\digamma_2(z)} {z^3}  +  \frac { z  (1 + z)} {(1-z)^3} k_n,
  \label{eq:bullet3}
\end{equation}
\begin{equation}
    L_n^{\bullet \bullet \bullet \bullet}(z) =
  \frac {D\digamma_0(z)} z + 7 \,\frac {D\digamma_1(z)} {z^2} + 6 \, \frac {D\digamma_2(z)} {z^3} + \frac {D\digamma_3(z)} {z^4} +  \frac {z  (z^2 + 4z + 1)} {(1-z)^4} k_n.
  \label{eq:bullet4}
\end{equation}
We now use the asymptotic estimates~\cite[6.3.18, 6.4.11]{abramowitz-stegun}
\begin{equation}
  \digamma (x) \underset{x \rightarrow +\infty}{\sim} \ln(x) + O\left(  \frac 1 x \right) \text{ and }  \digamma^{(m)} (x) \underset{x \rightarrow +\infty}{\sim} (-1)^{m-1} \frac{(m-1)!}{x^m} + O\left(  \frac 1 {x^{m+1}} \right) \text{for }m \geq 1
  \label{eq:asympt_polygamma}
\end{equation}
valid for all \(x\) with \(\left| \arg(x) \right| \leq \frac {3\pi} 4 \).
We then deduce that
\begin{equation*}
  \frac {D\digamma_0(r_n)} {r_n} \underset{n \rightarrow +\infty}{\sim}  \frac{ \ln(n/2) + 2 \ln(\beta_n) } { 2 \beta_n } \sqrt{\frac n 2} + \widetilde O(1)
\end{equation*}
and for $m \geq 1$,
\begin{equation*}
  \frac {D\digamma_m(r_n)} {r_n^{\ m+1}} \underset{n \rightarrow +\infty}{\sim}  (-1)^m\frac{ (m-1)! } { \beta_n } \sqrt{\frac n 2} + \widetilde O(1),
\end{equation*}
where the \(\widetilde O(1)\) terms are independent of \(\alpha_n\)
(similarly as we did in the proof of Lemma~\ref{lem:uniform_bound_saddlepoint}).
Likewise, we have the uniform bounds
\begin{equation*}
   \frac { r_n } {(1-r_n)^2} k_n
   \underset{n \rightarrow +\infty}{\sim}
    \beta_n \sqrt{n/2} + \widetilde O(1)
    \quad \text{and} \quad
    \frac { r_n  (1 + r_n)} {(1-r_n)^3} k_n
    \underset{n \rightarrow +\infty}{\sim}
    \widetilde O\left(  \sqrt{n} \right).
\end{equation*}
Plugging these asymptotic estimates into Equations~\eqref{eq:bullet2} and \eqref{eq:bullet3} yields the first two inequalities of the statement.

\begin{figure}[ht!]
\begin{center}
\includegraphics[width=0.5\textwidth]{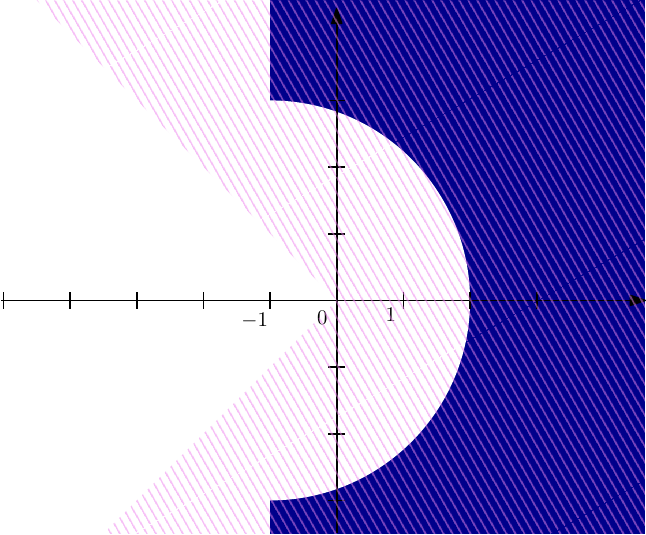}
\end{center}
\vspace{-7pt}
\caption{ The plain blue region depicts the domain \(U\),
where the functions \( \dfrac{ D\digamma_m(z) } { z^{m}  \ln(n) } \) are bounded.
The hatched purple region corresponds to the domain  \(\left| \arg(x) \right| \leq \frac {3\pi} 4 \) where Equation~\eqref{eq:asympt_polygamma} holds.
}
\label{fig:domain_psi}
\end{figure}

To establish the third inequality, we consider
 \(
  U := \left\{x \ | \  |x+1| \geq 3, |\arg (x+1)| \leq \frac \pi 2 \right\}
 \)
 (see Figure~\ref{fig:domain_psi}).
For all \(\theta \in \left[ - \frac \pi 2, \frac \pi 2 \right]\),
and for all \(n\) sufficient large (with the threshold independent of \(\alpha_n\)),
both
\(k_n + \frac{1 - r_n e ^{i \theta}} { r_n e ^{i \theta}} \)
and \(\frac{1 - r_n e ^{i \theta} } { r_n e ^{i \theta}} \) belong to \(U\) since they are of the form \(k + \frac {e^{-i \theta}} {r_n }\) where $k \geq -1$.
But \( \psi(x) / \ln(x) \) and \( x^m \times \psi^{(m)}(x) \) are bounded over
$U$:
this follows from the above asymptotic expansions of polygamma functions
(which are uniform in argument) and their continuity on \(U\).
Consequently, for any \( m \geq 0\),
\( \dfrac{ D\digamma_m(z) } { z^{m}  \ln(n) } \) is bounded  over the half-circle $|z| = r_n$ and $|\arg z| \leq \frac \pi 2 $, uniformly in \(n\).

For the rational term in \eqref{eq:bullet4},
note that
\(\frac{z^2 + 4z + 1}{(1-z)^4}\) is continuous with value \(1\) at \(z=0\).
Thus for sufficiently large \(n\),
we have \(\left| \frac {z^2 + 4z + 1} {(1-z)^4}\right|  < 2\)
for all \(\left| z \right| = r_n\).
This gives
\begin{equation*}
  \left|\frac{z^2(z^2 + 4z + 1)}{(1-z)^4} k_n\right| < 2 r_n^2 k_n = \frac{k_n}{n} \beta_n^2 \leq \beta_n^2 < 4
\end{equation*}
(we used Lemma~\ref{lem:beta_between_half_and_2} for the last inequality).

Combining these bounds in Equation~\eqref{eq:bullet4} shows that \(  \dfrac { L_n^{\bullet \bullet \bullet \bullet} \left(  r_n e^{i \theta} \right)  } {r_n \ln(n)} \) is uniformly bounded,
yielding the third inequality.
\end{proof}

We are now equipped to estimate the central part of Cauchy's integral with Taylor's Theorem.

\begin{proof}[Proof of Lemma~\ref{lem:dominating}]
  We apply Lemma~\ref{lem:taylor_theorem}
  to
  \(L_n(z) = \ln \left(  G_{k_n}(z) \right) - n \ln(z) \) with \(r = r_n\).
  Since \(     L_n^\bullet ( r_n ) = 0    \) by definition
  (equivalently,  \(r_n G_{k_n}'(r_n) = n G_{k_n}(r_n)\)),
  we obtain the Taylor expansion

  \begin{equation*}
    L_n(r_n e^{i \theta}) = L_n(r_n)
    - L_n^{\bullet \bullet}(r_n) \frac{\theta^2}2
    - i  L_n^{\bullet\bullet\bullet}(r_n) \frac{\theta^3}6 + \frac 1 6 \int_0^\theta L_n^{\bullet\bullet\bullet\bullet}(r_n e^{it})(\theta - t)^3 dt
  \end{equation*}
  for all \(\theta \in \left[ - 4 n^{1/4},  4 n^{1/4} \right]\).
  By Lemma~\ref{lem:Lnbullet},
  there exists a \(\widetilde O(1)\) sequence \(\delta_n\)
  such that for all $n$ sufficiently large, and for all \(\theta \in \left[ - 4 n^{1/4},  4 n^{1/4} \right]\),
  \begin{equation*}
    \left| \int_0^\theta L_n^{\bullet\bullet\bullet\bullet}(r_n e^{it})(\theta - t)^3 dt \right|
    < \delta_n \sqrt{n} \,\theta^4
    \leq 4^4 \delta_n \frac 1 {\sqrt{n}}.
  \end{equation*}
  Similarly, by the same lemma, we have for all \(\theta \in \left[ - 4 n^{1/4},  4 n^{1/4} \right]\),
  \begin{equation*}
    -L_n^{\bullet \bullet}(r_n) \frac{\theta^2}2 = - \Delta_n \frac {\theta^2} 2  + \widetilde{O}\left( \frac 1 {\sqrt n} \right)
    \quad \text{and} \quad
    - i L_n^{\bullet \bullet \bullet}(r_n) \frac{\theta^3}6 = \widetilde{O}\left( n^{-\frac 1 4} \right),
  \end{equation*}
  where the \(\widetilde{O}\) terms are independent of \(\alpha_n\) and \(\theta\).
  We get by combining these estimates
  \begin{equation}
   \frac{  G_{k_n} \left( r_n e^{i \theta}\right) }
   {\left( r_n e^{i \theta} \right)^{n} }
    = e^{ L_n(r_n e^{i \theta}) }
    = \exp\left( L_n(r_n)
    - \Delta_n \frac{\theta^2}2 \right) 
    \exp \left(  - i L_n^{\bullet \bullet \bullet}(r_n) \frac{\theta^3}6 +  \widetilde{O}\left( \frac 1 {\sqrt n} \right) \right).
    \label{eq:expansion_integrant}
  \end{equation}
  Since \( - i L_n^{\bullet \bullet \bullet}(r_n) \frac{\theta^3}6\) tends to zero, 
  we have the approximation
  \begin{equation*}
     \exp \left(  - i L_n^{\bullet \bullet \bullet}(r_n) \frac{\theta^3}6 +  \widetilde{O}\left( \frac 1 {\sqrt n} \right) \right) 
     =
      1 - i L_n^{\bullet \bullet \bullet}(r_n) \frac{\theta^3}6 +  \widetilde{O}\left( \frac 1 {\sqrt n} \right)
  \end{equation*}
  still uniformly in \(\alpha_n\) and \(\theta\).
  Using this estimate and 
  \(\exp\left( L_n\left( r_n \right) \right)  = G_{k_n}(r_n)r_n^{\ -n} \) 
  in \eqref{eq:expansion_integrant}, we obtain after integration
  \begin{multline*}
    \int_{-4 n^{-1/4}}^{4 n^{-1/4}}
      \frac{  G_{k_n} \left( r_n e^{i \theta}\right) }
   {\left( r_n e^{i \theta} \right)^{n} }  d\theta
   = \frac{
      G_{k_n} \left( r_n \right) 
  }
   {r_n^{\ n}}
   \left( 1 +  \widetilde{O}\left( \frac 1 {\sqrt n} \right) \right)
   \int_{-4 n^{-1/4}}^{4 n^{-1/4}} e^{-\Delta_n \frac{\theta^2}2} d\theta 
   \\
   - \frac{
      G_{k_n} \left( r_n \right) i L_n^{\bullet \bullet \bullet}(r_n)
  }
   {r_n^{\ n}}
   \int_{-4 n^{-1/4}}^{4 n^{-1/4}} 
    \frac{\theta^3}6 e^{-\Delta_n \frac{\theta^2}2} d\theta.
  \end{multline*}
  The second integral vanishes because the integrand \( \theta \mapsto \frac{\theta^3}6 e^{-\Delta_n \frac{\theta^2}2} \) is an odd function integrated over the symmetric interval  \(\left[ -4n^{-1/4}, 4n^{-1/4} \right]\).
  As for the first integral, it reduces by a variable change to
  \begin{equation*}
    \int_{-4  n^{-\frac 1 4}}^{4  n^{-\frac 1 4}} e^{-\Delta_n \frac{\theta^2}2} d\theta 
    = \sqrt{\frac 2 {\Delta_n}}
     \int_{-2 \sqrt 2  n^{-\frac 1 4} \sqrt{\Delta_n}}^{2 \sqrt 2  n^{-\frac 1 4} \sqrt{\Delta_n}} e^{-x^2} dx 
     = \sqrt{\frac {2\pi} {\Delta_n}} \left( 1 - \text{erfc}\left( 2 \sqrt 2  n^{-\frac 1 4} \sqrt{\Delta_n} \right) 
       \right),
  \end{equation*}
  where \(\text{erfc}\)
  is the complementary error function.
  Observe that \( n^{-1/4} \sqrt{\Delta_n} \) is of the same order as \(\sqrt{\ln (n)}\) 
  (remember that \(\beta_n\) ultimately lies between \(1/2\) and \(2\) by Lemma~\ref{lem:beta_between_half_and_2}), 
  hence diverges to \(+\infty\).
  But \(\text{erfc}\) admits the following expansion~\cite[7.1.23]{abramowitz-stegun}
  \begin{equation*}
    \text{erfc}(x)  \underset{x \rightarrow +\infty}{\sim} \frac{ 1 + O(x^2) } {x e^{x^2} \sqrt{\pi}}.
  \end{equation*}
  This yields the uniform estimate
    \begin{equation*}
    \text{erfc}\left( 2 \sqrt 2  n^{-\frac 1 4} \sqrt{\Delta_n} \right)
    = \widetilde O\left( n^{-\frac{2\sqrt2}{\beta_n}}\right)
    = \widetilde O\left( n^{-\sqrt 2} \right).
  \end{equation*}
  Substituting back into our integral formula, we get
  \begin{equation*}
    \int_{-4 n^{-1/4}}^{4 n^{-1/4}}
      \frac{  G_{k_n} \left( r_n e^{i \theta}\right) }
   {\left( r_n e^{i \theta} \right)^{n} }  d\theta
   = 
   \frac{
      G_{k_n} \left( r_n \right) 
  }
   {r_n^{\ n}}
   \sqrt{\frac {2\pi} {\Delta_n}}
   \left( 1 +  \widetilde{O}\left( \frac 1 {\sqrt n} \right) \right),
  \end{equation*}
  as claimed.
\end{proof}

\subsubsection{Non-dominating Part of Cauchy's Integral}

We now establish a uniform bound on the region
\( \left| \theta \right| \geq 4 n^{-1/4}\),
complementing Lemma~\ref{lem:dominating}.
%he error term is of order \(n^{-2\sqrt{2}-1/2}\),
%which leaves comfortable margin for the overall asymptotic analysis.

\begin{lemma}
  Under the hypotheses of Proposition~\ref{prop:asympt_cauchy_integral},
  there exist
  an integer \(n_0\) and
  a \(\widetilde O(1)\) sequence \(\delta_n\) such that for every $n \geq n_0$,
  \begin{equation*}
    \left|
      \int_{ 4 n^{-1/4} \leq \left| \theta \right| \leq \pi }
      \left( r_n e^{i \theta} \right)^{-n} G_{k_n} \left( r_n e^{i \theta}\right)  d\theta
    \right|
      <
      \frac{\delta_n} {n^{\sqrt{2}}} \times  r_n^{-n} G_{k_n}(r_n).
  \end{equation*}
  \label{lem:negligible_part_cauchy}
\end{lemma}

The proof strategy is to simply bound the integrand by its maximum value. 
The following lemma establishes that the integrand is unimodal for \(\theta \in [-\pi,\pi]\) with a peak at \(0\), 
which implies that the maximum on the tail region is attained at the endpoints \(\theta = \pm 4 n^{-1/4}\).

\begin{lemma}
  For any radius~$0 < r < 1$ and for any~$k \geq 2$, the function~$\theta \mapsto|G_k(r e^{i\theta})|$ is strictly increasing on~$[-\pi, 0]$ and strictly decreasing on~$[0, \pi]$.
  \label{lem:unimodal}
\end{lemma}
\begin{proof}
  For any~$j \geq 1$, we have
  \begin{equation*}
    \left|1 + \frac{j r e^{i\theta}}{1 - r e^{i\theta}}\right|^2
    = \frac{|1 + (j-1)r e^{i\theta}|^2}{|1 - r e^{i\theta}|^2}
    = \frac{1 + (j-1)^2 r^2 + 2(j-1)r \cos(\theta)}{1 + r^2 - 2 r \cos(\theta)}
  \end{equation*}
  which is an even function of~$\theta$ and decreasing on~$[0, \pi]$
  as the ratio between a decreasing and an increasing function.
  Taking square roots preserves this monotonicity.
  We multiply those quantities for~$j \geq 1$, and deduce that
  \(
    |G_k(r e^{i\theta})| = r^k \prod_{j=1}^{k-1} \left|1 + \frac{j r e^{i\theta}}{1 - r e^{i\theta}}\right|
  \)
     is unimodal in~$\theta$.
\end{proof}

\begin{proof}[Proof of Lemma~\ref{lem:negligible_part_cauchy}]
  We start by giving an asymptotic estimate of
  \(  \left| G_{k_n} \left(  r_n e^{i 4 n^{-1/4}} \right) \right|\)
   relatively to \(  \left| G_{k_n} \left(  r_n  \right) \right|\).
To do so, we apply the same analysis as in the proof of Proposition~\ref{prop:evalr}, which yields\footnote{ this is the same calculation that led to Equation~\eqref{eq:asympt_with_mixed_alpha_and_beta}}
\begin{equation}
    \frac {G_{k_n} \left( r_n e^{i 4 n^{-1/4}}  \right)}
    {\left( r_n e^{i 4 n^{-1/4}}  \right)^{n}}
      =
      \left( \frac n {2e} \right)^{\frac n 2}
    \exp\left(
        \begin{aligned}
           \left(
            \frac{\ln\frac n 2 }{2 \widehat \beta_n} + 2\alpha_n\ln\widehat\beta_n + \widehat \beta_n + \frac{\ln \widehat \beta_n}{\widehat \beta_n} + \frac{1}{\widehat \beta_n}\right)
            \sqrt{\frac n 2}
          + \alpha_n\widehat \beta_n  \\
          + \frac{\widehat \beta_n^2}{2}
          + \frac{(\alpha_n + \frac{1}{\widehat \beta_n})^2}{2}
          - \frac{3}{4}\ln \frac n 2 - \frac{3}{2}\ln\widehat \beta_n
          +\widetilde O \left( \frac 1 {\sqrt n}\right)
        \end{aligned}
        \right)
        \label{eq:with_i_theta}
  \end{equation}
  where \(\widehat \beta_n := \beta_n e^{i 4 n^{-1/4}}\) and the \(\widetilde O \left( \frac 1 {\sqrt n}\right)\) term is independent of \(\alpha_n\).
  We now divide Equation~\eqref{eq:with_i_theta} by Equation ~\eqref{eq:asympt_with_mixed_alpha_and_beta}. 
  To simplify this quotient, we substitute \(\widehat \beta_n\) by its Taylor expansion
    \begin{equation*}
    \widehat \beta_n = 
    \beta_n \left( 1 + i 4 n^{-1/4} - 8 n^{-1/2} + O\left( n^{-3/4} \right)  \right) 
  \end{equation*}
  and form the differences between similar terms in the exponent.
  We observe that:
  \begin{itemize}
    \item the difference of terms scaling in \(O\left( \sqrt{n} \right)\) or less gives purely imaginary contributions, or \(O(1)\) terms;
    \item   the real part of \( 
  \frac{\ln (n/2)}{2 \widehat \beta_n} \sqrt{n/2}  - \frac{\ln (n/2)}{2 \beta_n} \sqrt{n/2}
  \) is equivalent to \( - \frac{2 \sqrt 2}{\beta_n} \ln \left( \frac n 2  \right)
  + O\left( \ln(n) n^{-1/4} \right)\);
  \item the difference \(2 \alpha_n \ln \widehat \beta_n \sqrt{n/2} - 2 \alpha_n \ln \beta_n  \sqrt{n/2} \) is equal to \(8 i \alpha_n  n^{-1/4} \sqrt{n/2} \) (it is purely imaginary). 
  \end{itemize} 
  After  these simplifications, we arrive at
  \begin{equation*}
    \frac {G_{k_n} \left( r_n e^{i 4 n^{-1/4}}  \right)}
    {G_{k_n} \left( r_n  \right) e^{i 4 n^{3/4}} }
      =  
      \exp 
      \left( i \times \text{real\_sequence}_n 
      - \frac{2 \sqrt 2}{\beta_n} \ln \left( \frac n 2  \right)+ O(1)
       \right)
  \end{equation*}
  where \(\text{real\_sequence}_n \) is some sequence of real numbers,
  and the \(O(1)\) term is uniform.
  By taking the modulus, 
  we deduce that there exists a bounded sequence $\delta_n$, independent of \(\alpha_n\), 
  such that for all sufficiently large \(n\),
  \begin{equation}
    \left| 
\frac {G_{k_n} \left( r_n e^{i 4 n^{-1/4}}  \right)}
    {G_{k_n} \left( r_n  \right) }
     \right|
     \leq  \delta_n  \exp\left(  - \frac{2 \sqrt 2}{\beta_n} \ln \left( \frac n 2   \right) \right) 
     \leq \delta_n \left( \frac n 2 \right)^{-\sqrt{2}}.
     \label{eq:bound_ratio}
  \end{equation}
  (For the last inequality, we have used Lemma~\ref{lem:beta_between_half_and_2} 
  which states that ultimately \(\beta_n \leq 2\).)
 We now bound
  the integral
  \begin{equation*}
     \left|
      \int_{ 4 n^{-1/4} \leq \left| \theta \right| \leq \pi }
      \frac{G_{k_n} \left( r_n e^{i \theta}\right)}
      {\left( r_n e^{i \theta} \right)^{n} }  d\theta
    \right|
    < 2 \pi r_n ^{\ -n}
    \max_{ 4 n^{-1/4} \leq \left| \theta \right| \leq \pi }
    \left| 
       G_{k_n} \left( r_n e^{i \theta}\right)
    \right|.
  \end{equation*}
  By Lemma~\ref{lem:unimodal},
  this maximum is reached either at \(\theta = -4 n^{-1/4}\), or at \(\theta =  4 n^{-1/4}\). 
  But since \(G_k(\overline z) = \overline{G_k(z)}\),
  the values at these points are equal:
  \(  \left| G_{k_n} \left( r_n e^{i 4 n^{-1/4}}\right) \right| =  \left| G_{k_n} \left( r_n e^{-i 4 n^{-1/4}}\right) \right|\).
  So, by Inequality~\eqref{eq:bound_ratio}, for all \(n\) sufficiently large,
  \begin{equation*}
    \left|
      \int_{ 4 n^{-1/4} \leq \left| \theta \right| \leq \pi }
      \frac{G_{k_n} \left( r_n e^{i \theta}\right)}
      {\left( r_n e^{i \theta} \right)^{n} }  d\theta
    \right|
    < 2 \pi r_n ^{\ -n} \left| 
       G_{k_n} \left( r_n e^{i 4 n^{-1/4}}\right)
    \right|
    \leq 2 \pi \delta_n \left( \frac n 2 \right)^{- \sqrt 2} \frac{G_{k_n}\left( r_n \right)} {r_n^{\ n}}.
  \end{equation*}
 Up to a constant factor that could be integrated to the definition of \(\delta_n\),
 we recover the inequality of the lemma.
\end{proof}

The proof of the main result of this subsection is now straightforward.

\begin{proof}[Proof of Proposition~\ref{prop:gnk_asympt}]
  By Corollary~\ref{cor:cauchy},
  using the contour \(z = r_n\,e^{i \theta}\) for \(\theta \in [-\pi,\pi ]\),
  we see that
  \begin{equation*}
    g_{n,k_n} = \frac 1 {2 \pi} \int_{-\pi}^{\pi} \left( r_n e^{i \theta} \right)^{-n} G_{k_n} \left( r_n e^{i \theta}\right)  d\theta.
  \end{equation*}
  We split the integration domain at \(\theta = \pm 4 n^{-1/4}\) 
  and
  apply Lemmas~\ref{lem:dominating} and~\ref{lem:negligible_part_cauchy}
  to infer
  \begin{equation*}
    \left| g_{n,k_n} - \dfrac{r_n^{-n} G_{k_n}(r_n)}{\sqrt{2 \pi \Delta_n}} \right|
    <
     \frac{\delta_n} {\sqrt n} \times \dfrac{r_n^{-n} G_{k_n}(r_n)}{\sqrt{2 \pi \Delta_n}}
    +
    \frac{\delta_n} {n^{\sqrt{2}}} \times  \dfrac{r_n^{-n} G_{k_n}(r_n)}{2 \pi}.
  \end{equation*}
  Since \begin{equation*}
    \sqrt{2\pi \Delta_n}
     = O \left( \sqrt{ \ln(n)} \, n^{-1/4} \right)
     = o \left( \frac{n^{\sqrt 2}}{\sqrt n} \right),
  \end{equation*}
  we deduce that \( \delta_n +  \frac{\sqrt{2 \pi \Delta_n} \sqrt n} {n^{\sqrt{2}}} \delta_n \) 
  is a \(\widetilde{O} (1) \) sequence satisfying the inequality of the lemma.
\end{proof}

\subsection{Step 4. Summation Around the Mode}

By the heuristic discussion of Remark~\ref{rmk:heuristic}, we expect the range of values of~$k$ such that~$g_{n,k}$ contributes significantly to the sum~$g_n = \sum_{k}g_{n,k}$ to correspond to values of~$\alpha_n$ that differ from~$\frac 1 4 \ln\left( \frac n 2 \right) - \frac 1 2$ by at most~$O(n^{-1/4} \ln n)$.
The purpose of this section is to estimate the sum of these contributions.
The final step of the proof will then be to show that the rest of the sum is indeed negligible.

\begin{proposition}\label{prop:central_sum}
  Let~$k_n^\star := \frac n 2 + \left( \frac 1 4 \ln\left( \frac n 2 \right) - \frac 1 2 \right) \sqrt{\frac n 2}$
  denote the asymptotic peak location of~\( \left(g_{n,k}\right)_k\).
  The following estimate holds:
  \begin{multline*}
    \sum_{\substack{k  \text{ such that}\\ {|k - k_n^\star| \leq (\frac n 2)^{\frac 1 4} \ln\frac n 2}}} g_{n,k} \underset{n \to \infty}\sim \\
    \frac{e^{1 / 8}}2
    \left(\frac n 2\right)^{- 3 / 8}
    \left(\frac n {2e}\right)^{\frac n 2}
    \exp\left(\frac 1 2 \ln\left( \frac n 2 \right) \sqrt{\frac n 2} + \sqrt{2n} + \frac{\ln\left( \frac n 2 \right)^2}{32}\right)
    \left( 1 + \widetilde O\left( n^{-1/2} \right) \right).
  \end{multline*}
\end{proposition}

\paragraph*{Proof idea.}
The proof consists of two main steps:
\begin{itemize}
  \item \textbf{Reparametrization.} 
  Proposition~\ref{prop:asympt_cauchy_integral},
  coupled with Proposition~\ref{prop:asympt_saddle_bound},
  gives a uniform asymptotic estimate of \(g_{n,k_n}\) 
  as a function of the saddle-point parameter \(\beta_n\).
  Since \(\beta_n\) is implicitly defined by \(k_n\) (via Lemma~\ref{lem:estimation_saddlepoint}),
  we rewrite the entire expression in terms of the deviation
  \(k_n - k_n^\star\).
  \item \textbf{Summation.} 
  With the coefficients now expressed as a function
  of the displacement from \(k_n^\star\),
  we sum over the window \(|k - k_n^\star| \leq \left( \frac n 2 \right)^{\frac 1 4} \ln\left( \frac n 2 \right)\)
  using the Euler-Maclaurin formula,
  which translates the sum into a Gaussian-type integral.
\end{itemize}

\subsubsection{Eliminating the Parameter \texorpdfstring{\(\boldsymbol{\beta_n}\)}{beta\_n} From the Asymptotic Formula}

Our heuristic analysis indicates that
\(\beta_n\) lies close to~\(1\) throughout the relevant range of \(\alpha_n\).
The following lemma quantifies this proximity by expressing~\(1 - \beta_n\)
as a \emph{uniform} asymptotic expansion in the parameter \(\alpha_n - \alpha_n^*\),
where \(\alpha_n^\star\) parametrizes the peak location \(k_n^\star\).

\begin{lemma}\label{lem:beta_expansion*}
  Let~$\alpha_n^\star = \frac 1 4 \ln\left( \frac n 2 \right) - \frac 1 2$ 
  and let~$\alpha_n$ be any sequence 
  such that~$k_n = \frac n 2 + \alpha_n \sqrt{\frac n 2}$ is an integer 
  and~$|\alpha_n - \alpha_n^\star| \leq \left(\frac n 2\right)^{-1/4} \ln\left(  \frac n 2 \right)$.

  If~$\beta_n$ is as defined in Lemma~\ref{lem:estimation_saddlepoint},
  then \(\beta_n\) tends uniformly to \(1\).
  More precisely, we have the approximations
  \begin{equation}
    1 - \beta_n = 
  \frac{4}{\ln\left( \frac n 2 \right)}(\alpha_n - \alpha_n^\star)
  + O \left( \frac 1 {\sqrt n} \right),
  \label{eq:estimate_epsilon_n}
\end{equation}
  \begin{equation*}
    (1 - \beta_n)^{\,2} = 
  \frac{16 (\alpha_n - \alpha_n^\star)^2}{\ln\left( \frac n 2 \right)^2} 
  + 128 \left( 3 - \ln\left( \frac n 2 \right) \right)  \frac{(\alpha_n - \alpha_n^\star)^3}{\ln\left( \frac n 2 \right)^4} 
  -24 \frac{\alpha_n - \alpha_n^\star}{\ln\left( \frac n 2 \right)^2} \sqrt{\frac 2 n}
  + O \left( \frac 1 n \right),
  \label{eq:epsilon_n^2}
\end{equation*}
  \begin{equation*}
    (1 - \beta_n)^{\,3} = 
  \frac{64}{\ln\left( \frac n 2 \right)^3}(\alpha_n - \alpha_n^\star)^3
  + O \left( \frac 1 n \right),
  \label{eq:epsilon_n^3}
\end{equation*}
where every O-term is independent of \(\left( \alpha_n \right)\).
\label{lem:beta_in_terms_of_alpha}
\end{lemma}

% J'essaie ma preuve
\begin{proof}
  Define 
  \begin{equation*}
    \epsilon_n := 1 - \beta_n.
  \end{equation*}

  \paragraph*{\(\boldsymbol{\epsilon_n}\) tends uniformly to \(\boldsymbol 0\).}
  The sequences \( \lfloor \frac n 2 + \alpha_n^\star \sqrt{\frac n 2} - \left(\frac n 2\right)^{-1/4} \ln\left( \frac n 2 \right) \rfloor \)
  and \( \lceil \frac n 2 + \alpha_n^\star \sqrt{\frac n 2} + \left(\frac n 2\right)^{-1/4} \ln\left( \frac n 2 \right) \rceil\)
  satisfy the requirements of Lemma~\ref{lem:estimation_saddlepoint},
  so we may define \(\beta_n^+\) and \(\beta_n^-\) as the respective saddle-point parameters.
  By Lemma~\ref{lem:decreasing_beta},
  we have \(\beta_n^- \leq \beta_n \leq \beta_n^+\), which implies
  \begin{equation*}
    \left| \epsilon_n \right| \leq \mu_n
  \end{equation*}
  where \(\mu_n := \max \left( \left| 1 - \beta_n^- \right|, \left| 1 - \beta_n^+ \right| \right)\).
  By Corollary~\ref{cor:beta_in_terms_of_alpha}, both \(\beta_n^-\) and \(\beta_n^+\) tend to \(1\). 
  Consequently \(\left| \epsilon_n \right|\) is bounded by a sequence \(\mu_n\) converging to \(0\) independent of \(\alpha_n\).

  \paragraph*{\(\boldsymbol{\epsilon_n = O\left( n^{-1/4} \right) }\)  uniformly.}
  For \(n\) sufficiently large,
  \(|\alpha_n| \leq \alpha_n^\star + \left(\frac n 2\right)^{-1/4} \ln \frac n 2 \leq \ln \left( n \right) \).
  % If~$n \geq 2 e^2$, then~$\alpha_n^\star \geq 0$ and~$|\alpha_n| \leq \alpha_n^\star + \bigl(\frac n 2\bigr)^{-1/4} \ln\frac n 2 \leq \bigl(\frac 1 4 + e^{-1/2}\bigr) \ln \frac n 2 \leq \ln n$.
  We can thus apply Lemma~\ref{lem:saddlepoint:unif}, which states that
  \begin{equation*}
    \left|\alpha_n
      -\frac 1 {4\beta_n} {\ln\left( \frac n 2 \right)} + \frac{\beta_n}{2} - \frac{\ln\beta_n}{2\beta_n}
      - E_n(\beta_n)\sqrt{\frac{2}{n}}
      \right| < \frac{ \delta_n }{n}
  \end{equation*}
  for some \(\widetilde O(1)\) sequence~$\delta_n$ which does not depend on~$\alpha_n$.
  Substituting \(\beta_n = 1 - \epsilon_n\) 
  and using \( 0 = - \alpha_n^\star + \frac 1 4 {\ln\left( \frac n 2 \right)} - \frac 1 2\),
  the inequality becomes
  \begin{equation*}
    \left|\alpha_n - \alpha_n^\star + \frac 1 4 {\ln\left( \frac n 2 \right)}
    - \frac 1 2   
    -\frac {\ln\left( \frac n 2 \right)} {4 \left( 1 - \epsilon_n \right) }  + \frac{1 - \epsilon_n }{2} - \frac{\ln\left( 1 - \epsilon_n \right)}{2\left( 1 - \epsilon_n \right)}
      - E_n\left( 1 - \epsilon_n \right)\sqrt{\frac{2}{n}}
      \right| < \frac{ \delta_n }{n}.
  \end{equation*}
  which reduces to
  \begin{equation}
    \left|\alpha_n - \alpha_n^\star 
      -\frac {\epsilon_n} {4 \left( 1 - \epsilon_n \right) } \ln\left( \frac n 2 \right)  - \frac{ \epsilon_n }{2} - \frac{\ln\left( 1 - \epsilon_n \right)}{2\left( 1 - \epsilon_n \right)}
      - E_n\left( 1 - \epsilon_n \right)\sqrt{\frac{2}{n}}
      \right| < \frac{ \delta_n }{n}.
      \label{eq:diff_alpha_alphastar_with_epsilon}
  \end{equation}
  By the definition of \(E_n\) (cf. Lemma~\ref{lem:estimation_saddlepoint}), we observe that
  \[E_n(1-x) =
    \frac {x \left( x - 2\right)} 8 \ln  \left( {\frac n 2}
  \right) 
  - \frac {{(1-x)}^{2}} 4    
  + \frac 1 {2{(1-x)}^{2}}   
  + \frac 1 2  
  - \frac {\ln  \left(1-x \right)} 4 
  + {\frac {\ln  \left( 1-x \right) }{4 {(1-x)}^{2}}}.
\]
This expression decomposes as \(E_n(1-x) = f(x) \ln\left( \frac n 2 \right) + g(x)\)
where \(f\) and \(g\) are continuous at \(x = 0\).
Consequently \( \sup_{\left| x \right| \leq \mu_n}  E_n(1-x) / \ln\left( \frac n 2 \right) \) is a convergent sequence, hence bounded.
Thus, \(E_n(1 - \epsilon_n) / \ln\left( \frac n 2 \right)\)
is uniformly bounded.

Then, by isolating the  \(\ln\left( \frac n 2 \right)\) term  in Equation~\ref{eq:diff_alpha_alphastar_with_epsilon} via the triangle inequality,
we get
\begin{equation*}
   \frac {\left|\epsilon_n \right|} {4 \left| 1 - \epsilon_n \right|}   
   \ln\left( \frac n 2 \right)
   <
  \left|\alpha_n - \alpha_n^\star\right|
  + \frac{  \left| \epsilon_n \right| }{2} 
  + \frac{\left| \ln\left( 1 - \epsilon_n \right) \right|}{2 \left| 1 - \epsilon_n \right|}
  + \left|E_n\left( 1 - \epsilon_n \right) \right| \sqrt{\frac{2}{n}}
  + \frac{ \delta_n }{n}
\end{equation*}
  We now use the inequalities 
  \(\left|\alpha_n - \alpha_n^\star\right| \leq 
  \left( \frac n 2 \right)^{-\frac 1 4} \ln \left( \frac n 2 \right)\),
  \( \left| \ln \left( 1 - \epsilon_n \right) \right|  
  \leq  \left| \epsilon_n \right|
  \)
  and \(\left| 1 - \epsilon_n  \right| < 1 + \mu_n \)
  to deduce
  \begin{equation*}
    \left| \epsilon_n \right| 
    < 
    4 \left( 1 + \mu_n \right) 
    \left( \frac n 2 \right)^{-\frac 1 4} 
    + \frac{2 (\mu_n + 2) \left| \epsilon_n \right| }{\ln\left( \frac n 2 \right)}
    +  \left| \frac{E_n\left( 1 - \epsilon_n \right)}
    {\ln\left( \frac n 2 \right)}  \right|4 (\mu_n + 1)
    \sqrt{\frac{2}{n}}
    + \frac{ 4 (\mu_n+1) \delta_n}
    {\ln\left( \frac n 2 \right) n }.
  \end{equation*}
    By absorbing the second term into the left-hand side (for \(n\) sufficiently large), we establish
    \begin{equation*}
    \left| \epsilon_n \right| 
    < 
    \frac 1 {1 - \frac{2 (\mu_n + 2) }{\ln\left( \frac n 2 \right)}}
    \left[ 4 \left( 1 + \mu_n \right) 
    \left( \frac n 2 \right)^{-\frac 1 4} 
    +  \left| \frac{E_n\left( 1 - \epsilon_n \right)}
    {\ln\left( \frac n 2 \right)}  \right|4 (\mu_n + 1)
    \sqrt{\frac{2}{n}}
    + \frac{ 4 (\mu_n+1) \delta_n}
    {\ln\left( \frac n 2 \right) n }  \right].
  \end{equation*}
  This shows that 
  \(\epsilon_n =  O\left( n^{-1/4} \right)\) uniformly in \(\alpha_n\).

  % Setting
  % \begin{multline*}
  %   F_n(x) := - \frac{1} {4 (1 - x)} \ln\left( \frac n 2 \right)
  %   - \frac 1 2
  %   - \frac 1 {2 (1-x)} \frac {\ln\left( 1-x \right)} x
  %   + \frac {x - 2} {8 (1-x)^2} \frac{ \ln\left( \frac n 2  \right)}{\sqrt{\frac n 2}} \\
  %   +  \frac{
  %     \left(  x - 2 \right) \ln(1-x)
  %     + x^3 - 4x^2 + 7 x - 6
  %   }
  %   {4 (1-x)^2} 
  %   \sqrt{\frac 2 n}    
  % \end{multline*},
  % we get after reorganizing all the terms
  % \begin{equation*}
  %   \left|\alpha_n - \alpha_n^\star + \epsilon_n F_n \left( \epsilon_n \right)
  %   - \frac 3 4 \sqrt {2 n}
  %     \right| < \frac{ \delta_n }{n}
  % \end{equation*}

\paragraph*{Approximating \(  \boldsymbol{1 - \beta_n } \)  in terms of \(\boldsymbol{\alpha_n - \alpha_n^\star}\).}
We perform a Taylor expansion around \(x=0\) of the function
\[
x \mapsto
-\frac {x} {4 \left( 1 - x \right) } \ln\left( \frac n 2 \right)  
- \frac{ x }{2} - \frac{\ln\left( 1 - x \right)}{2\left( 1 - x \right)}
- E_n\left( 1 - x \right)\sqrt{\frac{2}{n}}.
\] 
Substituting this expansion
into Equation~\eqref{eq:diff_alpha_alphastar_with_epsilon}
and discarding terms of order \(\widetilde{O} \left( n^{-\frac 3 4} \right)\)
yields the uniform bound
\begin{equation*}
  \alpha_n
    - \alpha_n^\star
    - \frac 1 4 \ln\left( \frac n 2 \right) \epsilon_n
    + \frac{3 - \ln\left( \frac n 2 \right)}{4} \epsilon_n^2
    - \frac 3 4 \sqrt{\frac 2 n}
    = O\left( \ln (n) \, n^{-3/4} \right)
\end{equation*}
and thus
\begin{equation}
  \epsilon_n =
      \frac{4(\alpha_n - \alpha_n^\star)}{\ln\left( \frac n 2 \right)}
      + \frac{3 - \ln\left( \frac n 2 \right)}{\ln\left( \frac n 2 \right)}\epsilon_n^2
      - \frac{3}{\ln\left( \frac n 2 \right)} \sqrt{\frac 2 n}
      + O\left( n^{-3/4} \right).
  \label{eq:epsilon_n}
\end{equation}
To leading order, this becomes Equation~\eqref{eq:estimate_epsilon_n}.

\paragraph*{Approximating  \(  \boldsymbol{  \left( 1 - \beta_n \right)^2 } \) in terms of \(\boldsymbol{\alpha_n - \alpha_n^\star}\).}
Squaring the estimate~\eqref{eq:epsilon_n}, we find
\begin{equation*}
  \epsilon_n^{\,2} = 
  \frac{16 \,(\alpha_n - \alpha_n^\star)^2}{\ln\left( \frac n 2 \right)^2} 
  + 8\left( 3 - \ln\left( \frac n 2 \right) \right)  \frac{\alpha_n - \alpha_n^\star}{\ln\left( \frac n 2 \right)^2} \epsilon_n^{\,2} 
  -24 \frac{\alpha_n - \alpha_n^\star}{\ln\left( \frac n 2 \right)^2} \sqrt{\frac 2 n}
  + O \left( \frac 1 n \right).
  \label{eq:epsilon2_not_approximated}
\end{equation*}
We now substitute \(\epsilon_n^{\,2}\) on the right side by the estimate given by  Equation~\eqref{eq:estimate_epsilon_n}. 
This recovers the expansion of \(\left( 1 - \beta_n \right)^2\) of the lemma.

\paragraph*{Approximating  \(  \boldsymbol{  \left( 1 - \beta_n \right)^3 } \) in terms of \(\boldsymbol{\alpha_n - \alpha_n^\star}\).}
The estimate for \( \left( 1 - \beta_n \right)^3\)
is straightforwardly inferred by cubing Equation~\eqref{eq:estimate_epsilon_n}.
\end{proof}

Equipped with these estimates on the proximity of~$\beta_n$ to~$1$, 
we can now express the asymptotic behavior of \(g_{n,n/2 + \alpha_n \sqrt{n/2}}\)
purely in terms of \(\alpha_n\) (rather than \(\beta_n\))
whenever \(\alpha_n\) lies near \(\alpha_n^\star\).

\begin{lemma}
    Let~$\alpha_n^\star = \frac 1 4 \ln\left( \frac n 2 \right) - \frac 1 2$ 
  and let~$\alpha_n$ be any sequence 
  such that~$k_n = \frac n 2 + \alpha_n \sqrt{\frac n 2}$ is an integer 
  and~$|\alpha_n - \alpha_n^\star| \leq \left(\frac n 2\right)^{-1/4} \ln\left(  \frac n 2 \right)$.

We have the uniform asymptotic estimate
% \noindent
% \textbf{(i) Leading-order approximation}  with relative error term \( \widetilde O \left(  n ^{-1/4}\right)\)
% \begin{multline}
%   g_{n,k_n} \sim
%   \frac 1 {\sqrt{2\pi \Delta_n}}
%   \left( \frac {n} {2e} \right)^{\frac n 2}
%   \exp\left[
%     \left(  \frac {\ln\left( \frac n 2 \right)} 2   + 2 \right) 
%     \sqrt{\frac n 2}
%     + \frac {\ln\left(  \frac n 2 \right)^2 } {32} 
%     - \frac {3 \ln \left( \frac n 2 \right)} 8 
%     + \frac 1 8     
%     \right]  
%     \times \\
%     \exp\left( 
%       - \frac 4 {\ln\left( \frac n 2 \right)} 
%       \left( \alpha_n - \alpha_n^\star \right)^2
%       \sqrt{ \frac n 2}
%       \right)
%       \times  \left( 1 + \widetilde O \left(  n ^{-1/4}\right) \right)
% \end{multline}
% where \(\Delta_n = \left(\frac{\ln\frac n 2}{2\beta_n} + \frac{\ln\beta_n}{\beta_n} - \frac{1}{\beta_n} + \beta_n\right)\sqrt{\frac n 2} \)
%
% \noindent
% \textbf{(ii) Refined approximation} with relative error term \( \widetilde O \left(  n ^{-1/2}\right)\).
\begin{multline*}
  g_{n,k_n} \sim 
  \frac { \left( \frac n 2 \right)^{-1/4} } {\sqrt{\pi \ln \left( \frac n 2 \right)  }}
  \left( \frac {n} {2e} \right)^{\frac n 2}
\exp\left[
    \left(  \frac {\ln\left( \frac n 2 \right)} 2   + 2 \right) 
    \sqrt{\frac n 2}
    + \frac {\ln\left(  \frac n 2 \right)^2 } {32} 
    - \frac {3 \ln \left( \frac n 2 \right)} 8 
    + \frac 1 8     
    \right]  
    \times \\
    \exp\left( 
      - \frac 4 {\ln\left( \frac n 2 \right)} 
      \left( \alpha_n - \alpha_n^\star \right)^2
      \sqrt{ \frac n 2}
      \right) \times
      \left(
        1 + u_n \sqrt{\frac n 2}(\alpha_n - \alpha_n^\star)^3 + v_n (\alpha_n - \alpha_n^\star) + \widetilde O\left( \frac 1 {\sqrt n} \right)
        \right)
\end{multline*}
where 
\( u_n := \frac{16} {3 \ln\left( n/2 \right) ^2} - \frac {32} { \ln\left( n/2 \right)^3 }\), 
\(v_n := 
\frac{\ln\left( n/2 \right)} 4
+ \frac 3 2
+ \frac 4 {\ln\left( n/2 \right)} 
+ \frac {12} {\ln \left( n/2 \right)^2}
\),\footnote{The exact values of \(u_n\) and \(v_n\) are not important here since they will disappear in the summation.}
and the \(\widetilde O\left( \frac 1 {\sqrt n} \right)\) term is independent of \(\alpha_n\).
\label{lem:gnk_in_terms_of_diff_alpha}
\end{lemma}

\begin{proof}
  The proof, largely computational, consists of a chain of Taylor expansions.
  We start from the known asymptotic estimate of \(g_{n,k_n}\) (cf. Lemma~\ref{lem:evalr:unif} and Proposition~\ref{prop:gnk_asympt})
  \begin{equation}
    g_{n,k_n} \sim
    \frac 1 {\sqrt {2 \pi \Delta_n} }
    \left( \frac n {2 e} \right)^{\frac n 2}
    \exp\left(\lambda_n(\beta_n) \sqrt{\frac{n}{2}} + \mu_n(\beta_n)\right)
    \left( 1 + \widetilde O\left( \frac{1}{\sqrt n} \right) \right).
    \label{eq:general_asympt_gnk}
  \end{equation}
  By expanding this formula via Taylor series around \(\beta_n = 1\),
  we recast it as an expansion in \(1 - \beta_n\).
  Lemma~\ref{lem:beta_expansion*} then allows us to replace each power of \(1-\beta_n\) with its expansion in \(\alpha_n - \alpha_n^\star\),
  completing the proof.

  Crucially, all the estimates are uniform:
  \(\beta_n\) converges uniformly to \(1\),
  and since the functions appearing in Equation~\eqref{eq:general_asympt_gnk} are continuous at \(\beta_n =1\),
  the induced estimates inherit uniformity.

  % Proving the estimate from the statement of the lemma boils down to combining all of our previous estimates.
  % An important remark here is that this is only permitted because we have established that the errors terms are uniform in~$k$ (within a certain range).

  \paragraph*{Estimation of~$\boldsymbol{\lambda_n(\beta_n)}$.}
  By Lemma~\ref{lem:study_lambda}, we have
  \begin{equation*}
    \lambda_n\left( \beta_n \right) =
    \frac {\ln\left( \frac n 2 \right)} 2 + 2 
        - \frac{\ln\left( \frac n 2 \right)}{4} \left( 1 - \beta_n \right)^2
        + \left(1 - \frac {5 \ln\left( \frac n 2 \right) } {12} \right) \left( 1 - \beta_n \right)^3 
        + O\left(\ln\left( n \right) \left( 1 - \beta_n \right)^4\right).
  \end{equation*}
  By Lemma~\ref{lem:beta_in_terms_of_alpha},
  we can express \(\left( 1 - \beta_n \right)^2\) and \(\left( 1 - \beta_n \right)^3\) as a uniform expansion of \(\alpha_n - \alpha_n^\star\). Straightforward computation gives
\begin{equation*}
  \lambda_n\left( \beta_n \right) =
    \frac {\ln\left( \frac n 2 \right)} 2 + 2 
        - \frac{4}{\ln\left( \frac n 2 \right)} \left( \alpha_n - \alpha_n^\star \right)^2
        + u_n \left( \alpha_n - \alpha_n^\star \right)^3
        + \frac 6 {\ln\left( \frac n 2 \right)} \sqrt{\frac 2 n} \left( \alpha_n - \alpha_n^\star \right)
        + \widetilde O \left(  \frac 1 n \right).
\end{equation*}

  \paragraph*{Estimation of~$\boldsymbol{\mu_n(\beta_n)}$.}
  Similarly,
  the Taylor expansion of \(\mu_n\) function around $\beta_n=1$ is
  \begin{equation*}
    \mu_n(\beta_n) =
    \frac{ \ln \left( \frac n 2 \right)^2 }{32}
    - \frac {3 \ln \left( \frac n 2 \right)} 8
    + \frac 1 8 
    + \left( 
       \frac{\ln \left(  \frac n 2 \right)^2}{16}
       + \frac {3 \ln \left( \frac n 2 \right) } {8}
     \right)
     \left( 1 - \beta_n \right)
    + \widetilde O \left( \frac 1 {\sqrt n} \right).
  \end{equation*}
  Replacing \( 1 - \beta_n \) with the help of Lemma~\ref{lem:beta_in_terms_of_alpha}, we get
  \begin{equation*}
    \mu_n(\beta_n) =
    \frac{ \ln \left( \frac n 2 \right)^2 }{32}
    - \frac {3 \ln \left( \frac n 2 \right)} 8
    + \frac 1 8 
    + \frac {\ln \left( \frac n 2 \right) + 6 } 4
     \left( \alpha_n - \alpha_n^\star \right)
    + \widetilde O \left( \frac 1 {\sqrt n} \right).
  \end{equation*}

  \paragraph*{Estimation of~$\boldsymbol{\Delta_n}$.}
  Recall that~$\Delta_n = \left(\frac{\ln\frac n 2}{2\beta_n} + \frac{\ln\beta_n}{\beta_n} - \frac{1}{\beta_n} + \beta_n\right)\sqrt{\frac n 2}$.
  Taylor's theorem gives
  \begin{equation*}
    \Delta_n^{-1/2} 
    =
    \sqrt{ \frac { 2} { \ln \left( \frac n 2 \right)} }
    \left( \frac n 2 \right)^{-\frac 1 4}
    +
    \frac
    {6 - \ln \left( \frac n 2 \right) }
    { \sqrt 2 \ln \left( \frac n 2 \right)^{3/2}}
    \left( \frac n 2 \right)^{-\frac 1 4}
    \left( 1 - \beta_n \right)
    + \widetilde O \left( \frac 1 {\sqrt n} \right)
  \end{equation*}
  and so, by Lemma~\ref{lem:beta_in_terms_of_alpha},
  \begin{equation*}
    \Delta_n^{-1/2} 
    =
    \sqrt{ \frac { 2} { \ln \left( \frac n 2 \right)} }
    \left( \frac n 2 \right)^{-\frac 1 4}
    +
    2 \sqrt 2 \,
    \frac
    {6 - \ln \left( \frac n 2 \right) }
    { \ln \left( \frac n 2 \right)^{5/2}}
    \left( \frac n 2 \right)^{-\frac 1 4}
    \left( \alpha_n - \alpha_n^\star \right)
    + \widetilde O \left( \frac 1 {\sqrt n} \right).
  \end{equation*}

  \paragraph*{Combining everything.}
  We decompose the expansion 
  \begin{equation*}
  \lambda_n(\beta_n) \sqrt{\frac{n}{2}} + \mu_n(\beta_n)
  = \text{Principal}_n
  +
  \text{Correction}_n
  \end{equation*}
  into a principal term (non-vanishing as \(n \to \infty\)) 
  \begin{equation*}
    \text{Principal}_n
    =
  \left(  \frac {\ln\left( \frac n 2 \right)} 2   + 2 \right) 
    \sqrt{\frac n 2}
    + \frac {\ln\left(  \frac n 2 \right)^2 } {32} 
    - \frac {3 \ln \left( \frac n 2 \right)} 8 
    + \frac 1 8  
    - \frac 4 {\ln\left( \frac n 2 \right)} 
      \left( \alpha_n - \alpha_n^\star \right)^2
      \sqrt{ \frac n 2}
  \end{equation*}
  and correction terms (converging to \(0\))
  \begin{equation*}
    \text{Correction}_n
    = 
  u_n \sqrt{\frac n 2} \left( \alpha_n - \alpha_n^\star \right)^3 
  + 
  \left( \frac {\ln \left( \frac n 2 \right)} 4 + \frac 3 2 + \frac 6 {\ln \left( \frac n 2  \right)} \right)
  \left( \alpha_n - \alpha_n^\star \right)
  + \widetilde O \left( \frac 1 {\sqrt n}  \right).
  \end{equation*}
Some straight computations show that \(\frac {\exp\left( {\text{Correction}_n } \right)} {\sqrt {\Delta_n} } \) is indeed of the form
  \begin{equation*}
    \frac {e^{\text{Correction}_n }} {\sqrt {\Delta_n} } 
    = 
\sqrt{ \frac { 2} { \ln \left( \frac n 2 \right)} }
    \left( \frac n 2 \right)^{-\frac 1 4}
    \left(
        1 + u_n \sqrt{\frac n 2}(\alpha_n - \alpha_n^\star)^3 + v_n (\alpha_n - \alpha_n^\star) + \widetilde O\left( \frac 1 {\sqrt n} \right)
        \right).
  \end{equation*}  
  Injecting all these expansions back into Equation~\eqref{eq:general_asympt_gnk}
  yields the desired result.
\end{proof}

\subsubsection{Summation via the Euler-Maclaurin Formula}

We are finally computing the main term of the asymptotic expansion of \(g_n\).
We are going to sum the pointwise estimates using the Euler-Maclaurin formula.

\begin{proof}[Proof of Proposition~\ref{prop:central_sum}]
  The uniform estimate of Lemma~\ref{lem:gnk_in_terms_of_diff_alpha} applies to each $g_{n,k}$
  where $k$ is an integer such that \(|k - k_n^\star| \leq (\frac{n}{2})^{1/4}\ln\left( \frac{n}{2} \right)\):
    \begin{multline*}
  g_{n,k} \sim 
  \frac { \left( \frac n 2 \right)^{-1/4} } {\sqrt{\pi \ln \left( \frac n 2 \right)  }}
  \left( \frac {n} {2e} \right)^{\frac n 2}
\exp\left[
    \left(  \frac {\ln\left( \frac n 2 \right)} 2   + 2 \right) 
    \sqrt{\frac n 2}
    + \frac {\ln\left(  \frac n 2 \right)^2 } {32} 
    - \frac {3 \ln \left( \frac n 2 \right)} 8 
    + \frac 1 8     
    \right]  
    \times \\
    \exp\left( 
     - \frac 4 {\ln\left( \frac n 2 \right)} 
      \sqrt{ \frac 2 n}
      \left( k - k_n^\star \right)^2
      \right) \times
      \left(
        1 + \frac{ 2 u_n } n (k -  k_n^\star )^3 + v_n \sqrt{\frac 2 n} (k -  k_n^\star) + \widetilde O\left( \frac 1 {\sqrt n} \right)
        \right).
\end{multline*}
  We sum these asymptotic estimates 
  and factorize to obtain
\begin{multline}
  {
\sum_{\substack{k \text{ such that} \\ |k - k_n^\star| \leq (\frac{n}{2})^{1/4}\ln\frac{n}{2}}} g_{n,k} }
\sim 
 \frac{\left(\frac{n}{2}\right)^{-1/4}}{\sqrt{\pi\ln\frac{n}{2}}} 
\left( \frac {n} {2e} \right)^{\frac n 2}
e^{
    \left(  \frac 1 2 {\ln\left( \frac n 2 \right)}    + 2 \right) 
    \sqrt{\frac n 2}
    + \frac {1} {32} \ln\left(  \frac n 2 \right)^2 
    - \frac {3 } 8 \ln \left( \frac n 2 \right)
    + \frac 1 8     
} \times \\
\left( \sum_{\substack{k \text{ such that} \\ |k - k_n^\star| \leq (\frac{n}{2})^{1/4}\ln\frac{n}{2}}} f_n(k)  \right)
\times \left( 1 + \widetilde O \left( \frac 1 {\sqrt n} \right) \right)
\label{eq:sum_gnk_intermediate}
\end{multline}
where
\begin{equation*}
    f_{n}(t) :=
    \left(
        1 + \frac{ 2 u_n } n (t -  k_n^\star )^3 + v_n \sqrt{\frac 2 n} (t -  k_n^\star) 
        \right)
    \exp \left(  - \frac 4 {\ln\left( \frac n 2 \right)} 
      \sqrt{ \frac 2 n}
      \left( t - k_n^\star \right)^2
       \right).
  \end{equation*}
  Since the asymptotic expansion in Lemma~\ref{lem:gnk_in_terms_of_diff_alpha} is uniform over the range, the factorization by \(1 + \widetilde O \left( \frac 1 {\sqrt n} \right) \) is justified.\footnote{And there's the payoff for all the careful bookkeeping in the preceding lemmas.}

  \paragraph*{Euler-Maclaurin summation.}
  We need to estimate \(\sum_{\substack{k \text{ such that} \\ |k - k_n^\star| \leq (\frac{n}{2})^{1/4}\ln\frac{n}{2}}} f_n(k)\)
  to achieve the proof. 
  To do so, we use the Euler-Maclaurin formula
  (see for example~\cite[Theorem~9.2.2]{number-theory}):
 \begin{multline}\label{eq:eulermaclaurin}
    \sum_{\substack{k \text{ such that} \\ |k - k_n^\star| \leq (\frac{n}{2})^{1/4}\ln\left( \frac{n}{2} \right)}}  f_n(k)
    =
    \frac{f_n(k_n^-) + f_n(k_n^+)}2
    + \int_{k_n^-}^{k_n^+} f_n(t)dt
    + \frac{f_n'(k_n^+)- f_n'(k_n^-)}{12} \\
    + \int_{k_n^-}^{k_n^+} f_n''(t)
    \left( \left(t - \lfloor t \rfloor \right)^2 - \left(t -  \lfloor t \rfloor\right) + \frac 1 6  \right)
    dt
  \end{multline}
  where~$k_n^- :=  \left\lceil k_n^\star - \left( \frac n 2 \right)^{1/4}\ln \left( \frac n 2 \right) \right\rceil $ and~$k_n^+ := \left\lfloor k_n^\star + \left( \frac n 2 \right)^{1/4}\ln \left( \frac n 2 \right) \right\rfloor $.

 \paragraph*{Euler-Maclaurin, first fraction.}
  Since the floor and ceiling shift  by at most \(1\),
  we have 
  \[  
  k_n^\pm - k_n^\star \sim \pm \left( \frac n 2 \right)^{1/4}\ln \left( \frac n 2 \right)+ O(1).
  \]
  Using the definition of \(f_n\), we find that
 \begin{align}
    f_n \left( k_n^\pm \right) &
    \sim
    \left( 1 + \widetilde{O} \left( n^{-1/4} \right) \right)
     \exp \left(  - \frac 4 {\ln\left( \frac n 2 \right)} 
      \sqrt{ \frac 2 n}
       \left( \sqrt{\frac n 2} \ln\left( \frac n 2 \right)^2
        + \widetilde O \left( n^{1/4} \right)
       \right)
       \right) \notag \\ &
       \sim \exp \left( -4 \ln\left( \frac n 2 \right) \right)
        = \left( \frac n 2 \right)^{-4}.
        \label{eq:asympt_fn_kn_pm}
      \end{align}
  We deduce that \(\frac{f_n(k_n^-) + f_n(k_n^+)}2 \sim \left( \frac n 2 \right)^{-4}\).

  \paragraph*{Euler-Maclaurin, central integral.}
  The central integral of Equation~\eqref{eq:eulermaclaurin} can be approximated by an integral over a symmetrical domain:
  \begin{align*}
    \int_{k_n^-}^{k_n^+} f_n(t) \, dt
    &= \int_{
      k_n^\star - \left( \frac n 2 \right)^{1/4}\ln \left( \frac n 2 \right)
      }^{
      k_n^\star +\left( \frac n 2 \right)^{1/4}\ln \left( \frac n 2 \right)
      } f_n(t) \, dt
       +
        \int_{
      k_n^- 
      }^{
      k_n^- + O(1)
      } f_n(t) \, dt
      +
        \int_{
      k_n^+ + O(1) 
      }^{
      k_n^+ 
      } f_n(t) \, dt
    \\
    &= \int_{
      - \left( \frac n 2 \right)^{1/4}\ln \left( \frac n 2 \right)
      }^{
        \left( \frac n 2 \right)^{1/4}\ln \left( \frac n 2 \right)
        } f_n(k_n^\star + x)\, dx
       +
      \underbrace{ O\left( f_n(k_n^-) \right)}_{= \, O \left( n^{-4} \right)}
      +
      \underbrace{ O\left( f_n(k_n^+) \right)}_{= \, O \left( n^{-4} \right)}.
  \end{align*}
  Since
  \( x \mapsto \left(
        \frac{ 2 u_n } n x^3 + v_n \sqrt{\frac 2 n} x
        \right)
    \exp \left(  - \frac 4 {\ln\left( \frac n 2 \right)} 
      \sqrt{ \frac 2 n}
      x ^2
       \right)\)
       is an odd function, 
       its integral over a symmetric domain vanishes.
       We are left with a Gauss integral
  \begin{align*}
    \int_{k_n^-}^{k_n^+} f_n(t) \, dt
    & = 
    \int_{
      - \left( \frac n 2 \right)^{1/4}\ln \left( \frac n 2 \right)
      }^{
        \left( \frac n 2 \right)^{1/4}\ln \left( \frac n 2 \right)
        } 
    \exp \left(  - \frac 4 {\ln\left( \frac n 2 \right)} 
      \sqrt{ \frac 2 n}
      x ^2
       \right)
      dx
      + O\left( n^{-4} \right) \\
      & = 
      \frac {\sqrt{ \ln\left( \frac n 2 \right) }}
      2 \left( \frac n 2 \right)^{1/4}
      \int_{
      - 2 \sqrt { \ln \left( \frac n 2 \right) }
      }^{
         2 \sqrt { \ln \left( \frac n 2 \right) }
        } 
    \exp \left(  - 
      u ^2
       \right)
       du
       + O\left( n^{-4} \right) \\
      & = 
      \frac {\sqrt{ \ln\left( \frac n 2 \right) }}
      2 \left( \frac n 2 \right)^{1/4}
      \sqrt\pi \left( 
      1
      - \text{erfc}
      \left( 2 \sqrt { \ln \left( \frac n 2 \right) } \right)
       \right)
       + O\left( n^{-4} \right),
  \end{align*}
  where \(\text{erfc}\)
  is the complementary error function, whose asymptotic behavior is~\cite[7.1.23]{abramowitz-stegun}
  \begin{equation*}
    \text{erfc}(x)  \underset{x \rightarrow +\infty}{\sim} \frac{ 1 + O(x^2) } {x e^{x^2} \sqrt{\pi}}.
  \end{equation*}
  We conclude
  \begin{equation}
    \int_{k_n^-}^{k_n^+} f_n(t) \, dt \sim 
     \frac {\sqrt{\pi}}
      2 \sqrt{ \ln\left( \frac n 2 \right) } \left( \frac n 2 \right)^{1/4}
      + O\left( n^{- 4 + \frac 1 4 } \right).
      \label{eq:int_fn}
  \end{equation}

  \paragraph*{Bounding \(\boldsymbol{f'_n}\) and \(\boldsymbol f''_n\).} By a clear induction,
  \(f_n(t)\) and all of its derivatives are of the form \(p_n(t-k_n^\star)
      e^{  - \frac 4 {\ln\left( \frac n 2 \right)} 
      \sqrt{ \frac 2 n}
      \left( t - k_n^\star \right)^2
      }\),
where $p_n(x)$ is a polynomial where the degree does depend on \(n\).

  But observe that for any such polynomial \(p_n\)
  and~\(t\) any number such that \(\left| t - k_n^\star \right| \leq  \left( \frac n 2 \right)^{1/4} \ln\left( \frac n 2 \right) \),
  \begin{multline}
    \frac{\partial}{\partial t} \left( 
      p_n(t-k_n^\star)
      e^{  - \frac 4 {\ln\left( \frac n 2 \right)} 
      \sqrt{ \frac 2 n}
      \left( t - k_n^\star \right)^2
      }
      \right) 
      = \\
      p_n'(t-k_n)
      e^{  - \frac 4 {\ln\left( \frac n 2 \right)} 
      \sqrt{ \frac 2 n}
      \left( t - k_n^\star \right)^2
      }
      - 8 \frac
      { \left( t - k_n^\star \right) p_n\left( t - k_n^\star \right)
      }
      {
        \ln \left( \frac n 2 \right) \sqrt{ \frac n 2}
      }
      e^{  - \frac 4 {\ln\left( \frac n 2 \right)} 
      \sqrt{ \frac 2 n}
      \left( t - k_n^\star \right)^2
      }.
      \label{eq:partial_t_form_fn}
  \end{multline}
  Because 
  we have for any positive integer \(i\) 
  \begin{equation*}
    i \left( t - k_n^\star \right)^{i-1} = \left( t - k_n^\star \right)^{i} \times \widetilde O \left( n^{-1/4} \right),
  \end{equation*}
  we can deduce that \(p_n'(t-k_n^\star) = p_n(t-k_n^\star) \times \widetilde O \left(n^{-1/4}\right)\).
  Moreover, since
  \(- \frac
      { \left( t - k_n^\star \right)    }
      {
        \ln \left( \frac n 2 \right) \sqrt{ \frac n 2}
      }
      =
      \widetilde O \left( n^{-1/4}\right)
      \),
  we can infer from Equation~\eqref{eq:partial_t_form_fn} 
  \begin{equation*}
    \frac{\partial}{\partial t} \left( 
      p_n(t-k_n^\star)
      e^{  - \frac 4 {\ln\left( \frac n 2 \right)} 
      \sqrt{ \frac 2 n}
      \left( t - k_n^\star \right)^2
      }
      \right) 
      = p_n(t-k_n^\star)
      e^{  - \frac 4 {\ln\left( \frac n 2 \right)} 
      \sqrt{ \frac 2 n}
      \left( t - k_n^\star \right)^2}
      \times
      \widetilde O \left(  
      {n^{-1/4}}
       \right).
  \end{equation*}
  In particular, we conclude that \(f_n'(t) = f_n(t) \times  \widetilde O \left( n^{-1/4} \right) \) and
  \(f_n''(t) = f_n(t) \times \widetilde O \left( n^{-1/2} \right) \).
  Note that all the \(\widetilde O\)-terms in the above estimates are uniform in \(t\), since they all come from the inequality \(\left| t - k_n^\star \right| \leq  \left( \frac n 2 \right)^{1/4} \ln\left( \frac n 2 \right) \).

  \paragraph*{Euler-Maclaurin, second fraction.}
  Using the property we have just established,
  \[
    \frac{f_n'(k_n^+)- f_n'(k_n^-)}{12}
    = \widetilde O \left( \frac{ f_n(k_n^\pm)} {n^{1/4}} \right).
  \]
  By Equation~\eqref{eq:asympt_fn_kn_pm},
  \begin{equation*}
    \frac{f_n'(k_n^+)- f_n'(k_n^-)}{12}
    = \widetilde O \left( {n^{-4 - \frac 1 4 }} \right).
  \end{equation*}

  \paragraph*{Euler-Maclaurin, remainder integral.}
  Now using \(f_n''(t) = f_n(t) \times \widetilde O \left( n^{-1/2} \right) \) 
  and the fact that 
  \( t \mapsto \left(t - \lfloor t \rfloor \right)^2 - \left(t -  \lfloor t \rfloor\right) + \frac 1 6\) is bounded over \(\mathbb R\),
  we see that
  \begin{equation*}
  \int_{k_n^-}^{k_n^+} f_n''(t)
    \left( \left(t - \lfloor t \rfloor \right)^2 - \left(t -  \lfloor t \rfloor\right) + \frac 1 6  \right)
    dt
    =
    \int_{k_n^-}^{k_n^+} f_n(t) dt
    \times \widetilde O \left( \frac 1 {\sqrt n} \right).
  \end{equation*}
  By Equation~\eqref{eq:int_fn},
  \begin{equation*}
  \int_{k_n^-}^{k_n^+} f_n''(t)
    \left( \left(t - \lfloor t \rfloor \right)^2 - \left(t -  \lfloor t \rfloor\right) + \frac 1 6  \right)
    dt
    =
    \widetilde O \left( n^{-1/4} \right).
  \end{equation*}

  \paragraph*{Conclusion.}
  Gathering all these estimates in Equation~\eqref{eq:eulermaclaurin},
  we obtain
  \begin{equation*}
    \sum_{\substack{k \text{ such that} \\ |k - k_n^\star| \leq (\frac{n}{2})^{1/4}\ln\left( \frac{n}{2} \right)}}  f_n(k)
    \sim
     \frac {\sqrt{\pi}}
      2 \sqrt{ \ln\left( \frac n 2 \right) } \left( \frac n 2 \right)^{1/4}
     +  \widetilde O \left( n^{- \frac 1 4} \right).
  \end{equation*}
  We conclude by plugging this into Equation~\eqref{eq:sum_gnk_intermediate}.
\end{proof}

\subsection{Step 5. Upper Bounds for the Tails}\label{sec:tails}

We now address the tail behavior.
The objective is to establish that \(k\)-values 
outside the window \( |k - k_n^\star| \leq (\frac{n}{2})^{1/4}\ln\left( \frac{n}{2} \right) \)
yield a negligible contribution to \(g_n = \sum_{k=2}^n g_{n,k}\).

Classically, one proves such tail bounds via unimodality of \(g_{n,k}\):
a unimodal sequence immediately implies that 
all remote values are dominated by near-mode values,
rendering their total contribution asymptotically negligible.
Unfortunately, a direct proof of unimodality for \(g_{n,k}\) seems hard to come by 
since we lack sufficient information about \(g_{n,k}\) far from the mode.
To overcome this obstacle, we employ the following key idea:
rather than proving unimodality of the sequence \(g_{n,k}\) itself,
we prove unimodality of an upper bound, namely the saddle-point bound
\(
\frac { G_k(r_{n,k}) } { {r_{n,k}}^{n} } 
\),
where \(r_{n,k}\) is the unique solution of  \(r_{n,k} G_{k}'(r_{n,k}) = n G_{k_n}(r_{n,k})\) 
(this was denoted~$r_n$ in Lemma~\ref{lem:def_saddlepoint}: the dependence in \(k\) is now explicit).

The unimodality is simpler to prove, notably because of the following lemma.

\newcommand\SB{\text{SB}}

\begin{lemma}
Set \(n \geq 5\) and $k \in \{2,\dots,n-1\}$.
Let \(\SB_{n,k}\) be the saddle-point bound
\begin{equation*}
  \SB_{n,k} := \frac { G_k(r_{n,k}) } { {r_{n,k}}^{n} } 
\end{equation*}
where \(r_{n,k}\) denotes the unique solution of  \(r_{n,k} G_{k}'(r_{n,k}) = n G_{k_n}(r_{n,k})\),\footnote{cf. Lemma~\ref{lem:def_saddlepoint} for the existence and uniqueness of the solution} and \(G_k(z) = \sum_{n \geq 0} g_{n,k} z^n\) the generating function of \gitgraph{}s with \(k\) \main{}s.

The following properties hold:
\begin{enumerate}
  \item \(g_{n,k} \leq \SB_{n,k}\) ;
  \item if \(r_{n,k+1} \times \left( 1 + k \frac{r_{n,k+1}} {1-r_{n,k+1}}\right) \geq 1\),
then \(\SB_{n,k} \leq \SB_{n,k+1}\) ;
  \item if \(r_{n,k} \times \left( 1 + k \frac{r_{n,k}} {1-r_{n,k}}\right) \leq 1\),
then \(\SB_{n,k} \geq \SB_{n,k+1}\).
\end{enumerate}
\label{lem:unimodality}
\end{lemma}
\begin{proof}
  Item 1 is straightforward since \(g_{n,k} {r_{n,k}}^n \leq \sum_m g_{m,k} {r_{n,k}}^m = G_{k}(r_{n,k})\).

  \paragraph*{Item 2.} Since \(z \mapsto \frac{G_k(z)}{z^n} \) is upward convex (see \cite[VIII.4]{FS2009}) and 
  \( \frac{\partial } {\partial z} \left.\left(  \frac{G_{k}(z)}{z^n}\right)\right|_{z=r_{n,k}} = 0\) by definition of \(r_{n,k}\),
  the minimum of \(\frac{G_k(z)}{z^n}\) is attained at \(z = r_{n,k}\).
  So 
  \begin{equation*}
    \SB_{n,k} 
    = \frac{G_k\left( {r_{n,k}} \right)} {{r_{n,k}}^n}
    \leq \frac{G_k\left( {r_{n,k+1}} \right)} {{r_{n,k+1}}^n}.
  \end{equation*}
  But by Proposition~\ref{prop:Gkz}, 
  \(G_{k+1}(z) = z \left( 1 + k \frac z{1-z} \right) G_k(z)\),
  so
 \begin{equation*}
    \SB_{n,k} 
    \leq \frac{G_k\left( {r_{n,k+1}} \right)} {{r_{n,k+1}}^n}
    =  \frac 1 {  {r_{n,k+1}} \left( 1 + k \frac {r_{n,k+1}}{1-{r_{n,k+1}}} \right)     }
    \frac{G_{k+1}\left( {r_{n,k+1}} \right)} {{r_{n,k+1}}^n}
    = 
    \frac {\SB_{n,k+1}} {  {r_{n,k+1}} \left( 1 + k \frac {r_{n,k+1}}{1-{r_{n,k+1}}} \right)     }.
  \end{equation*}
  If the denominator is at least $1$,
  then \(\SB_{n,k} \leq \SB_{n,k+1}\).

  \paragraph*{Item 3.} The idea is similar to Item 2, but we use:
  \begin{equation*}
    \SB_{n,k+1}
    = \frac{G_{k+1}\left( {r_{n,k+1}} \right)} {{r_{n,k+1}}^n}
    \leq \frac{G_{k+1}\left( {r_{n,k}} \right)} {{r_{n,k}}^n}
    = r_{n,k} \left( 1 + k \frac{r_{n,k}} {1-r_{n,k}}\right) \SB_{n,k},
  \end{equation*}
  which implies \(\SB_{n,k+1} \leq \SB_{n,k}\) if \( r_{n,k} \left( 1 + k \frac{r_{n,k}} {1-r_{n,k}}\right) \leq 1\).
\end{proof}

Our strategy for bounding the tails is to show that
each \(g_{n,k}\) outside the central range is dominated by \(\SB_{n,k_n^\pm}\).
(\(k_n^-\) and \(k_n^+\) denote the endpoints of the central sum --- cf. Proposition~\ref{prop:central_sum}.)
To this end, we partition the tails into four regions:
\begin{itemize}
  \item Far left: \(k \leq \frac n 2 - \sqrt{\frac n 2}\),
  \item Near left: \(\frac n 2 - \sqrt{\frac n 2} \leq k \leq k_n^-\),
  \item Near right: \(k_n^+ \leq k \leq \frac n 2 + \ln\left( \frac n 2 \right) \sqrt{\frac n 2} \),
  \item Far right:  \(k \geq \frac n 2 + \ln\left( \frac n 2 \right) \sqrt{\frac n 2} \).
\end{itemize}

\subsubsection{Far Left Region}

  In the lower region~$0 \leq k \leq \frac n 2 - \sqrt{\frac n 2}$, we can prove that the sequence~$g_{n,k}$ is increasing, 
  exploiting the explicit formula  involving Stirling numbers of the first kind.

  \begin{lemma}
    For \(n \geq 3\) and for~$0 \leq k \leq \frac n 2 - \sqrt{\frac n 2}$, we have~$g_{n,k} \leq g_{n,k+1}$.
    \label{lem:farleft}
  \end{lemma}
  \begin{proof}
    By Corollary~\ref{cor:formula}, we have
    \begin{equation*}
      g_{n,k} = \sum_{f = 1}^{k - 1} \stirling k f \binom{n-k-1}{k-f-1}
    \end{equation*}
    Thus, we have
    \begin{equation*}
      g_{n,k+1}
      = \sum_{f = 1}^k \stirling{k+1}{f} \binom{n-k-2}{k-f}
      \geq \sum_{f=1}^{k-1} \stirling{k+1}{f} \binom{n-k-2}{k-f}
      \geq \sum_{f=1}^{k-1} k \stirling k f \binom{n-k-2}{k-f}
    \end{equation*}
    using the recurrence formula for the Stirling numbers of the first kind:~$\stirling[1]{k+1}{f} = k \stirling[1] k f + \stirling[1] k {f-1}$.
    Hence, it is enough to show  
    \( k  \binom{n-k-2}{k-f}  \geq {\binom{n-k-1}{k-f-1} } \) to prove~$g_{n,k+1}\geq g_{n,k}$.
    But the quantity
    \begin{equation*}
      k \frac{ \binom{n-k-2}{k-f} } {\binom{n-k-1}{k-f-1} } = \frac{k(n-2k+f)(n-2k+f-1)}{(n-k-1)(k-f)}
    \end{equation*}
     is increasing in~$f$ so it is minimal at~$f=1$, hence
    \begin{align*}
      k  \frac{ \binom{n-k-2}{k-f} } {\binom{n-k-1}{k-f-1} }
      &\geq \frac{k(n-2k)(n-2k+1)}{(n-k-1)(k-1)} \\
      &\geq \frac{(n-2k)(n-2k+1)}{(n-k-1)}       &\text{because~$\frac{k}{k-1} \geq 1$} \\
      &\geq \frac{\sqrt{2n}(\sqrt{2n}+1)}{(n-2)} \geq \frac{2n}{n-2} &\text{because~$2k \leq n - \sqrt{2n}$ and~$k\geq 1$} 
    \end{align*}
    and the last quantity is larger than~$1$ for~$n \geq 3$.
  \end{proof}

\subsubsection{Near Left Region}

We now analyze the region just before the mode,
where proximity to the peak allows us
to obtain a precise estimate of  \(\left( 1 + k \frac{r_{n,k}} {1-r_{n,k}}\right)  r_{n,k}\).

\begin{lemma} There exists an integer \(n_0\) such that 
  for all \(n_0 \leq n\)
  and \[
    \frac n 2 - \sqrt{\frac n 2} - 1 \leq k < k_n^-,
  \]
  where \( k_n^- =
  \left\lceil
  \frac n 2 + \left( \frac 1 4 \ln\left( \frac n 2 \right) - \frac 1 2 \right) \sqrt{\frac n 2} - \left( \frac n 2 \right)^{\frac 1 4}\ln \left( \frac n 2 \right)
  \right\rceil
  \),
  we have \(\SB_{n,k} \leq \SB_{n,k+1}\)
  (see Lem\-ma~\ref{lem:unimodality} for the definition of \(\SB\)).
  \label{lem:central_left}
\end{lemma}

\begin{proof}
  By Lemma~\ref{lem:unimodality}, it is sufficient to prove 
  \( \left( 1 + k \frac{r_{n,k+1}} {1-r_{n,k+1}}\right) r_{n,k+1}  \geq 1\).

  By Lemma~\ref{lem:decreasing_beta},
  we have \(r_{n,k+1} \geq r_{n,k_n^-}\).
  Since \(x \mapsto x\left( 1 + k \frac{x} {1-x}\right)\) is increasing on \([0,1)\),
  we get
  \begin{equation*}
    \left( 1 + k \frac{r_{n,k+1}} {1-r_{n,k+1}}\right) r_{n,k+1} 
      \geq
     \left( 1 + k \frac{r_{n,k_n^-}} {1-r_{n,k_n^-}}\right) r_{n,k_n^-} 
     \geq
     \left( 1 +  \frac{r_{n,k_n^-} \left( \frac n 2 - \sqrt{\frac n 2} - 1\right) } {1-r_{n,k_n^-}}\right) r_{n,k_n^-}.
  \end{equation*}
  We now proceed to an asymptotic asymptotic analysis of the latter expression
  to prove that it eventually exceeds $1$.

  Set \(\beta_{n}^- = r_{n,k_n^-} \sqrt{\frac 2 n} \). 
  Lemma~\ref{lem:beta_in_terms_of_alpha} is applicable since
  \(k_n^- - k_n^\star = -\left( \frac n 2 \right)^{\frac 1 4}\ln \left( \frac n 2 \right) + O(1) \),
  so we have
  \begin{equation*}
    \beta_{n}^- = 1 + 4 \left( \frac n 2 \right)^{-\frac 1 4} +  O \left( \frac 1 {\sqrt{n}} \right),
  \end{equation*}
  hence
  \begin{equation*}
    r_{n,k_n^-} =  \left( \frac n 2 \right)^{-\frac 1 2}  + 4 \left( \frac n 2 \right)^{-\frac 3 4} +  O \left( \frac 1 {n} \right).
  \end{equation*}
  After routine algebra, we see that
  \begin{equation*}
    \left( 1 +  \frac{r_{n,k_n^-} \left( \frac n 2 - \sqrt{\frac n 2} - 1 \right)} {1-r_{n,k_n^-}}\right) r_{n,k_n^-} 
    = 
    1 + 8 \left( \frac n 2 \right)^{-\frac 1 4} + O\left( \frac 1  {\sqrt n} \right).
  \end{equation*}
  This shows that \(\left( 1 + k \frac{r_{n,k+1}} {1-r_{n,k+1}}\right) r_{n,k+1} \), which is larger than \(\left( 1 +  \frac{r_{n,k_n^-} \left( \frac n 2 - \sqrt{\frac n 2} - 1 \right)} {1-r_{n,k_n^-}}\right) r_{n,k_n^-}\), will ultimately surpass \(1\).
\end{proof}

\subsubsection{Near Right Region}

We repeat the previous proof to the region just after the mode.

\begin{lemma} There exists an integer \(n_0\) such that 
for all \(n \geq n_0\)
  and \(k\) such that
  \begin{equation*}
     k_n^+ \leq k < \frac n 2 + \ln \left( \frac n 2 \right) \sqrt{\frac n 2},
  \end{equation*}
  where \( k_n^+ =
  \left\lfloor
  \frac n 2 + \left( \frac 1 4 \ln\left( \frac n 2 \right) - \frac 1 2 \right) \sqrt{\frac n 2} + \left( \frac n 2 \right)^{\frac 1 4}\ln \left( \frac n 2 \right)
  \right\rfloor
  \),
  we have \(\SB_{n,k} \geq \SB_{n,k+1}\)
  (see Lem\-ma~\ref{lem:unimodality} for the definition of \(\SB\)).
  \label{lem:central_right}
\end{lemma}

\begin{proof}
  This is basically the same demonstration as Lemma~\ref{lem:central_left}, 
  but where this time we show that
  \( \left( 1 + k \frac{r_{n,k}} {1-r_{n,k}}\right)  r_{n,k} \leq 1\)
  ultimately.

  Indeed, we have
  \begin{equation*}
    \left( 1 + k \frac{r_{n,k}} {1-r_{n,k}}\right)  r_{n,k} 
    \leq
    \left( 1 +  \frac{r_{n,k_n^+}} {1-r_{n,k_n^+}} \left(  \frac n 2 + \ln \left( \frac n 2 \right) \sqrt{\frac n 2}\right)\right)  r_{n,k_n^+}. 
  \end{equation*}
  We deduce from Lemma~\ref{lem:beta_in_terms_of_alpha}
  \begin{equation*}
    r_{n,k_n^+} =  \left( \frac n 2 \right)^{-\frac 1 2}  - 4 \left( \frac n 2 \right)^{-\frac 3 4} +  O \left( \frac 1 {n} \right)
  \end{equation*}
  and so
  \begin{equation*}
    \left( 1 +  \frac{r_{n,k_n^+}} {1-r_{n,k_n^+}} \left(  \frac n 2 + \ln \left( \frac n 2 \right) \sqrt{\frac n 2}\right)\right)  r_{n,k_n^+}
    =
    1 - 8 \left( \frac n 2 \right)^{-\frac 1 4} + \widetilde O\left( \frac 1  {\sqrt n} \right).
  \end{equation*}
  We conclude by Lemma~\ref{lem:unimodality}.
\end{proof}

\subsubsection{Far Right Region}

In the final region, we employ the same strategy as before
but with coarser bounds.

\begin{lemma} There exists an integer \(n_0\) such that 
  for all \(n \geq n_0\) 
  and all \(k\) with
  \begin{equation*}
     \frac n 2 + \ln \left( \frac n 2 \right) \sqrt{\frac n 2} \leq k \leq n - 1 ,
  \end{equation*}
  we have \(\SB_{n,k} \geq \SB_{n,k+1}\)
  (see Lem\-ma~\ref{lem:unimodality} for the definition of \(\SB\)).
  \label{lem:far_right}
\end{lemma}

\begin{proof} 
  Let \(\widetilde {k_n} = \lfloor \frac n 2 + \ln \left( \frac n 2 \right) \sqrt{\frac n 2}\rfloor\).
  By Corollary~\ref{cor:beta_in_terms_of_alpha},
  \(
  r_{n,\widetilde {k_n}} \sim \frac{1} 4 \left( \frac n 2 \right)^{-1/2}.
    \)

  Consequently,
  \begin{equation*}
    \left( 1 + k \frac{r_{n,k}} {1-r_{n,k}}\right)  r_{n,k} 
    \leq
    \left( 1 + n \frac{r_{n,\widetilde {k_n}}} {1-r_{n,\widetilde {k_n}}}  \right)  r_{n,\widetilde {k_n}} \to \frac 1 4,
  \end{equation*}
  which proves that  \( \left( 1 + k \frac{r_{n,k}} {1-r_{n,k}}\right)  r_{n,k} \)
  is less than \(1\) eventually.
\end{proof}

\subsubsection{Conclusion}

We can now establish a uniform bound for \(g_{n,k}\) when $k$ is at least at distance \((\frac n 2)^{\frac 1 4} \ln\frac n 2\) from the mode.

\begin{lemma}
  Let \(k\) be such that \(|k - k_n^\star| > (\frac n 2)^{\frac 1 4} \ln\frac n 2\)  where \(k_n^\star = \frac n 2 + \left( \frac 1 4 \ln\left( \frac n 2 \right) - \frac 1 2 \right) \sqrt{\frac n 2}\).
  Then we have the uniform bound
  \[ g_{n,k} \leq \SB_{n,k_n^-} + \SB_{n,k_n^+} = g_n \times O\left( n^{-4}  \right)\]
 where
\(k_n^- =  \left\lceil  k_n^\star -(\frac n 2)^{\frac 1 4} \ln\frac n 2 \right\rceil \) 
and
\(k_n^+ =  \left\lfloor k_n^\star +(\frac n 2)^{\frac 1 4} \ln\frac n 2 \right\rfloor\) .
  \label{lem:far_from_mode}
\end{lemma}
\begin{proof} Let us show the inequality \( g_{n,k} \leq \SB_{n,k_n^-} + \SB_{n,k_n^+}\) depending on the value of \(k\):
  \begin{itemize}
  \item If \(\frac n 2 - \sqrt{\frac n 2} - 1 \leq k \leq k_n^-  \), 
  then \( g_{n,k} \leq \SB_{n,k} \leq \SB_{n,k_n^-} \) by Lemma~\ref{lem:central_left}.
  \item If \(k \leq \frac n 2 - \sqrt{\frac n 2} \), 
  then  \( g_{n,k} \leq g_{n, \left\lfloor \frac n 2 - \sqrt{\frac n 2} \right\rfloor}\) by Lemma~\ref{lem:farleft}, 
  and so by the previous point,  \( g_{n,k} \leq \SB_{n,k_n^-} \).
  \item If \(k \geq k_n^+\), then \(g_{n,k} \leq \SB_{n,k}  \leq \SB_{n,k_n^+}\)
  since by Lemmas~\ref{lem:central_right} and~\ref{lem:far_right}, \(\SB_{n,k}\) is decreasing in \(k\) over that range.
\end{itemize}
Let us now estimate \(\SB_{n,k_n^\pm}\) compared to $g_n$.
\begin{align*}
  \SB_{n,k_n^\pm} = 
  \frac{ G_{k_n^\pm}\left( r_{n,k_n^\pm} \right)
  }
  { {r_{n,k_n^\pm}}^n }
  & \sim \sqrt{ \pi \ln\left( \frac n 2 \right)}\left( \frac n 2 \right)^{\frac 1 4}
  g_{n,k_n^\pm} 
  & \text{by Proposition~\ref{prop:gnk_asympt},} \\
  & \sim 2 e^{
    - \frac 4 {\ln\left( \frac n 2 \right)} \sqrt{\frac 2 n}
     \left( k_n^\pm - k_n^\star \right)^2 
  } \sum_{\substack{k'  \text{ such that}\\ {|k' - k_n^\star| \leq (\frac n 2)^{\frac 1 4} \ln\frac n 2}}} g_{n,k'}
  & \begin{aligned}
    \text{by Lemma~\ref{lem:gnk_in_terms_of_diff_alpha}} 
    \\ \text{ and Proposition~\ref{prop:central_sum},}
  \end{aligned} \\
  & \sim 2 \left( \frac n 2 \right)^{-4} \sum_{\substack{k'  \text{ such that}\\ {|k'- k_n^\star| \leq (\frac n 2)^{\frac 1 4} \ln\frac n 2}}} g_{n,k'}.
\end{align*}
Since \(\sum_{\substack{k'  \text{ such that}\\ {|k' - k_n^\star| \leq (\frac n 2)^{\frac 1 4} \ln\frac n 2}}} g_{n,k'} \leq \sum_{k'=2}^{n}g_{n,k'} = g_n\), we obtain our uniform bound for \(g_{n,k}\).
\end{proof}

We have reached the end of our marathon proof.

\begin{proof}[Proof of Theorem~\ref{thm:equivalent}]
By Lemma~\ref{lem:far_from_mode},
\begin{align*}
  g_n = \sum_{k=2}^{n} g_{n,k} & = 
    \sum_{\substack{k  \text{ such that}\\ {|k - k_n^\star| > (\frac n 2)^{\frac 1 4} \ln\frac n 2}}} g_{n,k}
    +
  \sum_{\substack{k  \text{ such that}\\ {|k - k_n^\star| \leq (\frac n 2)^{\frac 1 4} \ln\frac n 2}}} g_{n,k}
  \\
  & = g_n \times O \left( n^{-3} \right)
  +
  \sum_{\substack{k  \text{ such that}\\ {|k - k_n^\star| \leq (\frac n 2)^{\frac 1 4} \ln\frac n 2}}} g_{n,k}.
\end{align*}
Proposition~\ref{prop:central_sum} yields an estimate for the sum, and thus concludes the proof. \label{end:proof}
\end{proof}

\section{Distribution of Some Relevant Parameters}
\label{s:limit_law}

Beyond an asymptotic estimate of the total number of \gitgraph{}s,
we investigate the asymptotic distribution of several structural parameters.
We establish limit theorems for three quantities in a uniform \gitgraph{}:
the number of \main{}s,
the number of free commits (both asymptotically Gaussian),
and the so-called width (for which we compute the expectation)
\subsection{Number of Commits on the Main Branch}

In Section~\ref{sec:asymptotic},
we estimated the total number of \gitgraph{}s of size~$n$ by analyzing~$g_{n,k}$ 
%, where \(k\) refers to the number of \main{}s,
and then summing over~$k$.
This approach revealed Gaussian behavior for the number~$K_n$ of \main{}s in a uniform Git graph of size~$n$
(cf. Figure~\ref{fig:plan}).

We now formalize this as a convergence in distribution and a local limit law.

\begin{theorem}
  Let~$K_n$ denote the number of \main{}s of a uniform \gitgraph{} of size~$n$.

  The random variable~$K_n$ 
  is asymptotically normal with mean $k_n^* = \frac n 2 + \left(\frac 1 4 \ln\left( \frac n 2  \right)- \frac 1 2\right)\sqrt{\frac n 2}$ 
  and standard deviation $\sigma_n = \sqrt{\frac 1 8 \ln\left( \frac n 2  \right) }\left(\frac n 2\right)^{1/4}$. More precisely, two kinds of convergence hold:
  \begin{itemize}
    \item \textbf{Convergence in distribution.} We have
      \begin{equation*}
     \frac{K_n - k_n^\star}{\sigma_n} \to \mathcal N(0, 1).
     \end{equation*}
    \item \textbf{Local limit law} (in the sense of~\cite[Chapter IX.9]{FS2009}).
    There exists a \(\widetilde O (1)\) sequence \(\delta_n\) such that for all~$n \geq 2$
  \begin{equation*}
    \sup_{x \in \mathbb R} \left| \sigma_n \mathbb P(K_n = \lfloor k_n^\star + x \sigma_n \rfloor ) - \frac{1}{\sqrt{2\pi}} e^{-\frac{x^2}{2}}\right| \leq \frac{\delta_n}{n^{1/4}}.
  \end{equation*}    
  \end{itemize}
  \label{thm:expected_main}
\end{theorem}

\begin{proof}  These convergences are a byproduct of our proof of Theorem~\ref{thm:equivalent}.

  \paragraph*{Convergence in distribution.}
 
    For~$x \in \mathbb R$, we have 
  \begin{equation*}
    \mathbb P\left(\frac{K_n - k_n^\star}{\sigma_n} \leq x\right)
    = \sum_{k \leq k_n^\star + x \sigma_n} \frac{g_{n,k}}{g_n}
    = \sum_{k=k_n^-}^{\lfloor k_n^\star + x \sigma_n\rfloor} \frac{g_{n,k}}{g_n} + O(n^{-3})
  \end{equation*}
  where Lemma~\ref{lem:far_from_mode} permits neglecting terms below \(k_n^-\).

  When $n$ is sufficiently large, all integers \(k\) between \(k_n^-\) and \(\lfloor k_n^\star + x \sigma_n\rfloor\) satisfy the inequality
\(  \left| k - k_n^\star \right| \leq \left( \frac n 2 \right)^{1/4} \ln \left( \frac n 2 \right) \).
  So we can apply
  Lemma~\ref{lem:gnk_in_terms_of_diff_alpha} and Theorem~\ref{thm:equivalent}
 to deduce that for any~$k$ between $k_n^-$ and \(\lfloor k_n^\star + x \sigma_n\rfloor \),
  \begin{align}
    \frac{g_{n,k}}{g_n} & \sim 
    2
    \frac {\left( \frac n 2 \right)^{-1/4}}
    {\sqrt{\pi \ln \left( \frac n 2 \right)}}
   \exp
   \left(  
   {- \frac 4 {\ln\left( \frac n 2 \right)} 
      \sqrt{ \frac 2 n}
      \left( k - k_n^\star \right)^2} 
    \right)  
      \left( 1 + \widetilde O \left( n^{-1/4} \right) \right) 
      \notag
      \\
    & = \frac{1}{\sqrt{2\pi} \sigma_n} \
    \exp
   \left(  {-\frac{(k - k_n^\star)^2}{2 {\sigma_n}^2}}
   \right)  
    \left(1 + \widetilde O \left(n^{-1/4}\right)\right)
    \label{eq:asympt_proba}
  \end{align}
 uniformly with respect to $k$.
  
  Let $f_n(t)$ denote $\exp\left(- \frac{(t - k_n^\star)^2}{2 \sigma_n^2}\right)$.
  The Euler-Maclaurin formula gives
  \begin{equation*}
    \sum_{k=k_n^-}^{\mathclap{\lfloor k_n^\star + x \sigma_n\rfloor}} f_n(k)
    =
    \frac{f_n(k_n^-) + f_n(\lfloor k_n^\star + x \sigma_n\rfloor)}2
    ~+~ \int\limits_{k_n^-}^{\mathclap{\lfloor k_n^\star + x \sigma_n \rfloor}} f_n(t) dt
    ~+~ \int\limits_{k_n^-}^{\mathclap{\lfloor k_n^\star + x \sigma_n \rfloor}} f_n'(t)\left(t - \lfloor t \rfloor - \frac 1 2\right) dt
  \end{equation*}
  The quantities
  $f_n(k_n^-)$ and ~$f_n(\lfloor k_n^\star + x \sigma_n\rfloor)$ are bounded by \(1\).
  In addition, since~$f_n'(t) = - \frac{t - k_n^\star}{\sigma_n^2} f_n(t)$,  the bound ~$f_n'(t) = \softO(n^{-1/4}) f_n(t)$ holds uniformly over the integration domain,
  making the second integral $\softO(n^{1/4})$ times smaller than the first one.
  For this integral, substitution $s = \frac{t - k_n^\star}{\sigma_n}$ yields
  \begin{equation*}
    \int_{k_n^-}^{\lfloor k_n^\star + x \sigma_n \rfloor} f_n(t) dt
    = \sigma_n \int_{\frac{k_n^- - k_n^\star}{\sigma_n}}^{\frac{\lfloor k_n^\star + 
    x \sigma_n \rfloor - k_n^\star}{\sigma_n}} e^{-\frac{s^2}{2}} ds
    = \sigma_n \int_{-\infty}^x e^{-\frac{s^2}{2}} ds + O(1). 
  \end{equation*}
  Combining these estimates, \(\mathbb P\left(\frac{K_n - k_n^\star}{\sigma_n} \leq x\right)\)
  converges to 
  \(\frac 1 {\sqrt{2 \pi}}
  \int_{-\infty}^x e^{-\frac{s^2}{2}} ds 
  \),
  the CDF of $\mathcal{N}(0,1)$.

    \paragraph*{Local limit law.}

  Let~$x \in \mathbb R$ and set~$k_n(x) = \lfloor k_n^\star + x\sigma_n\rfloor$.
  
  If \(|k_n(x) - k_n^\star| > \ln\left( \frac n 2 \right) \left( \frac n 2 \right)^{1/4}\), 
  then
  $|x| > \sqrt{8 \ln\left( \frac n 2 \right)} - 1$, so
  \(\frac{1}{\sqrt{2\pi}} e^{-\frac{x^2}{2}} \) decays as 
  \(\widetilde O\left( n^{-4} \right)\).
  The term  \(\sigma_n \frac{g_{n,k_n(x)}} {g_n}   \) 
  decays similarly
  by virtue of Lemma~\ref{lem:far_from_mode}. 
  So the local limit law holds trivially in this regime.

  Now assume that $|k_n(x) - k_n^\star| \leq \ln\frac n 2 \left( \frac n 2 \right)^{1/4} = \sqrt{8 \ln\frac n 2} \sigma_n$.
  By definition of the floor function, the estimate \(
  k_n(x) = k_n^\star + x \sigma_n + O(1)
  \) holds uniformly.
  Thus, 
  \begin{equation*}
    \left( \frac{k_n(x) - k_n^\star}{\sigma_n} \right)^2
  = x^2 + \widetilde O(n^{-1/4}).
  \end{equation*}
  By Equation
  \eqref{eq:asympt_proba},  
  \begin{equation*}
    \sigma_n \frac{g_{n,k_n(x)}}{g_n}
    = \frac 1 {\sqrt{2\pi}} e^{-\frac {x^2} 2}+ \widetilde O \left( n ^{-1/4} \right)
  \end{equation*}
  holds uniformly, which establishes the local limit law.
\end{proof}

\subsection{Number of Free Commits}

\begin{theorem}\label{thm:freecommits}
  Let~$F_n$ be the number of free commits in a uniform Git graph of size~$n$ (cf. Definition~\ref{def:free}).

  The random variable~$F_n$ is asymptotically normal. More precisely, 
\(
    \dfrac{F_n - \mathbb E \left( F_n \right)  }
    { \sqrt{  \mathbb V \left( F_n \right) } }
 \)
  converges in distribution to  \(\mathcal N(0, 1)\),
  where \[\mathbb E \left( F_n \right) \sim  \frac {\ln\left( \frac n 2 \right)}{2} \sqrt{\frac n 2} + \widetilde O \left( n^{\frac 1 4} \right)
  \quad \text{ and } \quad
  \mathbb V \left( F_n \right) \sim  \left( \frac { \ln\left( \frac n 2 \right)} 2 - 1 \right) \sqrt{\frac n 2}   +\widetilde O \left( n^{\frac 1 4} \right).\footnotemark\]
\end{theorem}
\footnotetext{
  Even if we did not prove it, the true error terms should be \(\widetilde O (1)\).
}

%Compute the number of pages
\FPeval{\result}{clip(\getpagerefnumber{end:proof}-\getpagerefnumber{start:proof})}

\begin{proof}
The idea here is to apply the Generalized quasi-powers theorem~\cite[Theorem IX.13]{FS2009}
by incorporating to the generating function \(G_k\) a new variable \(v\) that tracks the number of free commits. 

We should revisit the demonstration of Theorem~\ref{thm:equivalent} with this new variable.
However, rather than reproducing the \result{} pages of proofs (a prospect we shall describe as "inadvisable"),
we summarize the main results already obtained while keeping track of the additional variable \(v\) throughout.
All the saddle-point estimation and asymptotic machinery work similarly.

First of all, let \(G_k(z,v)\) be the generating function of \gitgraph{}s with \(k\) \main{}s,
where~\(z\) still refers to the size, and \(v\) counts the number of free commits.
Proposition~\ref{prop:Gkz} extends to 
\[ G_k(z,v)
  = z^k \, \prod_{j=0}^{k-1} \left(v + \frac{jz}{1 - z}\right)
  = \frac{z^{2k}}{(1-z)^k}
   \frac{\Gamma\left( \frac{v(1-z)}{z} + k \right)}{\Gamma\left( \frac{v(1-z)}{z} \right)}. \]
For remainder, the \(\left(v + \dfrac{jz}{1 - z}\right)\) factor in the first expression
translates the possibility for the \({(j+1)}\)th \main{} to be either free (contribution: \(v\)) or the merging point of a feature branch (contribution: \(\frac{jz}{1 - z}\)). Consider $v$ as a complex number close to $1$.

From there, we define \(r_n\) as the unique solution of the saddle-point equation
\(r_n G_{k_n}'(r_n,v) = n\, G_{k_n}(r_n,v)\),
 and \(\beta_n\) as \(r_n \sqrt{\frac n 2}\). 
Following Proposition~\ref{prop:asympt_saddle_bound},
we prove that if $k_n - \frac n 2 = \widetilde O\left( {\sqrt{n}} \right)$, then 
\begin{equation*}
      r_n^{-n}G_{k_n}(r_n,v) \underset{n\to\infty}\sim
      \left(\frac{n}{2e}\right)^{\frac{n}{2}}
      \exp\left(\lambda_n(\beta_n,v) \sqrt{\frac{n}{2}} + \mu_n(\beta_n,v)\right),
    \end{equation*}
where \(\lambda_n(x,v)\) is a function with a local maximum at $x=1$, for any $v$. 
The expressions of \(\lambda_n\) and \(\mu_n\) being quite involved, let us make this asymptotic estimate explicit while $k_n$ is near the mode.

Let \(k_n^\star\) be 
\(\frac n 2 + \left( \frac v 4 \ln\left( \frac n 2 \right) - \frac 1 2 - \frac v 2  \ln(v) \right)\sqrt{\frac n 2}\), 
the approximate location of the mode. 
If \(k_n\) satisfies
\(\left| k_n - k_n^\star \right| < \left( \frac n 2 \right)^{1/4} \ln \left( \frac n 2 \right)  \), then the estimate (in line with Lemma~\ref{lem:gnk_in_terms_of_diff_alpha})
\begin{equation*}
  r_n^{-n}G_{k_n}(r_n,v)
  \sim
  \left( \frac n {2e} \right)^{\frac n 2}
  \exp\left(
        \begin{aligned}
          \left( \frac v 2 \ln\left( \frac n 2 \right)
           + v + 1 - v\ln\left( v \right)
          \right) 
          \sqrt{\frac n 2} \\
          - \frac 4 
          { v \ln\left( \frac n 2 \right) + 2 -2v -2v\ln(v)} 
          \left( k_n - k_n^\star \right)^2 \sqrt{\frac 2 n} \\
          + \frac {v^2}{32} \ln\left( \frac n 2 \right)^2
          - \frac{2+3v-2v^2 + v^2\ln(v)} 8\ln\left( \frac n 2 \right) \\
          + \frac 1 8 - \frac v 2 + \frac {v^2} 2 + \frac{v^2} 8 \ln(v)^2
          + \frac{2 + 3v - 2 v^2} 4 \ln(v)
        \end{aligned}
        \right)
\end{equation*}
holds with a (uniform) error factor of \(1 + \widetilde O\left( n^{-\frac 1 4} \right)\).

Then, we can evaluate \(g_{n,k_n}\) when \(k_n\) is within this range by Cauchy's integral formula. 
Mirroring Proposition~\ref{prop:gnk_asympt}, we can establish that
\begin{equation*}
  g_{n,k_n} = \frac{1}{\sqrt{ 2 \pi \Delta_n(v)  }} r_n^{-n}G_{k_n}(r_n,v)
  \left( 1 + \widetilde O \left( \frac 1 {\sqrt n}\right) \right)
\end{equation*}
uniformly in $k_n$ and $v$,
where 
\begin{equation*}
  \Delta_n(v) = \left(\frac{v}{2\beta_n} \ln\left( \frac n 2 \right)+ v \frac{\ln\beta_n}{\beta_n} - \frac{v}{\beta_n} + \beta_n -\frac { v \ln(v) } {\beta_n} \right)
  \sqrt{\frac n 2}.
\end{equation*}

Now, we use the Euler-Maclaurin formula to sum the \(g_{n,k}\)'s as with Proposition~\ref{prop:central_sum},
then proceed to bound the tails. We arrive at the generalization of Theorem~\ref{thm:equivalent} 
  \begin{equation*}
    g_n \underset{n \to \infty}\sim
    \frac 1 2
      \left( \frac n {2e} \right)^{\frac n 2}
  \exp\left(
        \begin{aligned}
          \left( \frac v 2 \ln\left( \frac n 2 \right)
           + v + 1 - v\ln\left( v \right)
          \right) 
          \sqrt{\frac n 2} \\
          + \frac {v^2}{32} \ln\left( \frac n 2 \right)^2
          - \frac{2+3v-2v^2 + v^2\ln(v)} 8\ln\left( \frac n 2 \right) \\
          + \frac 1 8 - \frac v 2 + \frac {v^2} 2 + \frac{v^2} 8 \ln(v)^2
          + \frac{2 + 3v - 2 v^2} 4 \ln(v)
        \end{aligned}
        \right)
        \left( 1 + \widetilde{O}\left( n^{-\frac 1 4} \right) \right),
  \end{equation*}
where the estimate is uniform in $v$. 

Using the notation of~\cite[Theorem IX.13]{FS2009}, we deduce that 
the probability generating function of \(F_n\) admits a representation of the form \(\exp\left( h_n(v) \right)  \left( 1 +  \widetilde{O}\left( n^{-\frac 1 4} \right)  \right)\)
uniformly in $v$, where
\begin{multline}
  h_n(v) = 
  \frac{
    (v-1) \ln\left( \frac n 2 \right) 
    + 2v - 2 - 2 v \ln(v)
  } 2
  \sqrt{\frac n 2}
  +
  \frac{v^2 - 1}{32}
  \ln\left( \frac n 2 \right)^2 \\
  + \frac {
  1 - 3v + 2v^2 - v^2 \ln(v)
  } 8
  \ln\left( \frac n 2 \right) 
  - \frac v 2 + \frac {v^2} 2 + \frac{v^2} 8 \ln(v)^2
  + \frac{2 + 3v - 2 v^2} 4 \ln(v).
  \label{eq:representation}
\end{multline}
The hypotheses of the Generalized quasi-powers theorem are clearly satisfied, 
implying that the random variable \(
\dfrac{F_n - h_n'(1)}
{\sqrt{h_n'(1) + h_n''(1)}}
\) 
converges a Gaussian of mean and variance \(1\),
where \(h_n'(1) = \frac 1 2 \ln \left( \frac n 2 \right) \sqrt{\frac n 2} + \widetilde O \left( n^{\frac 1 4} \right) \)
and
\(h_n''(1) = \left( \frac 1 2 \ln \left( \frac n 2 \right) - 1 \right) \sqrt{\frac n 2} + \widetilde O \left( n^{\frac 1 4} \right) \).
\end{proof}

\subsection{Width of Uniform Git Graphs}

\newcommand\randg{\ensuremath{\boldsymbol{G}_n}}
\newcommand\randgg{\ensuremath{ \overline{\boldsymbol{ G}_n  } } }

In his master thesis~\cite{stage-lucas}, Blondel introduced a parameter termed \emph{width} to study real \gitgraph{}s from GitHub repositories, 
and investigated its behavior.

\begin{definition}[Width]
  Let~\(G\) be a \gitgraph{} with \(k\) \main{}s.
  Label its \main{}s \(1,2,\dots,k\) from oldest to newest.
The width~$W(G, j)$ at position~$j \geq 1$ is defined as the number of branches
that fork off from a \main{} with label \(\leq j\) and merge into a commit with label \(>j\).
%In particular, if~$j$ is larger than the number of main commits of~$g$, we have~$W(g, j) = 0$.
\end{definition}

Blondel's experimental study compared the width profiles of uniform \gitgraph{}s with those of real \gitgraph{}s, revealing substantial differences.
 We prove here what he observed experimentally for our model:

\begin{theorem}
  Let~$\randg$ be a uniform random Git graph of size~$n$ and \(K_n\) be its number of \main{}s.  
  For~$t > 0$, we have
  \begin{equation*}
    \lim_{n \rightarrow + \infty}\mathbb E\left(\frac 1 n W(\randg, \left\lfloor t K_n\right\rfloor)\right) 
      = - \frac t 2 \ln(t).
  \end{equation*}
\end{theorem}

 \begin{figure}[ht!]
\begin{center}
\includegraphics[scale=1]{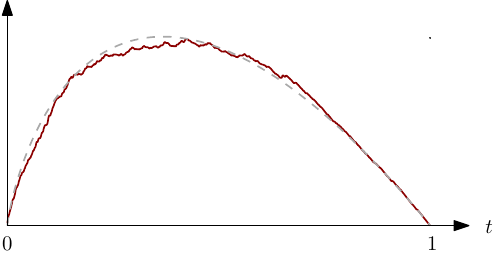}
\end{center}
\vspace{-7pt}
\caption{ Red: a realization of the profile \(W(\randg, \left\lfloor t K_n\right\rfloor)\) where \(\randg\) is a uniform \gitgraph{} of size \(2048\). Dashed: the asymptotic theoretical prediction \(- \frac t 2 \ln t\).}
\label{fig:profile_plot}
\end{figure}

Figure~\ref{fig:profile_plot} provides an evidence of the previous theorem.

\begin{proof}
  We employ a coupling argument to relate the width of an arbitrary random \gitgraph{} to one with no free commits.

  Consider a uniform random Git graph~$\randg$ of size \(n\),
  and denote by $K_n$ its number of \main{}s and \(F_n\) its numbers of free commits.
  Now construct a new \gitgraph~$\randgg$ as follows:
  for each free commit of \(\randg\),
  attach a branch made of a single \feature, 
  with the branch's base chosen uniformly at random from earlier \main{}s.
  Thus,~$\randgg$ has no free commit.

  The width at any given position may increase from \(\randg\) to \(\randgg\) by at most \(F_n\). 
  Consequently, we have for any~$t > 0$
 \begin{equation*}
    W(\randgg, \left\lfloor t K_n\right\rfloor) - F_n  
    \leq
    W(\randg, \left\lfloor t K_n\right\rfloor) 
    \leq 
    W(\randgg,\left\lfloor t K_n\right\rfloor).
  \end{equation*}
  By Theorem~\ref{thm:freecommits},
  we have \(\mathbb E \left( F_n \right) = O\left( \ln (n) \sqrt n \right)\).
  Therefore, if we take expectations, 
  \begin{equation*}
    \mathbb E \left( 
       W(\randg, \left\lfloor t K_n\right\rfloor)
     \right)
    =
    \mathbb E \left(
     W(\randgg, \left\lfloor t K_n\right\rfloor)
    \right)
    +  O\left( \ln (n) \sqrt n \right),
  \end{equation*}
  reducing the problem to analyzing the width of~\randgg.

By linearity of expectation, the expected width in \randgg at position \(\left\lfloor t K_n\right\rfloor\)
equals the sum over all branches of the probability that each branch crosses position
\(\left\lfloor t K_n\right\rfloor\):
\begin{equation*}
     \mathbb E \bigg( 
     W(\randgg, \left\lfloor t K_n\right\rfloor)
    \bigg)= 
    \sum_{k \geq 2} 
    \mathbb P \left( 
       K_n = k \right)
       \sum_{j=\left\lfloor t k\right\rfloor + 1}^k
    \mathbb P\left( \begin{aligned}
      \text{branch merging at \main{} } j \\
      \text{forks from a \main} \leq \left\lfloor t k\right\rfloor
    \end{aligned}
      \right)
\end{equation*}
 For a branch merging at \main{}~\(j\), 
 its base is a uniform \main{} among $\{1, \ldots, j-1\}$,
so
\begin{equation*}
  \mathbb E \bigg( 
     W(\randgg, \left\lfloor t K_n\right\rfloor)
    \bigg) = 
    \sum_{k \geq 2} 
    \mathbb P \left( 
       K_n = k \right)
     \sum_{j=\left\lfloor t k\right\rfloor + 1}^k
     \frac {\left\lfloor t k\right\rfloor} { j - 1}
    =
     \mathbb E 
    \bigg( 
      \left\lfloor t K_n\right\rfloor
    \left( H_{K_n - 1} - H_{\left\lfloor t K_n\right\rfloor} \right)
     \bigg).
\end{equation*}
where $H_n = \sum_{j=1}^n \frac{1}{j}$ is the $n$-th harmonic number.

  Set \[
  f_n(x) :
  = 
  \frac {  \left\lfloor t \,n \,x\right\rfloor}  n
   \left( H_{\lfloor nx \rfloor - 1} - H_{\left\lfloor t n x\right\rfloor} \right)
  \]
  so that
  \begin{equation*}
    \frac 1 n \,
    \mathbb E \bigg( 
     W(\randgg, \left\lfloor t K_n\right\rfloor)
    \bigg)
    = 
    \mathbb E 
    \left( 
    f_n\left(\frac {K_n}{n}
    \right)
     \right).
  \end{equation*}
  By Theorem~\ref{thm:expected_main}, \(K_n / n\) converges in distribution to \(1/2\).
  Moreover, \(f_n(x)\) is bounded for \(n \geq 2\) and \(x \in [0,1]\),\footnote{
    Using an  integral test for convergence,
    we prove that \(f_n(x) <  2 - \ln(t) \).
  }
  and converges over a neighborhood of $1/2$
  to a continuous function:
  \begin{equation*}
    f_n(x) 
    \underset{n \to \infty}\sim
   \left(t x + O\left( \frac 1 n \right)\right) 
   \left( \ln(nx) + \gamma - \ln(nxt) - \gamma  + O\left( \frac 1 n \right) \right)
    \rightarrow -t\,x\, \ln(t).
  \end{equation*}
  This guarantees us\footnote{We use the easy-to-prove property
  that if $X_n$ converges in law to a constant \(\ell\),
  and $f_n$ are functions, uniformly bounded, converging pointwise to a function \(f\) continuous at \(\ell\),
  then \(\lim_{n \to \infty}\mathbb E \left( f_n \left( X_n \right) \right) = f(\ell)\).}
   that \(\mathbb E 
    \left( 
    f_n\left(\frac {K_n}{n}
    \right)
     \right)\) converges to \(-\frac t 2\, \ln(t)\).
\end{proof}

\section{Random Generation}
\label{s:generation}

As discussed in the introduction,
random generators of \gitgraph{}s
should be valuable for property-based testing
and benchmarking of VCS algorithms.
In this section, we study two families of random generators:
uniform samplers and Boltzmann samplers.

%We will quickly treat uniform samplers since two factors limit its practical scope: 

While a conference paper by the same authors~\cite{short-version}
already addressed random generation of \gitgraph{}s,
the algorithms and analyses have been improved in the present work.

\subsection{Uniform Random Sampler}

The enumeration formula of Corollary~\ref{cor:formula}
naturally suggests a \emph{uniform random sampler}, 
i.e., an algorithm that, given an integer \(n\)
outputs any \gitgraph{} of size \(n\) with probability \(\dfrac{1}{g_n}\).

\begin{figure}[ht!]
\begin{center}
\includegraphics[width=0.45\textwidth]{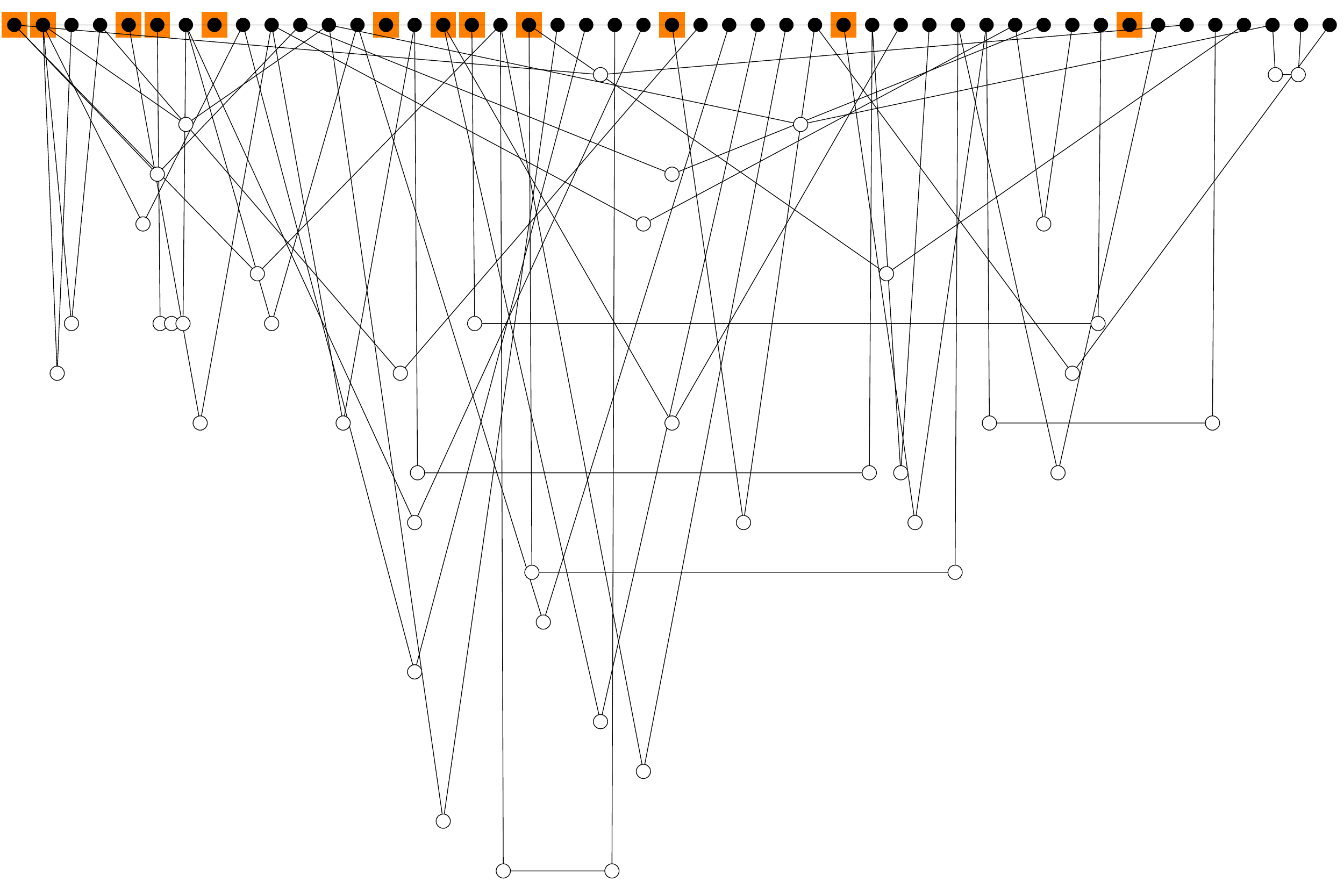}
\hfill
\includegraphics[width=0.45\textwidth]{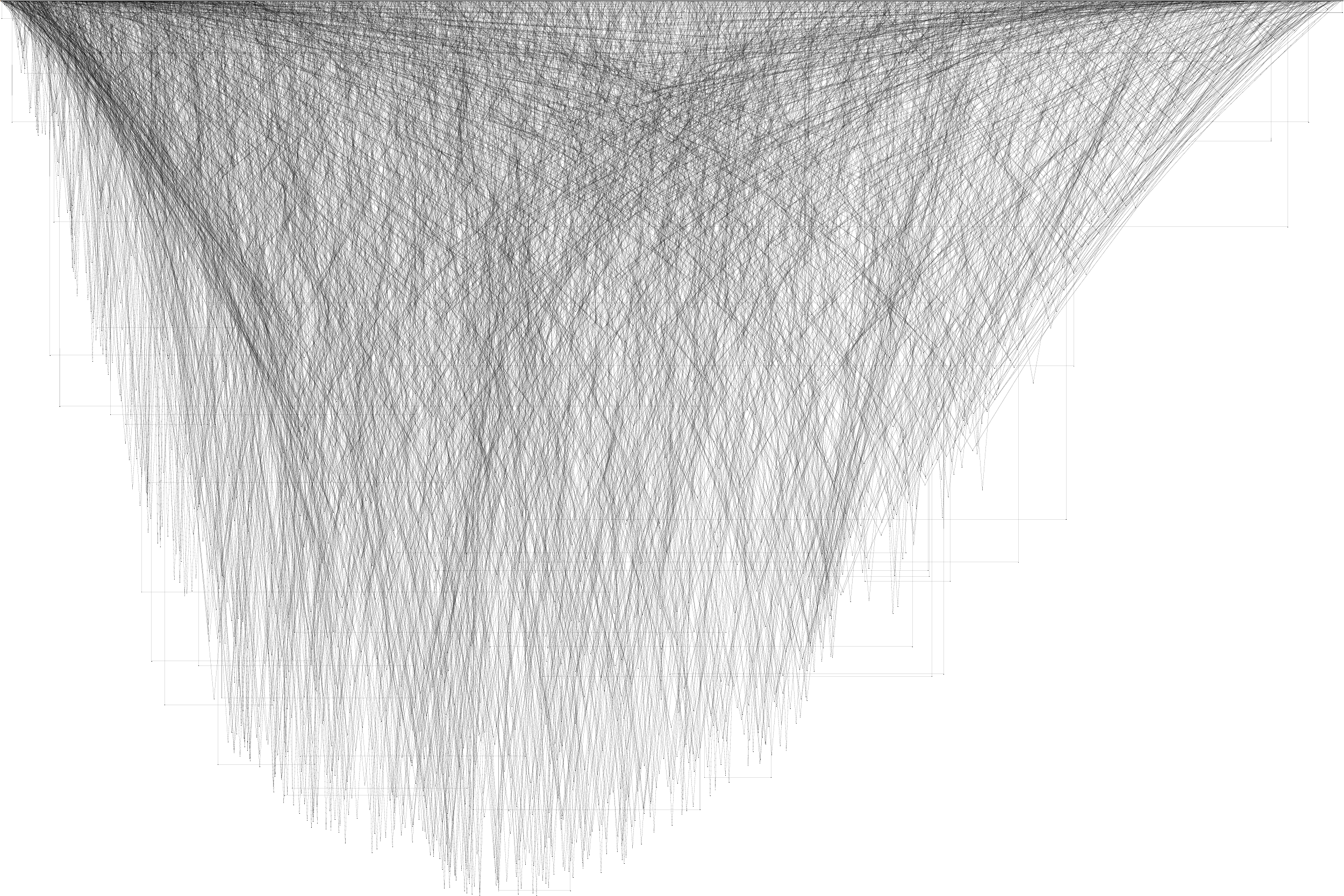}
\end{center}
\vspace{-7pt}
\caption{ Uniform \gitgraph{}s of sizes $90$ and $4000$. }
\label{fig:uniform}
\end{figure}

The key observation is that the probability
that a uniform \gitgraph{} of size \(n\)
have \(k\) \main{}s and \(f\) free commits is given by
\(\dfrac{{k \brack f} 
{n-k-1 \choose k-f-1}}{g_n}\).
The algorithm then proceeds in three steps:
\begin{enumerate}
  \item we precompute Stirling numbers of the first kind 
  and the probability distribution over pairs \( (k,f) \);
  \item we sample \(k\) and \(f\) the numbers of \main{}s  and free commits 
  according to this distribution;
  \item we construct the corresponding \gitgraph{}.
\end{enumerate}

\paragraph{Precomputation of probabilities.}

Before running the random sampler \textit{per se}, we compute and store, 
for all the relevant pairs \((k,f)\),
the Stirling numbers of the first kind
\( { k \brack f} \),
as well as the ratios \( {k-1 \brack f-1}/{{k \brack f}}\),
using the well-known recurrence
\begin{equation*}
  {k \brack f} = (k-1) {k-1 \brack f} + {k-1 \brack f-1}.
\end{equation*}
%Despite being a classical number sequence, 
%computing all necessary Stirling numbers dominates the runtime.
After this, we perform a second pass to compute 
the probability \({{k \brack f} 
{n-k-1 \choose k-f-1}}/{g_n}\) of each pair $(k, f)$.

\paragraph*{Sampling \( \boldsymbol{ (k,f) }  \) and constructing the \gitgraph{}.}
The procedure is given by Algorithm~\ref{algo:uniform}.
Two random \gitgraph{}s, drawn with this algorithm, are plotted in Figure~\ref{fig:uniform}.

% \paragraph{Step 2. Sampling parameters \(k\) and \(f\).}

% We sample a pair \((k,f)\) by 
% drawing a random number~\(x\) between \(0\) and \(1\), 
% and by subtracting probabilities  to \(x\)  (iterating through the probability table)
% until \(x\) becomes negative.
% For efficiency purposes, the table has been sorted in decreasing order.

% \paragraph{Step 3. Constructing the \gitgraph{}.}

% Once the pair \((k,f)\) has been fixed,
% we wish to build a \gitgraph{} of size \(n\), with \(k\) \main{}s and \(f\) free commits, uniformly at random.

% We start by determining the positions of the \(f\) free commits among the \(k\) \main{}s.
% To do so, we exploit the bijection with cyclariums (see Proposition~\ref{prop:bijection} and its proof).
% A uniform random cyclarium yields, 
% upon removing white vertex chains,
% a uniform random permutation of size \(k\) with \(f\) cycles.
% The positions of the free commits on the main branch of a random \gitgraph{}
% can then be retrieved
% from the smallest labels within the $f$ cycles of such a random permutation,
% since these two statistics are in correspondence through the bijection of Proposition~\ref{prop:bijection}.

% Algorithmically, we start from a chain of \(k\) black vertices. 
% Then we iterate from right to left over these vertices by considering positions
% \(x = k, k-1, \dots, 1\).
% We mark the commit at position \(x\) as free with probability
% \( { x- 1 \brack f - 1 } / { x \brack f  } \) (these probabilities have been precomputed). 
% We decrement \(f\) each time a commit is marked free.

\begin{algorithm}[h!]
  \caption{Uniform random sampler for \gitgraph{}s\label{algo:uniform}}
  \Input{ an integer \(n\) } \\
  \Outpout{ a \gitgraph{} of size \(n\), drawn uniformly at random }
  \begin{algorithmic}[1]
    \LComment{We start by sampling parameters \(k\) and \(f\). \label{start:samplingkf}}
    \State{ \(x \gets\) random number between \(0\) and \(1\) }
    \State{ \(P_n \gets\) precomputed table \((k,f) \mapsto {{k \brack f} 
{n-k-1 \choose k-f-1}}/{g_n}\), sorted by decreasing probability}

    \ForAll{ \( \left( k,f \right) \) in \(P_n\) } 
      \State{ 
        \(x \gets x - P_n \left[ \left( k,f \right) \right]\)
      }
      \If{ \(x < 0\)}
        \State{save \(k\) and \(f\); \textbf{break}}
      \EndIf
    \EndFor
    \label{end:samplingkf}
    \LComment{Once the pair \((k,f)\) has been fixed, we build a uniform \gitgraph{} with these parameters.\label{start:rest}}
    \State{
      \(G \gets \) branch of \(k\) \main{}s
    }
    \State{
      \( \ell \gets \) integer composition of \(n - k\) into \(k - f\) positive terms, taken uniformly at random\footnotemark \label{line:composition}
    }
    \State{ \(j \gets 0 \)}
    \For{ \(i\) from \(k\) to \(2\)}
      \State{ \(v \gets i\)-th \main{} of \(G\)}
      \If{ \( \left( \text{random number between }0\text{ and }1 \right) 
      <  \dfrac {{i-1 \brack f-1}} {{i \brack f}} \) \label{start:free}}
        \Comment{ \(\dfrac {{i-1 \brack f-1}} {{i \brack f}} \) is precomputed.}
        \LComment{ Case where \(v\) is a free commit}
        \State { \( f \gets f - 1\) \label{end:free}}
      \Else
        \LComment{ Case where a feature branch forks into \(v\)}
        \State { $base$ \( \gets \) random \main{} between the first and the \((i-1)\)-th} \label{line:base}
        \State { attach in \(G\) a feature branch of length \(\ell[j]\) from $base$ to \(v\)}
        \State{ \(j \gets j + 1\)}
      \EndIf
    \EndFor
    \label{end:rest}
    \State{\Return \(G\)}
  \end{algorithmic}
\end{algorithm}
\footnotetext{
  An \emph{integer composition} of \(n\) into \(p\) parts 
  is an ordered \(p\)-tuple of positive integers summing to \(n\).
  To sample this uniformly, we use a classical bijection:
  a composition of \(n\) into \(p\) parts
  corresponds to placing \(n - p\) indistinguishable stars amongst \(p-1\) bars.
  
  The algorithm starts with a sequence of \(p-1\) bars, 
  then inserts the \(n-p\) stars one at a time using Fisher-Yates sampling
  (i.e., it swaps each newly appended star with a random position in the current list).
  Finally, it counts the consecutive stars between bars to obtain the composition.
}

\begin{theorem}
  Algorithm~\ref{algo:uniform} samples a uniformly random \gitgraph{} of size~\(n\)
  with worst-case time complexity \(O(n^2)\)
  but expected time complexity \(O(n)\) (excluded precomputation).
  The precomputation phase requires \( O\left( n^2 \right) \) arithmetic operations
and \(\widetilde{O} \left( n^3 \right)\) bit operations.
\end{theorem}

\begin{proof}
  Let us justify first the correctness of the algorithm, i.e., 
  each \gitgraph{} of size \(n\) is indeed drawn with the same probability.

  The first phase (lines \ref{start:samplingkf}--\ref{end:samplingkf}) which consists in sampling \(k\) and \(f\) is quite standard:
  it enables to recover the correct distribution \({{k \brack f} 
{n-k-1 \choose k-f-1}}/{g_n}\).

  Once \((k,f)\) is drawn,
  the algorithm samples uniformly from all \gitgraph{}s of size \(n\)
  with \(k\) \main{}s and \(f\) free commits  
  (lines \ref{start:rest}--\ref{end:rest}).
  But such a \gitgraph{} is uniquely determined by three components:
  \begin{itemize}
    \item[(i)] the choice of which \(f\) of the \(k\) \main{}s are free;
    \item[(ii)] how the \(n - k\) white vertices are distributed amongst the \(k - f\) feature branches;
    \item[(iii)] the base assignment for each feature branch.
  \end{itemize}
  To justify why lines \ref{start:free}--\ref{end:free} induce the correct distribution for (i),
  we use the bijection between \gitgraph{}s and cyclariums
  (see Proposition~\ref{prop:bijection}).
  Under this bijection, the positions of the free vertices on the main branch
  correspond precisely to the smallest labels within each cycle.
  Thus, sampling component~(i) reduces to sampling a uniform permutation of  \(\{1,\dots,k\}\) with $f$ cycles,
  and recording the minimum elements of each cycle.
  Such a permutation can be sampled inductively using the recurrence relation
  \({k \brack f} = (k-1) {k-1 \brack f} + {k-1 \brack f-1}\): 
  element \(k\) is either inserted into an existing cycle of a uniform permutation of \(\{1,\dots,k-1\}\) with $f$ cycles 
  (which occurs with probability \(  (k-1) {k-1 \brack f} / {k \brack f} \)),
  or it is added as a fixed point to a uniform permutation of \(\{1,\dots,k-1\}\) with $f-1$ cycles
  (which occurs with probability  \(  {k-1 \brack f-1} / {k \brack f} \)).
  Crucially, when an element is added as a fixed point, 
  it remains the smallest element within its cycle during the construction of the permutation.
  This correspondence is exactly what lines \ref{start:free}--\ref{end:free} implement,
  thereby guaranteeing the good distribution for (i).

  It is not hard to see that the distribution for (ii) is uniform, and is handled by Line~\ref{line:composition}.
  The distribution for (iii) is also uniform, once (i) is fixed. This is consistent with Line~\ref{line:base}.

  This establishes correctness; we now turn to time complexity.

  The worst-case time complexity is straightforward. Indeed,
  if we are unlucky during the execution of lines \ref{start:samplingkf}--\ref{end:samplingkf},
  we need to go through all possible pairs \((k,f)\),
  which takes \(O\left( n^2 \right)\) iterations.
  The remaining part of the algorithm can be clearly run in \(O(n)\) time.

  The average-case time complexity is harder to establish formally
  and requires a concentration argument.
  What we want to prove is that if \(K_n\) and \(F_n\) denote the random variables which count \main{}s and free commits,
  then \(\left( K_n, F_n \right)\)
  belongs to the \(n\) most probable pairs \((k,f)\) with probability \(1 - O\left( \frac 1 n \right)\).
  With this assumption, the loop at lines \ref{start:samplingkf}--\ref{end:samplingkf}
  stops before \(n\) iterations with probability \(1 - O\left( \frac 1 n \right)\), and does at most \(n^2\) iterations with probability \(O\left( \frac 1 n \right)\), yielding an \(O(n)\) average time.

  Theorems~\ref{thm:expected_main} and~\ref{thm:freecommits} show that 
  \(K_n\) and \(F_n\) are Gaussian-concentrated around their expectations 
  with deviations   \(\widetilde O \left( n^{1/4} \right)\).
  These theorems corroborate that only \(\widetilde O \left( \sqrt n \right)\) values
  of pairs \((k,f)\) are likely to be drawn, but they do not yield the required \(1 - O\left( \frac 1 n \right)\) probability.
  We must appeal to large deviation principles: 
  the tail probability of \(K_n\) can be bounded by 
  Lemma~\ref{lem:far_from_mode},
  and the tail probability of \(F_n\) can be handled by the large deviation variant of the Quasi-Power Theorem~\cite[Theorem ~IX.15]{FS2009}, applied to Equation~\eqref{eq:representation}  (it requires some work). 
  Eventually we prove that 
  \begin{multline*}
    \mathbb P \left( \left| K_n - \mathbb E \left( K_n \right) \right| < \sqrt n \text{ and }  \left| F_n - \mathbb E \left( F_n \right) \right| < \sqrt n  \right) 
    \geq \\ 1 - 
     P \left( \left| K_n - \mathbb E \left( K_n \right) \right| \geq \sqrt n \right) 
     - P \left( \left| F_n - \mathbb E \left( F_n \right) \right| \geq \sqrt n\right) 
     = 1 - o \left( \frac 1 n \right),
  \end{multline*}
  which suffices to establish the announced \(O(n)\) average-case time complexity.

  Finally, let us address the precomputation cost.
  The table of Stirling numbers is computed via the recurrence in \(O(n^2)\) arithmetic operations.
  Since these numbers have \(\Theta\left( n \log n \right)\) bits,
  the total costs \(\Theta\left( n^3 \log n \right)\)  operations.
\end{proof}

This precomputation step is the computational bottleneck of the uniform sampler.
It limits the practical ranges to sizes below \(5000\) on standard laptops.
If we wish to sample \gitgraph{}s with a million vertices,
we have to turn to Boltzmann sampling.

\subsection{Boltzmann Sampler}

To address the limitations of the uniform sampler, 
we introduce a probabilistic model for random \gitgraph{}s
that leverages the simple form of the generating function \(\widetilde G \left( z,u \right)\) (see Theorem~\ref{theo:Gtilde}). 
This model not only simplifies sampling,
but also enables precise control over the number of \main{}s.

\begin{definition}[\labelled{} distribution]
  Given fixed positive parameters \(z\) and \(u\) (lying in the convergence domain of \(\widetilde G\)),
  the \emph{\labelled\ distribution} assigns to a \gitgraph\ of size $n$ with $k$ black vertices 
  the probability\footnote{This is indeed a probability distribution, as the probabilities sum to \(1\) by Theorem~\ref{theo:Gtilde}.}
  \begin{equation*}
    \dfrac{z^n \, u^k   }{k!\, \widetilde G(z,u)},
  \end{equation*}
  where 
  \[ \widetilde G (z,u) = {\left(1 - \frac{z^2 u}{1 - z}\right)}^{- \frac{1 - z}{z}}.
  \]
\end{definition}

The \labelled{} distribution is a multivariate Boltzmann model (exponential in~$u$ and ordinary in~$z$). 
Samplers based on this distribution belongs to the category of \emph{Boltzmann generators}, 
or which a large number of results exist, 
enabling efficient generation of large objects~\cite{DFLS2004}.

Any \gitgraph{} has a positive probability to be drawn, regardless of its size. 
The choice of parameters \(z\) and \(u\) is crucial to maximize
the probability of generating \gitgraph{}s with a specific size and a specific number of \main{}s.
This will be the subject of Subsection~\ref{ss:tune}.

\begin{algorithm}[h!]
  \caption{Boltzmann sampler for \gitgraph{}s\label{algo:boltzmann}}
  \Input{ positive numbers \(z\) and \(u\) in the convergence domain of \(\widetilde G \) } \\
  \Outpout{ a \gitgraph{} drawn under the \labelled{} distribution of  parameters \(z\) and \(u\) }
  \begin{algorithmic}[1]
      \State{$f \gets$ \Call{Poisson}{\( - \frac{1 - z}{z} \ln(1 - \frac{z^2 u}{1 - z})}\)}      
      \Comment{Poisson distribution}      \label{line:startboltz}
      \State{$Y \gets $ empty cyclarium}
      \Comment{cf. Definition~\ref{def:cyclarium}}
      \For{ \(i\) \textbf{from} \(1\) \textbf{to} \(f\)}
      \State{ insert into \(Y\) a cycle of ~\Call{Loga}{$\frac{uz^2}{1 - z}$} black vertices}
      \Comment{Logarithmic series distribution}
      \EndFor{}
      \State{$k \gets$ number of black vertices in \(Y\)}
      \State{randomly label the black vertices of \(Y\) with \(1, 2, \dots, k\)}

       \For{ \(j\) \textbf{from} \(2\) \textbf{to} \(k\)}
        \If{vertex \(j\) in \(C\) is \emph{not} the minimum within its cycle }%
        \State{ Attach to \(j\) a chain of  $(1 + \text{\Call{Geom}{$z$}})$ white vertices}
        \Comment{Geometric distribution}

        \EndIf{}
      \EndFor{}
      \label{line:endboltz}
    \State{\Return \(\phi(Y)\)}
    \Comment{\(\phi\) is the bijection from cyclariums to \gitgraph{}s described by Prop.~\ref{prop:bijection}.}

  \end{algorithmic}
\end{algorithm}

\begin{theorem}
  Let  \(z\) and \(u\)  be the positive numbers in the convergence domain of \(\widetilde G\).
  Algorithm~\ref{algo:boltzmann} samples a uniformly random \gitgraph{} 
  under the \labelled{} distribution of parameters \(z\) and \(u\) with time complexity \(O(n)\),
  where \(n\) is the size of the output graph.
\end{theorem}

\begin{proof}
  Lines~\ref{line:startboltz}-\ref{line:endboltz} of Algorithm~\ref{algo:boltzmann} are a textbook application of Boltzmann generators~\cite{DFLS2004} to sample cyclariums \(Y\) under the distribution
  \begin{equation*}
    \dfrac{z^{\text{size of } Y} \, u^{\text{number of black vertices in }Y}   }
    {\left( \text{number of black vertices in }Y \right)!\, \widetilde G(z,u)}.
  \end{equation*}
  Indeed, at the end of Subsection~\ref{specificition_cyclariums} (page \pageref{specificition_cyclariums}),
  we saw that the class of cyclariums, denoted \(\mathcal Y\), 
  admits the combinatorial specification
  \begin{equation}
    \mathcal Y =  \Set{\C}
    \quad\text{ and }\quad
    \Seq[\neq 0]{\Z} \times \C = \Cyc{\U\Z \Seq[\neq 0]{\Z}}.
    \label{eq:specification}
  \end{equation} 
  In the theory of Boltzmann generators,
  \(\Set{\cdot}\) corresponds to Poisson laws,
  \(\Cyc{\cdot}\) to logarithmic series distributions,
  and \(\Seq{\cdot}\) to geometric distributions.
  Thus, Lines~\ref{line:startboltz}-\ref{line:endboltz} of  Algorithm~\ref{algo:boltzmann} are the straightforward translation of the specification~\eqref{eq:specification}.

  The fact that \(\phi\) is a bijection
  ensures that the distribution of the output \gitgraph{}s
  is the \labelled{} distribution.

  The \(O(n)\) time complexity follows from Remark~\ref{remark:complexity_bijection},
  which justifies that \(\phi\) can be implemented in linear time. 
\end{proof}

In our Python program, we opt for a different (but equivalent) implementation of the Boltzmann sampler.
We directly gene\-rated \gitgraph{}s 
by transferring the parameter distributions from cyclariums to \gitgraph{}s,
thus avoiding the explicit implementation of the bijection~\(\phi\).

\subsection{How to Tune the Boltzmann Sampler}
\label{ss:tune}

The choice of parameters \(z\) and \(u\) is key to controlling the size of generated \gitgraph{}s.
Fortunately, this is straightforward here due the simple form of the generating function \(\widetilde G (z,u)\).

In this subsection, we provide an algorithm to compute these parameters for given target sizes, and we prove a concentration result 
ensuring a good success rate for the rejection algorithm.

\begin{algorithm}[h!]
  \caption{Tuning the Boltzmann sampler for \gitgraph{}s\label{algo:tune}}
  \Input{ positive integers \(n\) and \(k\) such that \(1 \leq k \leq n\)} \\
  \Outpout{two parameters \(z\) and \(u\) such that if \(\gamma\) is drawn under the \labelled{} distribution of parameters \(z\) and \(u\), then \(\mathbb E\left( \text{size of }\gamma \right) = n\) and   \(\mathbb E\left( \text{number of \main{}s in }\gamma \right) = k\) }
  \begin{algorithmic}[1]
      \State{Numerically solve for \(z \in \left( 0,1 \right)\) the equation \((1 - z) \ln\left(\frac{1 - z}{1 + z k - z}\right) - k z^2 + 2 k z + n z^2 - n z = 0\)}
      \label{line:eq}
      \State{\(u \gets \dfrac{k (1 - z)}{z (1 + z k - z)}\)}
      \State{\Return \((z, u)\)}
  \end{algorithmic}
\end{algorithm}

\begin{theorem}
  Algorithm~\ref{algo:tune} computes, 
  for given positive integers \(n\) and \(k\) satisfying \(1 \leq k \leq n\),
  para\-meters \(z\) and \(u\) such that a \gitgraph{} \(\gamma\) sampled under the \labelled{} distribution with these parameters satisfies
  \begin{equation*}
    \mathbb E\left( \text{size of }\gamma \right) = n \quad\text{ and }\quad   \mathbb E\left( \text{number of \main{}s in }\gamma \right) = k.
  \end{equation*}
  In particular,
  \begin{itemize}
    \item there exists a unique positive solution to \(f(z) = 0\)
    where \(f(z)\) is defined as
  \begin{equation*}
    f(z) := (1 - z) \ln\left(\frac{1 - z}{1 + z k - z}\right) - k z^2 + 2 k z + n z^2 - n z 
  \end{equation*}
  (see Line~\ref{line:eq} of Algorithm~\ref{algo:tune}).
  \item the pair \((z,u)\) lies in the convergence domain of \(\widetilde G\).
  \end{itemize}
  \label{theo:tune}
\end{theorem}

\begin{proof}
Under the \labelled{} distribution, we have
  \begin{equation*}
    \mathbb E\left( \text{size of }\gamma \right) 
    = \sum_{\substack{n \geq 0 \\ k \geq 0}} n \dfrac{g_{n,k}\,z^n \, u^k   }{k!\, \widetilde G(z,u)} 
    =  z \frac{\partial \widetilde G}{\partial z} (z,u) \times \frac{1}{\widetilde G(z,u)}
    = z \dfrac{\partial}{\partial z} \left( \ln \left( \widetilde{G} (z,u) \right) \right)
  \end{equation*}
  and similarly,
  \begin{equation*}
    \mathbb E\left( \text{number of \main{}s in }\gamma \right)
    = u \dfrac{\partial}{\partial u} \left( \ln \left( \widetilde{G} (z,u) \right) \right).
  \end{equation*}
  But Theorem~\ref{theo:Gtilde} gives
  \begin{equation*}
    \ln \left( \widetilde{G} (z,u) \right) =
    - \frac {1-z} z \ln \left( 1 - \frac{z^2 u}{1-z} \right).
  \end{equation*}
  So, to satisfy the conditions of the expected values, \(z\) and \(u\) must solve the system
  \begin{equation}
    n =   \frac 1 z \ln  \left(  \frac {1 - z - z^2u} {1-z}\right) +{\frac {uz
 \left( 2 - z \right) }{1-z-{z}^{2}u}}, 
 \quad
 k =  {\frac {uz \left( 1-z \right) }{1-z-{z}^{2}u}}.
 \label{eq:system}
  \end{equation}
  By inversion, we get the equivalent system
  \begin{equation*}
    f(z) = 0, \quad u = \dfrac{k (1 - z)}{z (1 + z k - z)},
  \end{equation*}
  which is consistent with Algorithm~\ref{algo:tune}.
  
  Differentiating  \(f\) twice shows that
  \begin{equation*}
    {f}''(z)
    = \frac {{k}^2}
    {\left( 1-z+ k z \right)^2 \left( 1 - z \right)}
    +2 \left( n - k \right)
  \end{equation*}
  which is positive for \(0 \leq z < 1\), 
  and so \(f\) is convex on this interval.
  Furthermore, as \(z \rightarrow 0^+\), \(f\) is negative since \(f(z) \underset{z \rightarrow 0}{\sim} \left( k - n \right) z  \), and as  \(z \rightarrow 1^-\), \(f(z) \underset{z \rightarrow 1}{\to} k >0\).
  This guarantees a unique solution \( z \in (0,1)\) to \(f(z) = 0\).

  Finally, to verify that \(\left( z,u \right)\) lies in the convergence domain of \(\widetilde G (z,u)\), we note that
  \begin{equation*}
    1 - \frac{z^2 u}{1 - z}  = \frac{1 - z}{ 1 + (k-1 )z} >0
  \end{equation*}
 This confirms that \(\left( z,u \right)\) is indeed in the domain of convergence.
\end{proof}

Interestingly, the parameters \((z,u)\) exhibit a phase transition at \(\alpha = \frac 1 2\)
when \(\alpha\) is fixed as the ratio \(k / n\).

\begin{lemma}
  Fix $\alpha \in (0,1)$.
  Let \(n\) and \(k = \alpha \,n\) be the target sizes for the Boltzmann sampler.
  Denote $(z_n,u_n)$ the output of Algorithm~\ref{algo:tune}. Then, as \(n \to \infty\),
  \begin{equation*}
    z_n \sim \begin{cases}
      (1 - 2 \alpha)/(1 - \alpha) & \text{if }\alpha < \frac 1 2 \\
      \sqrt{\ln (n) / {n} } & \text{if }\alpha = \frac 1 2 \\
      C_\alpha /  n & \text{if }\alpha > \frac 1 2
    \end{cases}
    \quad
    \text{and}
    \quad
    u_n \sim \begin{cases}
      \alpha (1 - \alpha) / (1 - 2 \alpha)^2 & \text{if }\alpha < \frac 1 2 \\
      n/ \ln(n) & \text{if }\alpha = \frac 1 2 \\
      \frac{\alpha}{C_\alpha\left( 1 + \alpha C_\alpha  \right)}  n^2 & \text{if }\alpha > \frac 1 2
    \end{cases}
  \end{equation*}
  where \(C_\alpha\) is the unique solution $C$ of the equation \( C + \ln \left( 1 + \alpha\,C  \right) = 2 \alpha \,C\).
\end{lemma}

\paragraph*{Remark.} While the lemma addresses the exact case \(k = n / 2\),
the behavior of \(z_n\) and \(u_n\) around \(\alpha = 1/2\)
depends on the second-order asymptotic term,
which is not accounted for in this setting.

\begin{proof}
  By definition,
  \(z_n\) satisfies \(f(z_n) = 0\) where \(f\) is defined in the statement of Theorem~\ref{theo:tune}.

  In the proof of the latter, we showed that \(f < 0\) on \((0,z_n)\), and \(f > 0\) on \((z_n,1)\). 
  Therefore, to approximate \(z_n\), 
  it suffices to identify two close points \(z_n^-\) and \(z_n^+\) such that \(f (z_n^-)< 0\)
  and \(f (z_n^+)> 0\).
  This ensures that \(z_n\) is sandwiched between \(z_n^-\) and \(z_n^+\).

  For \(\alpha < 1/2\), the points \(z_n^\pm\) are defined as
  \begin{equation*}
    z_n^\pm := \frac {1 - 2\alpha}{1 - \alpha} \pm \frac 1 {1 - 2\alpha} \frac{\ln(n)^2} n.
  \end{equation*}
  Evaluating \(f\) at these points yields the asymptotic behavior
  \(f\left( z_n^\pm \right) \sim \pm \ln\left( n \right)^2 + O\left( \ln n \right)\), 
  which proves the asymptotic estimate for ~\(z_n\) for this regime. 
  (Note: the term \(\frac {1 - 2\alpha}{1 - \alpha}\) arises from an asymptotic equivalent of  \(f \left( z \right)\) for fixed \(z\):
    \(f \left( z \right) \sim z\left( (1 - \alpha) z - 1 + 2\alpha  \right) n \) .)

  For \(\alpha = \frac 1 2\), we set
  \begin{equation*}
    z_n^\pm := \sqrt{\frac{\ln(n)} n} \left( 1 + \left( \frac 1 2 \pm 1 \right) \frac{ \ln \ln (n)} {\ln(n)}\right)
  \end{equation*}
  and observe that  \(f\left( z_n^\pm \right) \sim \pm \ln \ln \left( n \right)\).

  For \(\alpha > 1/2\), we first justify the existence of a unique solution $C_\alpha > 0$ to the equation \( g(C) :=  2 \alpha \,C - C - \ln \left( 1 + \alpha\,C  \right) = 0\).
  We have \(g''(C) = \frac{\alpha^2}{\left( 1 + \alpha C \right)^2 } > 0\),
  so $g$ is convex. 
  Furthermore, 
  as \(C \to 0^+\), \(g(C) \sim \left( \alpha - 1 \right)C < 0\), 
  and as \(C \to + \infty\), \(g(C) \sim \left( 2 \alpha - 1 \right)C > 0 \). 
  By the Intermediate Value Theorem, \(g(C) = 0\) has a unique solution \(C_\alpha\).

  We then define
  \begin{equation*}
    z_n^\pm := \frac {C_\alpha} n
     + \frac{\pm \left( 1 + \alpha C_\alpha \right) - \alpha^2 {C_\alpha}^3 } 
     { \alpha (2 \alpha - 1) C_\alpha + \alpha -1  } \times
     \frac {1} {n^2}
  \end{equation*}
  and find  \(f\left( z_n^\pm \right) \sim \pm \frac 1 n\).

  The analysis of $z_n$ is now complete.

  For \(u_n\), we derive the asymptotic equivalents  from the formula \( u_n = \dfrac{\alpha \, n\, (1 - z_n)}{z_n (1 + \alpha z_n n  - z_n)}\) for the three cases.
\end{proof}

The following proposition quantifies 
the fluctuations in the size and the number of \main{}s
of a random \gitgraph{} generated by 
a Boltzmann sampler where these parameters are targeted.
Three different asymptotic regimes emerge.

\begin{proposition}
  Consider a fixed ratio \(\alpha = k/n \in (0,1)\), 
  where $n$ and $k$ are the target sizes for the Boltzmann sampler.
Let \(S_n\) denote the size of a random \gitgraph{} under the \labelled{} distribution with parameters \((z_n,u_n)\) as produced by Algorithm~\ref{algo:tune}, and  let \(K_n\)  be the number of \main{}s.

As \(n\) goes to \(+ \infty\), the variances of these random variables are equivalent to
\begin{equation*}
  \mathbb V \left( S_n \right) \sim
  \begin{cases}
    \frac{1 - 2 \alpha} \alpha n^2 & \text{ if }\alpha < \frac 1 2 \\
  \sqrt{\ln(n)} \, n^{\frac 3 2} & \text{ if }\alpha = \frac 1 2 \\
   \left( 4 \alpha^2 C_\alpha + 2 \alpha - 1 \right) n  & \text{ if }\alpha > \frac 1 2
  \end{cases}
\quad
\mathbb V \left( K_n \right) \sim
  \begin{cases}
    \alpha \left( 1 - 2 \alpha \right)  n^2 & \text{ if }\alpha < \frac 1 2 \\
  \frac 1 4 \sqrt{\ln(n)} \, n^{\frac 3 2} & \text{ if }\alpha = \frac 1 2 \\
   \alpha \left(  \alpha C_\alpha + 1\right) n  & \text{ if }\alpha > \frac 1 2
  \end{cases}
\end{equation*}
where \(C_\alpha\) is the unique solution $C$ of the equation \( C + \ln \left( 1 + \alpha\,C  \right) = 2 \alpha \,C\).
\label{prop:variances}
\end{proposition}

\begin{proof}
A convenient way to compute the variance for a Boltzmann generator is to use the formulas
  \[
\mathbb{V}(S_n) = \left. \frac{\partial^2 \ln \widetilde{G}(e^t, u_n)}{\partial t^2} \right|_{t = \ln z_n}
\quad \text{and} \quad
\mathbb{V}(K_n) = \left. \frac{\partial^2 \ln \widetilde{G}(z_n, e^t)}{\partial t^2} \right|_{t = \ln u_n},
\]
as described in~\cite[Proposition~7]{maciej-olivier-sergey}.
After computation, we obtain expressions for the variances in terms of \(z_n\) and \(u_n\)
\begin{equation*}
  \mathbb{V}(S_n) =
  -\frac 1 {z_n} \ln  \left( {\frac {1 - z_n - {z_n}^{2}u_n}{1-z_n}} \right) 
  -
  {\frac {{z_n}
^{2}u_n \left( 3\,{z_n}^{2}u_n-4\,z_nu_n+{z_n}^{2}-2\,z_n+1 \right) }{ \left( 1 - z_n - {z_n}^{2}u_n \right) ^{2} \left( 1-z_n \right) }},
\end{equation*}
\begin{equation*}
  \mathbb{V}(K_n) =
  {\frac {z_nu_n \left( 1 -z_n \right) ^{2}}{ \left( 1 - z_n - {z_n}^{2}u_n \right) ^{2}
}}.
\end{equation*}
The asymptotic equivalents then follow from the asymptotic behavior of \(z_n\) and \(u_n\) established in the previous lemma.
\end{proof}

The following theorem guarantees that, 
for a fixed ratio \(\alpha = k/n\),
the Boltzmann sampler will require only a constant number of trials (typically via rejection sampling)
to generate a \gitgraph{} of size \(n\) with \(k\) \main{}s,
with relative error proportional to \(n\).

\begin{theorem}
Fix \(\alpha \in (0,1)\) 
and let \((z_n,u_n)\) be the parameters output by Algorithm~\ref{algo:tune}
for target sizes \(n\) and \(k = \alpha \, n\).
Let \(S_n\) and \(K_n\) denote the size and the number of \main{}s,
respectively, of a random \gitgraph{} sampled under the \labelled{} distribution with parameters \((z_n,u_n)\).

For any \(\varepsilon > 0\) and \(\varepsilon' > 0\), there exists a constant \(\delta > 0\) and an integer \(n_0\) such that for every $n \geq n_0$,
both \(S_n\) and \(K_n\) concentrate around their targets with probability
\begin{equation*}
  \mathbb P\big( \left| S_n - n \right| < \varepsilon n \text{ and }
                    \left| K_n - \alpha n \right| < \varepsilon' n
    \big)
    > \delta.
\end{equation*}
\end{theorem}
\begin{proof}
  \textbf{Case  \(\boldsymbol{\alpha \geq \frac 1 2}\) }. 
  To prove the concentration result, we use a union bound to write 
  \begin{equation*}
    \mathbb P\big( \left| S_n - n \right| < \varepsilon n \text{ and }
                    \left| K_n - \alpha n \right| < \varepsilon' n
    \big)
    \geq 1 - 
    \mathbb P\big( \left| S_n - n \right| \geq \varepsilon n \big) 
    -      \mathbb P\big(
                    \left| K_n - \alpha n \right| \geq \varepsilon' n
    \big).
  \end{equation*}
  Chebyshev's inequality yields
  \begin{equation*}
    \mathbb P\big( \left| S_n - n \right|=   \left| S_n - \mathbb E \left( S_n \right) \right| \geq \varepsilon n \big) \leq  \frac{\mathbb V\left( S_n \right)}{\left( \varepsilon\, n \right)^2}
  \end{equation*}
  and similarly for \(K_n\).
  But by Proposition~\ref{prop:variances}, 
  the variances \(\mathbb V\left( S_n \right)\) and  \(\mathbb V\left( K_n \right)\)
  are sub-quadratic in~\(n\) for \(\alpha \geq \frac 1 2\).
  Consequently,
  the probability of interest tends to \(1\),
  hence is bounded below by a positive constant for large \(n\).

  \textbf{Case  \(\boldsymbol{\alpha < \frac 1 2}\) (proof sketch)}. 
  We begin by lower-bounding 
  \(\mathbb P \big( \left| S_n - n \right| < \varepsilon n \big) \).
  To do so, we derive an asymptotic lower bound for
   \begin{equation}
    \mathbb P \left( S_n = j \right) = \frac{ [z^j] \widetilde G \left( z , u_n \right) \, {z_n}^j } 
    { \widetilde G \left( z_n , u_n \right)}.
    \label{eq:proba_sn=j}
   \end{equation}
  The strategy involves analyzing the singular behavior of \(z \mapsto \widetilde G \left( z,u \right)\)
  and applying a transfer theorem (cf. \cite[Theorem VI.3]{FS2009}).
  However, the challenge lies in the fact that \(u\) is not fixed.

  The radius of convergence of \(z \mapsto \widetilde G \left( z,u \right)\) is
  \(\rho_u := \dfrac{ \sqrt{ 1 + 4 u} - 1 }{2u}\), and \(\rho_{u_n} \to \dfrac{1-2\alpha}{1-\alpha}\).
  For each \(z\) near \(\frac{1-2\alpha}{1-\alpha}\) 
  and for each \(u\) near \(\lim_{n \to +\infty} u_n = \frac{\alpha(1 - \alpha)}{(1 - 2\alpha)^2}\)
  with \(z \notin [\rho_u, +\infty)\),
  the generating function \(\widetilde G\) admits the expansion
  \begin{equation*}
    \widetilde G \left( z, u \right)
    = 
    \left(\frac{1 - \rho_u}{2 - \rho_u}\right)^{\frac{1 - \rho_u}{\rho_u}} \left( 1 - \frac z {\rho_u} \right)^{- \frac{1 - \rho_u}{\rho_u}}
    + H(z,u),
  \end{equation*}
  where \(H(z,u) = O \left(  \left( 1 - \frac z {\rho_u} \right)^{- \frac{1 - \rho_u}{\rho_u}+1} \ln \left(  1 - \frac z {\rho_u} \right) \right)\) uniformly in \(u\).

  First, this justifies that the denominator in Equation~\eqref{eq:proba_sn=j} satisfies
  \begin{equation*}
    \widetilde G (z_n,u_n) \leq 
    \text{Constant} \times 
    \left( 1 - \frac {z_n} { \rho_{u_n}} \right)^{- \frac{1 - \rho_u}{\rho_u}} \underset{n \to +\infty}\sim 
    \text{Constant}'  \times  n^{- \frac{1 - \rho_{u_n}}{\rho_{u_n}}}.
  \end{equation*}
  (The asymptotic equivalent of \(1 -  {z_n} / { \rho_{u_n}}\) can be deduced from the second formula of Equation~\eqref{eq:system}, which can be read \(z - \rho_u = -{z(1-z)}/(  {k( z+\frac{\sqrt{1 + 4 u}+1}{2u})} )\).)

  Second, we derive a uniform estimate for \([z^j] \widetilde G \left( z , u_n \right) \)
   by applying a transfer theorem to \(H(z,u)\). 
  Although Theorem~IV.3 of~\cite{FS2009} does not directly cover this case 
  (since \(u\) varies),
  we adapt its proof almost verbatim by using a contour which is independent of \(u\). 
  This yields
  \begin{equation*}
    [z^j] \widetilde G \left( z , u_n \right)  \geq 
    \text{Constant} \times  j^{\frac{1- 2 \rho_{u_n}}{\rho_{u_n}}}   {\rho_{u_n}}^{-j}.
  \end{equation*}
  Going back to Equation~\eqref{eq:proba_sn=j}, we obtain
  \begin{equation*}
    \mathbb P \left( S_n = j \right)  \geq 
    \frac{ \text{Constant} } n  \left(  \frac j n\right)^{\frac{1- 2 \rho_{u_n}}{\rho_{u_n}}}  {\left( \dfrac {z_n}{\rho_{u_n}} \right)}^{j}.
  \end{equation*}

  For \( \left| j - n \right| < \varepsilon n \),
  the term \( {\left( \dfrac {z_n}{\rho_{u_n}} \right)}^{j}\) is bounded below by a constant since
  \begin{equation*}
     {\left( \dfrac {z_n}{\rho_{u_n}} \right)}^{j} \sim \left( 1 - \frac{\alpha}{1 - 2\alpha} n^{-1} \right)^{j} \geq \left( 1 - \frac{\alpha}{1 - 2\alpha} n^{-1}\right)^{ (1 + \varepsilon) n }
  \end{equation*}
  and the right-hand side converges to a positive constant as \(n \to \infty\).

  Summing these lower bounds over \(j\) in the range \(\left| j - n \right| < \varepsilon n\), we get
  \begin{equation*}
    \mathbb P \big( \left| S_n - n \right| < \varepsilon n \big)
    \geq  
     \frac{ \text{Constant} } n \sum_{ (1- \varepsilon) n  < j < (1+ \varepsilon) n}
     \left(  \frac j n\right)^{\frac{1- 2 \rho_{u_n}}{\rho_{u_n}}}.
  \end{equation*}
  By comparing the sum to an integral, this expression is bounded below a positive constant.

  Now consider
  \begin{equation*}
    \mathbb{P}\left( |S_n - n| < \varepsilon n \text{ and } |K_n - \alpha n| < \varepsilon' n \right) = 
    \hspace{-20pt}
    \sum_{(1 - \varepsilon)n < j < (1 + \varepsilon)n} 
    \hspace{-20pt}
    \mathbb{P}(S_n = j) \mathbb{P}\left( |K_n - \alpha n| < \varepsilon' n \mid S_n = j \right).
  \end{equation*}
  The conditional variance of \(K_n\) given \(S_n = j\) is
  \begin{equation*}
    \left.
    \dfrac{ [z^j] \frac{ \partial^2 } {\partial u^2} \widetilde G(z,u) 
    + \frac{ \partial } {\partial u} \widetilde G(z,u) 
    - \left( \frac{ \partial } {\partial u} \widetilde G(z,u)   \right)^2    
    }
    { [z^j] \widetilde G(z,u) }
    \right|_{u = u_n}
    \sim
    \alpha \left( 1 - \alpha \right) \left(  1 - 2 \alpha  \right) n,
  \end{equation*}
  where the calculation relies on a transfer theorem.

  The conditional variance of \(K_n\) is sub-quadratic, so, by Chebyshev's inequality, the conditional probability
  \(\mathbb{P}\left( |K_n - \alpha n| < \varepsilon' n \mid S_n = j \right)\) is bounded below by a positive constant. 
  Combining this with the lower bound on \(\mathbb{P}\left( |S_n - n| < \varepsilon n \right)\),
  we conclude the proof.
\end{proof}

\begin{figure}[ht!]
\begin{center}
\includegraphics[width=\textwidth]{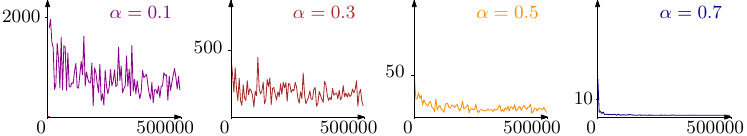}
\end{center}
\vspace{-7pt}
\caption{ $x$-axis: integer \(n\); \(y\)-axis: average number of trials (over $10$ independent runs) before generating a \gitgraph{} of size \( \in [0.99 n, 1.01n ]\) and with \(k\) \main{}s where \( k \in [0.99 \alpha n, 1.01 \alpha n ]\). }
\label{fig:nb_rejections}
\end{figure}

The previous theorem demonstrates that,
for every ratio \(\alpha \in \left( 0, \frac 1 2 \right)\),
only a constant number of rejections is required to obtain a random \gitgraph{}
of size \(n \pm \varepsilon n\) with \(\alpha \, n \pm \varepsilon' n\) \main{}s.
This theoretical result is confirmed by empirical experiments, 
as shown by Figure~\ref{fig:nb_rejections},
which plots the number of rejections to sample a \gitgraph{} with target sizes both met within a 1\% error margin. 

This proves that our Boltzmann generator operates in linear time when generating a \gitgraph{} with a linear-size window, demonstrating the generator's efficiency.
In contrast, exact sampling (without error tolerance) would require quadratic time.

\section{Perspectives}

In this paper, we have conducted a rather comprehensive enumerative study of a simple model of Git graphs.
While our results provide a solid foundation,
they are inherently preliminary,
as the model's simplicity restricts its applicability to real-world scenarios.
To bridge this gap, we have identified several key directions for future work.

First, the model could be extended to more complex workflows.
Currently, feature branches are born from 
and merged back into a single main branch.
Several extensions come to mind:
\begin{itemize}
  \item A \emph{multi-strata model}, 
  where branches can only be created or merged with branches from the immediately higher stratum.
  This would mirror large-scale projects where code modifications must pass through multiple approval layers before deployment.
  \item Allowing multiple merges of the same feature branch into the main branch, reflecting iterative updates of a same feature.
  This is the subject of a current work~\cite{en_cours}.
  \item Permitting \emph{pull merges} (also called \emph{update merges}), i.e. the operation that merges the main branch itself into a feature branch. 
  This typically occurs when developers pull the latest changes from the main branch 
  before pushing their own updates,
  ensuring their work is based on the most recent version of the codebase.
\end{itemize}

Second, the uniform distribution (as well as the \labelled{} variant)
provides a natural starting point for enumerative studies, 
but as highlighted in Section~\ref{s:limit_law},
it fails to reproduce realistic statistical properties of \gitgraph{}s 
(a limitation also noted in Blondel's internship~\cite{stage-lucas}).
To address this, alternative probabilistic models should be explored.
A first idea is to condition feature branches to have tightly coupled creation and merging dates,
mimicking the ephemeral nature of feature branches in agile workflows.
Another approach is to model Git operations as stochastic processes,
by assigning probabilities to branch creations, merges, and commits.

Third,
our work defines a subfamily of graphs modeling VCS repositories, 
raising natural algorithmic questions.
While the focus has been on enumeration, 
the practical relevance of these graphs calls for efficient algorithms. 
As already mentioned in the introduction,
key problems include reachability, lowest common ancestor, regression detection.
An interesting research direction would be then to design algorithms with provable efficiency guarantees, 
either in the worst case,
or under probabilistic models as discussed earlier.

Finally, a colleague of us pointed out that in many real-world scenarios,
free commits are rare.
This observation should motivate an enumerative study of \gitgraph{}s 
with a fixed number of free commits.
Given our existing formula (Corollary~\ref{cor:formula}),
this problem should be tractable.
Notably, the OEIS sequence A122852 (\url{https://oeis.org/A122852})
already counts \gitgraph{}s with exactly one free commit,
suggesting a bijective correspondence between these graphs and other combinatorial objects.

\section*{Statement regarding the (absence of) use of AI}

The authors declare that all the research presented here, as well as the writing of the article, was carried out without any use of large language models.

\bibliographystyle{alphaurl}
\bibliography{biblio}

\end{document}